\documentclass[12pt]{amsart}
\usepackage{amsfonts,amssymb,amscd,amstext,mathrsfs}
\usepackage[utf8]{inputenc}
\usepackage{hyperref}
\usepackage{verbatim}
\usepackage{color} 
\usepackage{times}
\usepackage{enumerate}
\usepackage[up,bf]{caption}
\usepackage{graphicx}
\usepackage[usenames,dvipsnames,svgnames,table]{xcolor}
\usepackage{tikz-cd}
\usepackage{mathrsfs}

\usepackage[top=1.8in, bottom=1.6in, left=1.6in, right=1.6in]{geometry}

\input xy
\xyoption{all}

\newtheorem{theorem}{Theorem}[section]
\newtheorem{proposition}[theorem]{Proposition}

\newtheorem{lemma}[theorem]{Lemma}
\newtheorem{corollary}[theorem]{Corollary}

\theoremstyle{definition}
\newtheorem{definition}[theorem]{Definition}
\newtheorem{remark}[theorem]{Remark}

\newtheorem{problem}[theorem]{Problem}
\newtheorem{example}[theorem]{Example}

\numberwithin{equation}{section}
\numberwithin{figure}{section}

\newcommand\Acal{\mathcal{A}}

\newcommand\Ocal{\mathcal{O}}

\newcommand\Cscr{\mathscr{C}}

\newcommand\Oscr{\mathscr{O}}

\newcommand\Vscr{\mathscr{V}}

\newcommand\B{\mathbb{B}}
\newcommand\C{\mathbb{C}}
\newcommand\D{\overline{\mathbb D}}
\newcommand\CP{\mathbb{CP}}

\renewcommand\D{\mathbb D}

\newcommand\N{\mathbb{N}}

\newcommand\R{\mathbb{R}}
\newcommand\RP{\mathbb{RP}}

\newcommand\Z{\mathbb{Z}}

\newcommand\igot{\mathfrak{i}}

\renewcommand\igot{\mathfrak{i}}

\renewcommand\imath{\igot}

\newcommand\hra{\hookrightarrow}

\newcommand\wt{\widetilde}
\newcommand\wh{\widehat}
\newcommand\di{\partial}
\newcommand\dibar{\overline\partial}

\newcommand\dist{\mathrm{dist}}

\newcommand\reg{\mathrm{reg}}

\newcommand\Conv{\mathrm{Conv}}

\newcommand\Aut{\mathrm{Aut}}

\newcommand\Id{\mathrm{Id}}

\def\Ell1{\mathrm{Ell_1}}
\def\CEll1{{\rm C-Ell$_1$}}
\def\CAP{\rm{CAP}}
\def\Bl{{\rm Bl}}

\numberwithin{equation}{section}
\begin{document}
\title%{Developments in Oka theory since 2017}
{Recent developments on Oka manifolds}
\author{Franc Forstneri{\v c}}

\address{Franc Forstneri\v c: Faculty of Mathematics and Physics, University of Ljubljana, and Institute of Mathematics, Physics, and Mechanics, Jadranska 19, 1000 Ljubljana, Slovenia}
\email{franc.forstneric@fmf.uni-lj.si}

\thanks{Research supported by the 
European Union (ERC Advanced grant HPDR, 101053085) and
by grants P1-0291, J1-3005, and N1-0237 from ARRS, Republic of Slovenia} 

\subjclass[2010]{Primary 32M17, 32Q56.  Secondary 14R10, 53D35}

\date{28 January 2023}

%%%%%

\keywords{Oka manifold; Oka map; Stein manifold; Elliptic manifold; Algebraically subelliptic manifold; Calabi--Yau manifold; Density property}

\begin{abstract}
In this paper we present the main developments in Oka theory since the publication 
of my book {\em Stein Manifolds and Holomorphic Mappings 
(The Homotopy Principle in Complex Analysis)},
Second Edition, Springer, 2017. We also give several new results, 
examples and constructions of Oka domains in Euclidean and projective spaces. 
Furthermore, we show that for $n>1$ the fibre $\C^n$ in a 
Stein family can degenerate to a non-Oka fibre, 
thereby answering a question of Takeo Ohsawa. Several open problems are discussed. 
\end{abstract}

\maketitle

\centerline{\em Dedicated to Jaap Korevaar  in honour of his 100th birthday}

\tableofcontents

%
%
%    INTRODUCTION
%
%
\section{Introduction: flexibility versus rigidity}\label{sec:intro} 
A major driving force of developments in complex analytic geometry 
is the dichotomy between flexibility and rigidity phenomena.

Prime examples of holomorphic flexibility include the approximation theorems of Runge and 
Oka--Weil and the extension--interpolation theorems of Weierstrass and Oka--Cartan.
These classical results show that complex Euclidean spaces $\C^n$ admit plenty of holomorphic maps 
from all Stein manifolds, that is, closed complex submanifolds of complex Euclidean spaces.
Oka theory deals with complex manifolds which admit plenty of holomorphic maps
from all Stein manifolds in a precise sense which is modelled on these classical theorems,
and with applications of these properties to problems in complex geometry and wider. 

Opposite to flexibility, the basic rigidity phenomena are 
described by Picard's theorem, saying that 
every holomorphic map $\C\to \C\setminus \{0,1\}$ is constant, 
and the Schwarz--Pick lemma to the effect that 
holomorphic self-maps of the disc $\D=\{z\in\C:|z|<1\}$ 
are distance-decreasing in the Poincar\'e metric. These are instances of 
the main rigidity property which a complex manifold $Y$ may have, 
Kobayashi hyperbolicity, asking that the size of the derivative of a holomorphic map $f:\D\to Y$ 
at $0\in \D$ is locally bounded above in terms of the value $f(0)\in Y$. In particular,
there are no nonconstant holomorphic maps $\C\to Y$ into a hyperbolic complex manifold.
There are several weaker notions of rigidity, such as the existence of nonconstant bounded plurisubharmonic functions and non-dominability by Euclidean spaces. 

For Riemann surfaces we have a clear dichotomy --- either the surface is Oka or it is 
Kobayashi hyperbolic \cite[Corollary 5.6.4]{Forstneric2017E}. 
The former ones are $\CP^1$, $\C$, $\C^*=\C\setminus \{0\}$ and tori, while 
the hyperbolic ones are quotients of the disc.
A majority of complex manifolds of dimension $>1$ are at least somewhat rigid. 
In particular, a compact complex manifold of general type is not dominable by Euclidean spaces 
(Kobayashi and Ochiai \cite{KobayashiOchiai1975}), and a generic projective hypersurface in 
$\CP^n$ of sufficiently large degree is hyperbolic (Brotbek \cite{Brotbek2017}).

The birth of Oka theory was the paper of Kiyoshi Oka \cite{Oka1939} (1939)  
in which he showed that the topological classification of complex line bundles on a domain
of holomorphy in $\C^n$ coincides with the holomorphic classification. 
In 1958, Hans Grauert \cite{Grauert1958MA} extended this  
to complex vector bundles of arbitrary rank and, more generally, to principal and associated 
fibre bundles on Stein spaces. This circle of results became known as the 
{\em Oka--Grauert principle}, with the following heuristic formulation given by 
Grauert and Remmert \cite[p.\ 45]{GrauertRemmert1979}:
 
{\em Analytic problems on Stein manifolds which can be cohomologically
formulated have only topological obstructions.} 

A major conceptual development was made by Mikhail Gromov \cite{Gromov1989} in 1989.
He emphasized the homotopy-theoretic aspect of Oka theory and the analogies to the h-principle 
in smooth geometry; see his monograph \cite{Gromov1986}. 
From Gromov's viewpoint, the main question is to find and characterize complex manifolds $Y$ 
having the property that every continuous map $X\to Y$ from a Stein manifold $X$ is homotopic to a holomorphic map, 
with natural additions modelled on the Oka--Weil approximation theorem 
and the Oka--Cartan extension theorem for holomorphic functions on Stein manifolds.
Furthermore, these properties should hold for families of maps $X\to Y$ depending continuously 
or smoothly on a parameter in a suitable space. Such {\em parameteric Oka properties} are 
very important in applications.  

The central notion of Oka theory is 
{\em Oka manifold}\footnote{MSC 2020 includes the new subfield {\em 32Q56 Oka principle and Oka manifolds}.}, 
a term introduced in 2009 in my paper \cite{Forstneric2009CR} when it became clear that most 
Oka-type properties considered in the literature are equivalent. 
The simplest characterization of this class of complex manifolds is the following
{\em convex approximation property} (CAP) which was introduced in 2006 in \cite{Forstneric2006AM}. 
For a list of known characterizations, see \cite[Sect.\ 5.15]{Forstneric2017E} 
and Subsection \ref{ss:Ell1} of this paper.

%
%  OKA MANIFOLD
%
\begin{definition}%[See \cite{Forstneric2009CR} and ] 
\label{def:Oka}
A complex manifold $Y$ is an Oka manifold if every holomorphic map $K\to Y$ 
from a neigbourhood of a compact convex set $K$ in a Euclidean space $\C^n$ 
(for any $n\in\N$) to $Y$ is a uniform limit on $K$ of entire maps $\C^n\to Y$.
\end{definition}

Here is a simplified form of the main result on Oka manifolds (see \cite[Theorem 5.4.4]{Forstneric2017E}). 

%
%   MAIN THEOREM OF OKA THEORY
%
\begin{theorem}\label{th:Oka}
Let $Y$ be an Oka manifold. 
Every continuous map $f:X\to Y$ from a Stein manifold (or a reduced Stein space) $X$ is homotopic to a
holomorphic map $f_1:X\to Y$. If in addition $f$ is holomorphic on a neighbourhood of 
a compact $\Oscr(X)$-convex set $K\subset X$ and on a closed complex subvariety 
$X'\subset X$, then a homotopy $\{f_t\}_{t\in[0,1]}$ from $f=f_0$ to $f_1$ can be chosen through maps 
$f_t:X\to Y$ having the same properties as $f$ which agree with $f$ on $X'$ and approximate $f$ uniformly on $K$
and uniformly in $t\in[0,1]$.

The analogous conclusion holds for a family of maps $f_p:X\to Y$ depending continuously
on a parameter $p$ in a compact Hausdorff space $P$, where the homotopy $f_{p,t}$ 
$(t\in[0,1])$ from $f_{p,0}=f_p$ to a family of holomorphic map $f_{p,1}:X\to Y$ may be kept fixed
for $p$ in a closed subset $Q\subset P$ provided that the map $f_p:X\to Y$ is holomorphic 
for all $p\in Q$.
\end{theorem}

It follows that the natural inclusion $\Oscr(X,Y)\hra \Cscr(X,Y)$ of the space
of holomorphic maps $X\to Y$ into the space of continuous maps is a weak homotopy equivalence
when $X$ is a Stein manifold and $Y$ is an Oka manifold. 
The classical Oka--Grauert theory fits this framework since complex
homogeneous manifolds are easily seen to be Oka, which implies the main results
of the Oka--Grauert theory. 

The proof of \cite[Theorem 5.4.4]{Forstneric2017E} and of related results in the cited 
monograph (see in particular \cite[Theorems 6.2.3 and 7.2.1]{Forstneric2017E}) 
also give the following result which is useful in applications; for simplicity we only 
state the basic case without parameters.

%
%   MAIN THEOREM 2
%
\begin{theorem}\label{th:Oka2}
Assume that $X$ is a reduced Stein space, $K$ is a compact $\Oscr(X)$-convex set in $X$, 
$X'$ is closed complex subvariety of $X$, $\Omega$ is an Oka domain in a complex manifold $Y$, 
and $f:X\to Y$ is a continuous map which is holomorphic on a neighbourhood of $K$ and on $X'$ such that 
$
	f(\overline{X\setminus K}) \subset \Omega.
$  
Then there is a homotopy $f_t:X\to Y$ $(t\in [0,1])$ connecting $f=f_0$ to a holomorphic map $f_1:X\to Y$ 
satisfying the conclusion of Theorem \ref{th:Oka} and also $f_t(\overline{X\setminus K}) \subset \Omega$
for all $t\in[0,1]$.
\end{theorem}

These results show that the CAP axiom in Definition \ref{def:Oka} localizes the Oka property 
of a complex manifold $Y$ to the Runge approximation problem for holomorphic maps from simplest possible domains 
to $Y$, namely, the convex domains in Euclidean spaces. This is often easier to verify, and it 
led to several new examples and constructions of Oka manifolds described in \cite{Forstneric2017E}.
It shows in particular that the class of Oka manifolds is invariant with respect to holomorphic 
fibre bundle projections with Oka fibres; see Theorem \ref{th:updown}. %\cite[Theorem 5.6.5]{Forstneric2017E}.

It is worth mentioning that every Oka manifold $Y$ is the image  of a strongly dominating 
holomorphic map $\C^n\to Y$ with $n=\dim Y$ (see \cite[Theorem 1.1]{Forstneric2017Indam}). Hence, when trying to decide which domains
$\Omega\subset\C^n$ $(n>1)$ are Oka, the first quintessential question is
to understand the shapes of images of nondegenerate entire maps $\C^n\to\C^n$.

Oka properties are also considered for holomorphic maps. The notion of an {\em Oka map} 
(see Definition \ref{def:Okamap}) was introduced by Finnur L\'arusson, 
who developed a model category for Oka theory \cite{Larusson2004,Larusson2005} 
in which Oka maps are fibrations and Stein inclusions are cofibrations. 
In particular, a complex manifold $Y$ is Oka if and only if the constant map 
$Y\to \mathrm{point}$ is an Oka map, and every holomorphic fibre bundle map with 
an Oka fibre is an Oka map. Modern Oka theory may thus be summarized as follows:  

\smallskip
{\em Analytic problems on Stein manifolds which can be formulated in terms of maps to
Oka manifolds, or liftings with respect to Oka maps, have only topological obstructions.}
\smallskip

Oka theory is an existence theory, providing solutions to a variety of complex analytic problems 
on Stein manifolds in the absence of topological obstructions.
On the other hand, Kobayashi hyperbolicity and related rigidity properties 
play the role of holomorphic obstruction theory. These two theories complement one another. 
Many challenging complex analytic problems lie in-between these two fields where there are no obvious obstructions to  
the existence of solutions, yet rigidity obstructions appear when trying to solve them. 

The state of the art of Oka theory up to 2017 is summarized in the monograph \cite{Forstneric2017E} 
and the  surveys \cite{Forstneric2013KAWA,ForstnericLarusson2011}; a brief historical account is included in 
Section \ref{sec:history}. In the remainder of the paper we discuss the main developments since 2017
and prove several new results. 

A series of major results by Yuta Kusakabe is described in Sections 
\ref{sec:elliptic}-\ref{sec:complements}. 
They provide a further unification of Oka theory through the axiom $\Ell1$
(see Definition \ref{def:elliptic} and Theorem \ref{th:K2018}), and they 
yield new constructions and a variety of new examples of Oka manifolds and Oka maps. 
Kusakabe showed that the Oka condition is Zariski local; see Theorem \ref{th:localization}.
He also proved that the complement $\C^n\setminus K$ of any compact polynomially 
convex subset $K\subset\C^n$ for $n>1$ is an Oka manifold; see Theorem \ref{th:complement}. 
The same holds in any Stein manifold having Varolin's density property introduced in 
\cite{Varolin2001}. Such manifolds share many global complex analytic properties with 
Euclidean spaces; see \cite{AndristForstnericRitterWold2016,Forstneric2019JAM,ForstnericKutzschebauch2022,ForstnericWold2020PAMS,Kutzschebauch2020} among others.

In Sections \ref{sec:complements} and \ref{sec:CPn} we describe several new examples 
of Oka manifolds. We show in particular that for a compact polynomially convex set $K$ in 
$\C^n$, $n>1$, and considering $\C^n$ as an affine chart in the projective space $\CP^n$, 
the complement $\CP^n\setminus K$ is Oka as well; see Corollary \ref{cor:complementCPn}. 
Under a mild additional assumption, the same holds for complements (in $\C^n$ and $\CP^n$) 
of compact sets of the form $C\cup K$, where $K$ is a compact polynomially convex set 
in $\C^n$ and $C$ is contained in a compact connected set 
of finite length; see Theorems \ref{th:CK} and \ref{th:CK-CPn}. 
In particular, if $C$ is a rectifiable Jordan curve in $\C^n$ for $n>1$ 
then $\C^n\setminus C$ and $\CP^n\setminus C$ are Oka manifolds.
Furthermore, the complement of any closed strictly convex set in $\C^n$ $(n>1)$ is Oka; 
see Theorem \ref{th:convexnoline} due to Wold and the author \cite{ForstnericWold2022Oka}. 
 
In Section \ref{sec:aOka} we describe recent progress in the study of Oka theory for  
algebraic maps from affine algebraic varieties into algebraic manifolds, due to 
L\'arusson and Truong \cite{LarussonTruong2019}, 
Kusakabe \cite{Kusakabe2020IJM,Kusakabe2021JGA,Kusakabe2022surjective,Kusakabe2022pi1}, 
Bochnak and Kucharz \cite{BochnakKucharz2020}, and 
Kaliman and Zaidenberg \cite{KalimanZaidenberg2023}.  

In another direction, Luca Studer extended Oka theory to certain 
Oka pairs of sheaves \cite{Studer2021APDE}, thereby generalizing the work of 
Forster and Ramspott \cite{ForsterRamspott1966IM1} from 1966. He also developed an abstract 
homotopy theorem based on Oka theory \cite{Studer2020MA}. See Section \ref{sec:Studer}. 

In Section \ref{sec:Carleman} we describe new approximation theorems of 
Carleman and Arakelian type for maps to Oka manifolds, due to
Brett Chenoweth \cite{Chenoweth2019PAMS} and the author \cite{Forstneric2019MMJ}. 

In Section \ref{sec:DG} we present a generalization % from \cite{Forstneric2022JMAA} 
of the Docquier--Grauert tubular neighbourhood theorem for Stein manifolds 
(see Theorem \ref{th:DG}),  and an application to the construction of large Euclidean domains 
in complex manifolds; see Theorem \ref{th:Euclidean}. 

In Section \ref{sec:degeneration} we give an example of a Stein submersion $X\to \C$ which 
is a trivial holomorphic fibre bundle with fibre $\C^2$ over $\C^*=\C\setminus\{0\}$ but whose fibre
over $0\in \C$ is the product of the disc with $\C$, so it fails to be Oka 
(see Theorem \ref{th:limitfibre1}). 
This answers a question asked by Takeo Ohsawa in \cite[Q3]{Ohsawa2020MRL} (2020).

In Section \ref{sec:future} we discuss the possible connections between Oka manifolds 
and special manifolds in the sense of Campana, and we pose several problems regarding the 
relationship between Oka properties and positivity of complete K\"ahler metrics. 

The paper contains an Appendix with a summary of results on holomorphic convexity
of compact sets in affine domains in projective spaces, based on Oka's criterion
for polynomial convexity. These results are used in Sections \ref{sec:complements} and \ref{sec:CPn}. 

Among the recent results not discussed in this survey, we mention  
the development of equivariant version of modern Oka theory by Frank 
Kutzschebauch, Finnur L\'arusson, and Gerald Schwarz 
\cite{KutzschebauchLarussonSchwarz2015JRAM,KutzschebauchLarussonSchwarz2017TAMS,KutzschebauchLarussonSchwarz2018MA,KutzschebauchLarussonSchwarz2021JGEA,KutzschebauchLarussonSchwarz2022}.
They introduced the notion of a $G$-Oka manifold where $G$ is a reductive complex Lie group.
Kusakabe (see \cite[Appendix]{Kusakabe2020IJM})
characterized $G$-Oka manifolds by a $G$-equivariant version of his 
characterization of Oka manifolds by condition $\Ell1$; see Definition \ref{def:elliptic} (b). 
A recent survey of this topic is available in \cite{KutzschebauchLarussonSchwarz2022}.

There were important developments in other areas of analysis and 
geometry closely intertwined with
Oka theory. One of them is the Anders\'en--Lempert--Varolin theory, which concerns 
Stein manifolds with large holomorphic automorphism groups.
This subject has a major impact on Oka theory. The main link is provided by 
the fact that every Stein manifold on which complete holomorphic vector fields densely generate 
the Lie algebra of all holomorphic vector fields (Varolin's density property, see \cite{Varolin2001,Varolin2000}) is an Oka manifold, and it is also Oka at infinity
(see Definition \ref{def:Okainfinity}).
Furthermore, the Oka principle holds for {\em proper} holomorphic maps, immersions
and embeddings of Stein manifolds of suitable dimension into such a manifold.
This subject is treated in \cite[Chapter 4]{Forstneric2017E} and in the recent surveys 
\cite{ForstnericKutzschebauch2022,Kutzschebauch2020}. 

Results and methods of Oka theory have lately found new applications. 
Foremost among them pertain to the study of minimal surfaces in Euclidean spaces \cite{AlarconForstneric2019JAMS,AlarconForstnericLopezMAMS,AlarconForstnericLopez2021} 
and holomorphic Legendrian curves in complex contact manifolds 
\cite{AlarconForstnericLopez2017CM,AlarconForstneric2019IMRN,AlarconForstnericLarusson2022GT,ForstnericLarussonIUMJ2022}.
Results on the latter topic were applied to the construction of superminimal surfaces 
in self-dual or anti-self-dual Einstein four-manifolds \cite{Forstneric2021JGA,Forstneric2022AFSTM} 
by exploring the Bryant correspondence, based on Penrose twistor spaces,  
between superminimal surfaces in this class of Riemannian four-manifolds and 
holomorphic Legendrian curves in three-dimensional complex contact manifolds. 
Another recent application are Vaserstein-type results on factorization of holomorphic
maps in the complex symplectic group $Sp_{2n}(\C)$; 
see Ivarsson et al.\ \cite{IvarssonKutzschebauchLow2020} and Schott \cite{Schott2022}. 
Earlier work on the related problem for maps to $SL_n(\C)$ was done by
Ivarsson and Kutzschebauch \cite{IvarssonKutzschebauch2012AM} in 2012. 

These applications might indicate the beginning of 
the development of the Oka principle for {\em holomorphic partial differential relations}. 
I have in mind a range of complex analytic problems where not only values of maps, 
but also their jets must satisfy certain relations. These include holomorphic differential equations
and a variety of open differential conditions. The aforementioned applications to minimal surfaces and 
directed holomorphic curves, such as null curves and Legendrian curves, 
are of this kind. The study of regular holomorphic maps such as immersions and submersions
also fits in this framework. In this direction, the Runge approximation problem for locally biholomorphic 
self-maps of Euclidean spaces $\C^n$ for $n>1$ remains a mystery, and understanding this 
subject would have major implications.

%
%   BRIEF HISTORY 
%
\section{A brief history of Oka theory up to 2017}\label{sec:history}

Oka theory evolved from the works of Kiyoshi Oka \cite{Oka1939} (1939), 
Hans Grauert \cite{Grauert1958MA} (1958), and Mikhail Gromov \cite{Gromov1989} (1989).
The principal motivation behind the works of Grauert and other contributors to the classical
theory, most notably Otto Forster and Karl Josef Ramspott and later also 
Gennadi Khenkin and J\"urgen Leiterer, was to understand the classification of 
principal bundles and their associated bundles (in particular, vector bundles) on Stein spaces. 
The main result of the classical Oka--Grauert theory asserts that the holomorphic classification 
of such bundles agrees with their topological classification. 
Problems of this type typically reduce to questions about maps to classifying spaces, and hence it 
became of interest to understand the class of complex manifolds having the property that
every continuous map from a Stein manifold or a Stein space to the given manifold 
is homotopic to a holomorphic map, with natural additions concerning approximation
on compact holomorphically convex sets and interpolation on closed complex subvarieties. 
Although this observation was known from the beginning, as can be seen 
in particular from Cartan's exposition \cite{Cartan1958} of Grauert's work \cite{Grauert1958MA},
the cohomological viewpoint prevailed in this early period. 

It took almost three decades till Gromov \cite{Gromov1989} proposed a 
more general viewpoint and developed new approaches, thereby releasing the theory 
from the constraints of complex Lie groups and homogeneous manifolds.
The emphasis shifted from the cohomological to the homotopy-theoretic viewpoint. 
Gromov introduced geometric sufficient conditions for validity of the Oka principle 
for maps from Stein manifolds in terms of the existence of dominating holomorphic sprays 
on the target manifold. 
In particular, he introduced the notion of an elliptic complex manifold and of an elliptic submersion, 
and he outlined the proof of the Oka principle under these assumptions. 
The first major application of Gromov's new methods was a solution of the 
optimal embedding problem for Stein manifolds in Euclidean spaces by Eliashberg and 
Gromov in 1992 \cite{EliashbergGromov1992}, with a subsequent improvement for 
odd dimensional manifolds by Sch\"urmann \cite{Schurmann1997}; 
see the exposition in \cite[Secs.\ 9.3--9.4]{Forstneric2017E}. 
Numerous other applications are described in \cite[Chaps. 8--10]{Forstneric2017E}.

After Gromov's seminal paper \cite{Gromov1989}, the first steps 
to understand and develop his ideas were made in my joint papers with Jasna Prezelj 
\cite{ForstnericPrezelj2000,ForstnericPrezelj2001,ForstnericPrezelj2002} during 2000--2002. 
These papers provide detailed proofs and some extensions of the main results from 
\cite{Gromov1989}. A weaker sufficient condition for the Oka principle, subellipticity, 
was already discussed by Gromov and formally introduced in \cite{Forstneric2002MZ}. 
The study of the Oka principle for sections of branched holomorphic maps was initiated 
in \cite{Forstneric2003FM}. In the same period, Finnur L\'arusson developed an abstract 
homotopy-theoretic approach which culminated in his construction of a model category 
for Oka theory; see \cite{Larusson2003,Larusson2004,Larusson2005}, 
\cite[Appendix]{Forstneric2013KAWA}, and \cite[Sect.\ 7.5]{Forstneric2017E}.

Subsequent developments focused on finding necessary and sufficient conditions 
on a complex manifold $Y$ to satisfy the Oka principle for maps $X\to Y$ 
from Stein manifolds and Stein spaces. Gromov asked in \cite{Gromov1989}
whether a Runge approximation property for maps from simple domains 
in Euclidean spaces might suffice. This question was answered affirmatively 
in my paper \cite{Forstneric2006AM} in 2006. In this paper, the convex approximation property 
(CAP) of a complex manifold $Y$ was introduced (see Definition \ref{def:Oka}), 
and it was shown that CAP implies the basic Oka property with approximation for maps 
from Stein manifolds to $Y$. Interpolation on closed complex subvarieties was added 
in \cite{Forstneric2005AIF}. It took a few more years to understand that CAP also implies 
the parameteric Oka properties \cite{Forstneric2009CR}, and that various Oka-type conditions 
considered in the literature, including their parameteric versions, are pairwise equivalent. 
This motivated the introduction of the class of Oka manifolds as complex manifolds satisfying 
all these equivalent conditions; see \cite{Forstneric2009CR}. 
The Oka property for sections of stratified subelliptic submersions onto reduced Stein 
spaces was established in \cite{Forstneric2010PAMQ} (2010), 
extending the result of Gromov \cite{Gromov1989} for sections of elliptic submersions
onto Stein manifolds. The stratified case was first considered with lesser precision in 
\cite[Sect.\ 7]{ForstnericPrezelj2001}. Most known applications of the Oka principle
are in the context of stratified (sub-) elliptic submersions. 

The related concept of an {\em Oka map} was introduced by L\'arusson \cite{Larusson2004}
in 2004; see also \cite{Forstneric2010CR}. 
A holomorphic map $Z\to Y$ between complex manifolds is said to be an Oka map if it is a 
topological fibration and it enjoys the Oka properties for lifts of holomorphic maps $f:X\to Y$
from reduced Stein spaces $X$ to holomorphic maps $F:X\to Z$; see Definition \ref{def:Okamap}. 
(By a topological fibration, we mean a Serre fibration or a Hurewicz fibration; these conditions
are equivalent for maps between manifolds, and they refer to the homotopy lifting property.)
For a precise definition, see L\'arusson \cite{Larusson2004} and 
\cite[Definition 7.4.7]{Forstneric2017E}. In particular, a complex manifold $Y$ is an 
Oka manifold if and only if the map $Y\to \mathrm{point}$ is an Oka map. 

These developments are summarized in the two editions (2011, 2017) of my monograph 
\cite{Forstneric2011E,Forstneric2017E} and in the surveys 
\cite{ForstnericLarusson2011,Forstneric2013KAWA}. 
The remainder of the article is mainly devoted to the exposition
of results obtained after 2017. Among them, we emphasize new characterizations of 
Oka manifolds and Oka maps due to Yuta Kusakabe \cite{Kusakabe2021IUMJ,Kusakabe2021MZ}; 
see Section \ref{sec:elliptic}. The ellipticity condition involved in his characterizations was 
introduced by Gromov \cite{Gromov1989} in 1989. 

Oka manifolds are the very opposite of Kobayashi hyperbolic manifolds, 
the latter not admitting any nonconstant holomorphic maps from $\C$. 
A majority of complex manifolds have at least some holomorphic rigidity.  
This holds in particular for compact complex manifolds of general type --- these are not 
dominable by Euclidean spaces according to 
Kobayashi and Ochiai \cite{KobayashiOchiai1975}, and hence are not Oka.
For a long time it seemed that Oka manifolds are few and very special.
However, it recently became clear that they are much more plentiful than previously thought, 
at least among noncompact complex manifolds; 
see Sections \ref{sec:complements} and \ref{sec:CPn}. 
These results opened new vistas of possibilities that remain to be fully explored.

%%%%%%%%%%%%%%%%%%%%%%%%%%%
%
%	ELLIPTIC characterizATION OF OKA MANIFOLDS
%
%%%%%%%%%%%%%%%%%%%%%%%%%%%

\section{Elliptic characterization of Oka manifolds and Oka maps} \label{sec:elliptic}
In this section we present a new conceptual unification of Oka theory, due to Yuta Kusakabe 
\cite{Kusakabe2021IUMJ} (2021). He proved that a restricted version 
Gromov's ellipticity condition $\Ell1$ for holomorphic maps from compact convex sets 
in Euclidean spaces to a given complex manifold $Y$ implies the convex approximation 
property (CAP) of $Y$; see Definition \ref{def:elliptic} (b) and Theorem \ref{th:K2018}.
It has been known since 2009 \cite{Forstneric2009CR} that CAP is equivalent 
to the validity of all Oka properties of $Y$ (see also \cite[Theorem 5.4.4]{Forstneric2017E}), 
and that it also implies condition $\Ell1$. This provides an affirmative answer to a question 
of Gromov \cite[p.\ 72]{Gromov1986}. 
Another result of Kusakabe \cite{Kusakabe2021MZ} gives the analogous 
characterization of Oka maps by convex ellipticity; see Theorem \ref{th:ellipticmap}.
An important consequence is a localization theorem for Oka manifolds 
(see Theorem \ref{th:localization}), which has already led to many new examples. 
A fascinating application of these new techniques is Kusakabe's result that the complement 
of every compact polynomially convex set in $\C^n$ for $n>1$ is Oka 
(see Section \ref{sec:complements}). 
Furthermore, it was recently shown by Wold and the author \cite{ForstnericWold2022Oka} 
that for most closed convex set $E\subset \C^n$ $(n>1)$ which are smaller than a halfspace, 
the complement $\C^n\setminus E$ is an Oka domain (see e.g.\ Theorem \ref{th:convexnoline}). 
In particular, there are concave Oka domains in $\C^n$ for $n>1$ which are only slightly bigger 
than a halfspace. Results of this kind seemed totally unimaginable a few years ago.

%
%	 ELLIPTICITY CONDITIONS
%
\subsection{Ellipticity conditions}
In \cite{Gromov1986,Gromov1989} Gromov introduced several ellipticity conditions for 
complex manifolds and holomorphic maps, which provide geometric sufficient conditions for 
Oka properties. These conditions are based on the notion of a dominating spray, 
a prime example of which is the exponential map on a complex Lie group. 

Let $X$ and $Y$ be complex manifolds. 
A {\em (holomorphic) spray of maps} $X\to Y$ is a holomorphic map $F:X\times \C^N\to Y$
for some $N\in\N$. The map $f=F(\cdotp,0):X\to Y$ is the {\em core} of $F$, and 
$F$ is called a {\em spray over $f$}. The spray $F$ is said to be {\em dominating} if 
\[
	\frac{\di}{\di w}\Big|_{w=0}F(x,w):\C^N\to T_{f(x)}Y
	\ \ \text{is surjective for every}\ \ x\in X.
\]
More generally, $F$ is dominating on a subset $U\subset X$ if the above condition holds for 
every point $x\in U$. A more general type of a spray is a holomorphic map $F:E\to Y$ from the 
total space $E$ of a holomorphic vector bundle $\pi:E\to X$; its core is the restriction of $F$ 
to the zero section of $E$ (which we identify with $X$), and the domination condition is defined in 
the same way by considering the derivative in the fibre direction.
In particular, a dominating spray on a complex manifold $Y$ over the identity map $\Id_Y$ is 
a holomorphic map $F:E\to Y$ from the total space of a holomorphic vector bundle $E\to Y$ such that 
\begin{equation}\label{eq:sprayonY}
	F(0_y)=y\ \ \text{and}\ \ dF_{0_y}(E_y) = T_y Y\ \ \text{for every $y\in Y$.}
\end{equation}
If $Y$ is a homogeneous manifold of a complex Lie group $G$
and $\mathrm e:\mathfrak g\to G$ is the exponential map on $G$, then
the map $Y\times \mathfrak g\to Y$ given by $(y,v)\mapsto \mathrm e^v y$ is a 
dominating spray on $Y$.

\newpage 

%
%   ELLIPTIC MANIFOLDS
%
\begin{definition}\label{def:elliptic}
Let $Y$ be a complex manifold.
\begin{enumerate}[\rm (a)]
\item (Gromov \cite[0.5, p.\ 8.5.5]{Gromov1989}; see also \cite[Definition 5.6.13]{Forstneric2017E}.)
The manifold $Y$ is {\em elliptic} if it admits a dominating holomorphic spray $F:E\to Y$ over $\Id_Y$,
and is {\em special elliptic} if such a spray exists on a trivial bundle $E=Y\times \C^N$. 
\item (Gromov \cite[p.\ 72]{Gromov1986}.) 
The manifold $Y$ enjoys condition $\Ell1$ if every holomorphic map $X\to Y$ from a Stein manifold 
is the core of a dominating holomorphic spray $X\times \C^N\to Y$.
\item The manifold  $Y$ enjoys condition \CEll1 if for every compact convex set $K\subset \C^n$
$(n\in\N)$, open set $U\subset \C^n$ containing $K$, and holomorphic map $f:U\to Y$ 
there are an open set $V$ with $K\subset V\subset U$ and a dominating holomorphic spray 
$F:V\times \C^N\to Y$ over $f|_V$. 
\end{enumerate}
\end{definition}

Every elliptic Stein manifold $Y$ is also special elliptic. Indeed, by an extension of 
Cartan's Theorem A (see Forster \cite[Corollary 4.4]{Forster1967MZ} or Kripke \cite{Kripke1969}), 
every holomorphic vector bundle $\pi:E\to Y$ over a Stein manifold admits finitely
many (say $N$) holomorphic sections which span the fibre $E_y=\pi^{-1}(y)$ over each point $y\in Y$. 
This gives a surjective holomorphic vector bundle map $\phi:Y\times \C^N\to E$, and  
precomposing a dominating spray $F:E\to Y$ by $\phi$ yields a dominating spray 
$Y\times\C^N\to Y$. This fails on non-Stein manifolds: every compact 
special elliptic manifold is complex homogeneous (see \cite[Proposition 6.2]{Forstneric2019MMJ}).

Condition $\Ell1$ obviously implies \CEll1, the latter being a restricted version 
of $\Ell1$ applied to compact convex sets in Euclidean spaces, and we ask that a dominating spray 
exists over a smaller neighbourhood of the set (this comes handy in applications).  

One of Gromov's main results in \cite{Gromov1989} is that every elliptic manifold is Oka
(see also \cite[Corollary 5.6.14]{Forstneric2017E}). In fact, ellipticity easily implies CAP 
(see Definition \ref{def:Oka}); this is a special case of \cite[Theorem 6.6.1]{Forstneric2017E}
which gives a Runge approximation theorem for homotopies of holomorphic maps.
Conversely, every Stein Oka manifold is easily seen to be elliptic 
\cite[Proposition 5.6.15]{Forstneric2017E}. 
Kusakabe \cite{Kusakabe2020PAMS} gave the first known examples of Oka manifolds 
which fail to be elliptic or even just (weakly) subelliptic, 
thereby negatively answering a question on Gromov. 
(See the discussion on \cite[p.\ 325]{Forstneric2017E}.)  % Kusakabe's 
The results in Sections \ref{sec:complements} and \ref{sec:CPn} provide 
a plethora of such examples. However, the following problem seems to remain open.

\begin{problem}\label{prob:Okaelliptic}
Is there a compact Oka manifold which fails to be elliptic or subelliptic?
\end{problem}

%
%   SUBSECTION: Ell_1
%
\subsection{Characterization of Oka manifolds by condition $\Ell1$} \label{ss:Ell1}
It is easily seen that every Oka manifold satisfies condition $\Ell1$
(see \cite[Corollary 8.8.7]{Forstneric2017E}).  
It came as a genuine surprise that the converse holds as well. 
The following result is due to Kusakabe \cite[Theorem 1.3]{Kusakabe2021IUMJ}.

%
%   KUSAKABE 2018
%
\begin{theorem}\label{th:K2018}
A complex manifold which satisfies condition \CEll1 is an Oka manifold.
Hence, the following conditions on a complex manifold are equivalent:
\[
	\text{\rm{Oka}}\ \Longleftrightarrow \ \Ell1 \ \Longleftrightarrow \ \text{\CEll1}.
\]
\end{theorem}

It follows that conditions $\Ell1$, Ell$_2$ and Ell$_\infty$ introduced by 
Gromov in \cite{Gromov1989} are pairwise equivalent, 
and they characterize the class of Oka manifolds. See also 
\cite[Conjecture 4.6 and Corollary 4.7]{Kusakabe2021IUMJ} for a more precise
description of Gromov's conjectures. 

Theorem \ref{th:K2018} enables the construction of many new examples of Oka manifolds;
see in particular Theorem \ref{th:localization} and the examples in Sections \ref{sec:complements}
and \ref{sec:CPn}. The main point is that it is often easier to construct sprays whose domain is 
a Stein manifold, or even just a convex domain in a Euclidean space, rather than a 
general complex manifold. 

%
%   O(K,Y)
%
Due to its importance, we include a proof of Theorem \ref{th:K2018}. The main point is to show
that  \CEll1 implies CAP; the rest follows from previously known results
(see Theorem \ref{th:Oka}).

We begin with preparations. Given a compact set $K$ in a complex manifold $X$ 
and a complex manifold $Y$, we denote by $\Oscr(K,Y)$ the space of germs on $K$ of 
holomorphic maps from open neighbourhoods $U\subset X$
of $K$ to $Y$. Thus, $\Oscr(K,Y)$ is the colimit (also called the direct limit) of the system $\Oscr(U,Y)$ over 
open sets $U\subset X$ containing $K$,  with the natural restrictions maps 
$r_{U,V}:\Oscr(V,Y)\to \Oscr(U,Y)$ given for any pair $U\subset V$ by $r_{U,V}(f)=f|_U$. 
The space $\Oscr(K,Y)$ carries the colimit topology defined as follows. 
Fix a distance function $\dist$ on $Y$ inducing the natural manifold topology. 
A basic open neighbourhood of an element of $\Oscr(K,Y)$, 
represented by a map $f\in \Oscr(U,Y)$, is a set of the form
\begin{equation}\label{eq:basicnbd}
	\Vscr(f,U',K',\epsilon) = \bigl\{g\in \Oscr(U',Y): \sup_{z\in K'}\dist(f(z),g(z)) <\epsilon\bigr\}
\end{equation}
where $K'$ is a compact set containing $K$ in its interior, $U'$ is an open set with 
$K'\subset U'\Subset U$, and $\epsilon>0$. Equivalently, let $(U_k)_{k\geq 1}$ be a decreasing 
basis of open neighbourhoods of $K$ such that $U_{k+1}$ is relatively compact in $U_k$ 
for all $k\geq 1$. The colimit topology on $\Oscr(K,Y)$ is the finest topology that makes all maps 
$\Oscr(U_k,Y) \to \Oscr(K, Y)$ continuous. By saying that a map $K\to Y$ is holomorphic,
we mean that it belongs to $\Oscr(K, Y)$.

A {\em (convex) polyhedron} in $\R^N$ 
is a compact set  which is the intersection of finitely many closed affine half-spaces.
Recall the following definition (cf.\ \cite[Definition 5.15.3]{Forstneric2017E}).

%
%  Definition: Special polyhedral pair
%
\begin{definition}\label{def:SPP}
A pair $K\subset L$ of compact convex sets in $\R^N$ is a {\em special polyhedral pair}
if $L$ is a polyhedron and $K=\{z\in L : \lambda(z)\le 0\}$ for some affine linear function 
$\lambda\colon\R^N\to\R$.
\end{definition}

The following observation is due to Kusakabe \cite{Kusakabe2017} 
(see \cite[Lemma 5.15.4]{Forstneric2017E}). 

%
%   CAP FOR POLYHEDRAL PAIRS
%  
\begin{lemma}\label{lem:SCAP}
Suppose that $Y$ is a complex manifold such that for each
special polyhedral pair $K\subset L$ in $\C^n$ $(n\in \N)$, every holomorphic map $K \to Y$ 
can be approximated uniformly on $K$ by holomorphic maps $L\to Y$. Then $Y$ enjoys {\CAP}  
and hence is an Oka manifold.
\end{lemma}

%
%
%     PROOF OF KUSAKABE'S THEOREM ON ELL1
%
%
\begin{proof}[Proof of Theorem \ref{th:K2018}]
Let $K\subset L$ be a special polyhedral pair in $\C^n$. 
Since $K$ is convex and $Y$ is connected, the space $\Oscr(K,Y)$ is connected. 
Indeed, every map $f\in\Oscr(K,Y)$ is homotopic to the constant map $K\to f(p)$ for any $p\in K$. 
Denote by $\Acal$ the set of all $f\in \Oscr(K,Y)$ which can be approximated uniformly 
on $K$ by maps in $\Oscr(L,Y)$. Clearly $\Acal$ is nonempty (since it contains constant maps)
and closed in $\Oscr(K,Y)$. In view of Lemma \ref{lem:SCAP} and connectedness 
of $\Oscr(K,Y)$ it remains to show that $\Acal$ is also open in $\Oscr(K,Y)$,
so $\Acal=\Oscr(K,Y)$.

Fix $f\in \Acal$ and represent it by a map $f\in \Oscr(U,Y)$ from an open set $U\subset\C^n$ 
containing $K$. Condition \CEll1 gives a convex open set $V$, with $K\subset V\subset U$,  
and a dominating holomorphic spray $F:V\times \C^N\to Y$ with $F(\cdotp,0)=f|_V$.
By factoring out the kernel of 
\[
	\di F(z,w)/\di w|_{w=0}:\C^N\to T_{f(z)}Y,\quad z\in V
\]
(which is a trivial holomorphic subbundle of $V\times \C^N$ 
with trivial quotient) we may assume that $N=\dim Y$ and the above map is an isomorphism
for every $z\in V$. Hence, up to shrinking $V$ around $K$ if necessary, there is an open ball $0\in W\subset\C^N$
such that the map $\wt F=(\Id,F):V\times \C^N\to V\times Y$ given by 
\begin{equation}\label{eq:wtF}
	\wt F(z,w) = (z,F(z,w)),\quad z\in V,\ w\in \C^N
\end{equation}
maps $V\times W$ biholomorphically onto its image in $V\times Y$.
Since $f\in \Acal$, there are a neighbourhood $\Omega\subset \C^n$ of $L$ and a map 
$g\in \Oscr(\Omega,Y)$ whose graph $\{(z,g(z)):z\in K\}$ over $K$ belongs to $\wt F(V\times W)$.
Up to shrinking $\Omega$ around $L$, \cite[Lemma 5.10.4]{Forstneric2017E} 
provides a local dominating holomorphic spray $G:\Omega \times W\to Y$ over 
$G(\cdotp,0)=g$. Replacing $G(z,w)$ by $G(z,tw)$ 
for a small $t>0$ we may assume that the map $\wt G(z,w) = (z,G(z,w))$ % $(z\in \Omega,\ w\in W)$ 
satisfies $\wt G(K\times W)\Subset \wt F(V\times W)$. Hence, there is an open convex set $U_1\subset\C^n$ 
with $K\subset U_1\Subset V\cap\Omega$ such that $\wt G(U_1\times W)\Subset \wt F(V\times W)$.
Since the map $\wt F$ \eqref{eq:wtF} is biholomorphic on $V\times W$, there is a unique
holomorphic map $H:U_1\times W \to W$ such that 
\begin{equation}\label{eq:H}
	F(z,H(z,w)) = G(z,w) \quad \text{for all $(z,w)\in U_1\times W$}. 
\end{equation}
%Note $h:=H(\cdotp,0):U_1\to Y$ satisfies $F(z,h(z))=g(z)$ for $z\in U_1$. 
Pick a slightly larger polyhedron $L'$ containing $L$ in its interior and a small $\epsilon >0$ and set
\[
	A = \{z\in L': \lambda(z)\le 2\epsilon\} \subset U_1, \quad\ 
	B=\{z\in L' : \lambda(z)\ge \epsilon\} \subset \Omega.
\]
The polyhedra $A$ and $B$ form a Cartan pair (see \cite[Definition 5.7.1]{Forstneric2017E}) 
with $A\cup B =L'$ and 
$C:=A\cap B = \{z\in L': \epsilon \le \lambda(z)\le 2\epsilon\}$. Let 
\[
	K'=\{z\in L': \lambda(z)\le \epsilon/2\}.
\]
Pick a convex open set $U_0\subset \C^n$ such that $K'\subset U_0 \subset U_1$ 
and $\overline U_0\cap C=\varnothing$. Choose any holomorphic map $\phi:U_0\to\C^N$. 
Since $K'$ and $C$ are disjoint compact convex sets in $\C^n$,
their union is polynomially convex.
% and $W\subset\C^N$ is a ball, $(K'\cup C) \times \overline W$ is a polynomially convex subset of $\C^n\times \C^N$. 
Hence, the Oka--Weil theorem furnishes a holomorphic map 
$\tilde \phi:A\times W\to \C^N$ which approximates $\phi(z)$ on $(z,w)\in K'\times W$ 
(with $\phi$ independent of $w$) and approximates
$H$ on $C\times W$. In view of \eqref{eq:H}, the holomorphic map $\Phi:A\times W\to Y$ defined by 
\[
	\Phi(z,w) = F(z,\tilde \phi(z,w))\quad \text{for $z\in A$ and $w\in W$}
\]
then approximates  $G$ on $C\times W$, while on $K'\times W$ it is close to the map 
\begin{equation}\label{eq:fphi}
	(z,w)\mapsto f_\phi(z):=F(z,\phi(z)) \quad \text{for $z\in K'$ and $w\in W$}. 
\end{equation}
Recall that the spray $G$ is dominating over $C$. Hence, if the approximations are close enough,
we can apply \cite[Proposition 5.9.2]{Forstneric2017E} on the Cartan pair $(A,B)$
to glue $\Phi$ and $G$ into a holomorphic spray $\Theta:L'\times W'\to Y$ 
for a smaller parameter ball $0\in W'\subset W$. By the construction, its core  
$\tilde f:=\Theta(\cdotp,0):L'\to Y$ then approximates the map $f_\phi$ given by \eqref{eq:fphi} 
on $K'$, which shows that $f_\phi\in \Acal$. Since the map $\wt F$ in \eqref{eq:wtF} is 
injective holomorphic on $V\times W$, every holomorphic map 
$K'\to Y$ sufficiently uniformly close to $f$ on $K'$ is of the form $f_\phi$ in \eqref{eq:fphi}
for a suitable choice of $\phi$, and hence it belongs to the set $\Acal$ of approximable maps. 
This shows that the set $\Acal$ is open as claimed, and therefore $\Acal=\Oscr(K,Y)$.
\end{proof}

%%%%%%%%%%%%%%%%%%
%
%   LOCALIZATION PRINCIPLE
%
%%%%%%%%%%%%%%%%%%
\subsection{A localization theorem for Oka manifolds}\label{ss:localization}
A domain $U$ in a complex manifold $Y$ is 
said to be {\em Zariski open} if $Y\setminus U$ is a closed complex subvariety of $Y$. 
An important application of Theorem \ref{th:K2018} is the following localization criterion for Oka manifolds. 

%
%    LOCALIZATION THEOREM
%
\begin{theorem}\label{th:localization}
{\rm (Kusakabe, \cite[Theorem 1.4]{Kusakabe2021IUMJ}.)}
If $Y$ is a complex manifold which is a union of Zariski open Oka domains, then $Y$ is an Oka manifold.
\end{theorem}

This is one of the most important new results in Oka theory and a
wonderful tool for constructing new examples of Oka manifolds. Several of them 
are described in Kusakabe's paper \cite{Kusakabe2021IUMJ}, and many more will be pointed
out in the sequel. Previously, a localization theorem was known only for algebraically subelliptic
manifolds (see \cite[Lemma 3.5B]{Gromov1989} or \cite[Proposition 6.4.2]{Forstneric2017E}). 
The following is an immediate corollary to Theorem \ref{th:localization}.

\begin{corollary}
Assume that $Y$ is a complex manifold and $Y'$ is a closed complex subvariety of $Y$ such that
$Y\setminus Y'$ is an Oka domain. If for every point $y\in Y'$ there exists a holomorphic automorphism 
$\phi\in\Aut(Y)$ such that $y\notin \phi(Y')$, then $Y$ is an Oka manifold.
\end{corollary}

The proof of Theorem \ref{th:localization} uses the following corollary to 
\cite[Theorems 7.2.1 and 8.6.1]{Forstneric2017E}. 

%
%   REMARK: In the next edition of the book, Theorem 7.2.1 (perhaps also Theorem 8.6.1) 
%   should already be proved in Chapter 5. In this way, the localization theorem for Oka manifolds 
%   could also be included in Chapter 5.
%
%   There is no such problem with Theorem  \ref{th:K2018} which only uses the gluing of sprays.
%
%
\begin{proposition}[Proposition 3.1 in \cite{Kusakabe2021IUMJ}]\label{prop:K2018P31}
Let $\Omega$ be a Zariski open Oka domain in a complex manifold $Y$.
Given a Stein manifold $X$ and a holomorphic map $f:X\to Y$, there is a holomorphic
spray $F:X\times \C^N\to Y$ over $f$ which is dominating on $f^{-1}(\Omega)$.
\end{proposition}

%
%   PROOF OF THE LOCALIZATION THEOREM
%
\begin{proof}[Proof of Theorem \ref{th:localization}]
By Theorem \ref{th:K2018} it suffices to show that $Y$ enjoys condition \CEll1. 
Let $K$ be a compact convex set in $\C^n$ and $f\in \Oscr(U,Y)$ be a holomorphic map 
on an open neighbourhood $U\subset \C^n$ of $K$. Let $\Omega_i$ be a
collection of Zariski open domains in $Y$ with $\bigcup_i \Omega_i=Y$. Since $K$ is compact, 
$f(K)$ is contained in the union of finitely many $\Omega_i$'s; 
call them $\Omega_1,\ldots, \Omega_m$.
Proposition \ref{prop:K2018P31} furnishes a spray $F_1:U\times \C^{N_1}\to Y$ with the core $f$
which is dominating on $f^{-1}(\Omega_1)$. Applying Proposition \ref{prop:K2018P31}
to $F_1$ furnishes another spray $F_2:(U\times \C^{N_1}) \times \C^{N_2}\to Y$ with the core 
$F_1$ which is dominating on $F_1^{-1}(\Omega_2)$. Considering $F_2$ as a spray over $f:U\to Y$,
it is dominating on $f^{-1}(\Omega_1\cup \Omega_2)$ 
(since $F_1$ is dominating on $f^{-1}(\Omega_1)$).
After $m$ steps of this kind we obtain a spray $F:U\times \C^N\to Y$ over $f$ which is
dominating on a neighbourhood of $K$.
\end{proof}

%
%   Sprays generating tangent spaces 
%
\subsection{Sprays generating tangent spaces}\label{ss:sprays}
Theorem \ref{th:K2018} implies several other criteria for a manifold $Y$ to be
Oka. The following result combines Corollaries 4.1 and 4.2 in 
Kusakabe's paper \cite{Kusakabe2021IUMJ}.

\begin{corollary}\label{cor:Kcor4.1}
The following conditions are equivalent for every complex manifold $Y$. 
\begin{enumerate}[\rm (a)]
\item The manifold $Y$ is Oka.
\item For every Stein manifold $X$, holomorphic map $f:X\to Y$, and holomorphic section $V$
of $f^*TY\to X$ %(such $V$ may be thought of as a holomorphic vector field on $Y$ along the map $f$) 
there is a holomorphic spray $F:X\times \C\to Y$ over $f$ such that 
\[
	\di_t|_{t=0}F(x,t)=V(x)\in T_{f(x)}Y\quad \text{for all $x\in X$}.
\]
\item For every Stein manifold $X$, holomorphic map $f:X\to Y$ and point $x\in X$ 
there are finitely many holomorphic sprays $F_j:X\times \C^{N_j}\to Y$ $(j=1,\ldots,k)$ over $f$
such that 
\[
	\sum_{j=1}^k \di_t|_{t=0} F_j(x,t)(\C^{N_j}) = T_{f(x)} Y. 
\]
\item Condition {\rm (c)} holds for every convex domain $X\subset \C^n$, $n\in\N$.
\end{enumerate}
\end{corollary}

%Condition (c), applied to sprays over the identity map $f=\Id_Y$ on the manifold $Y$ (which need not be Stein), implies {\em weak subellipticity} of $Y$ (see \cite[Definition 5.6.13 (f)]{Forstneric2017E}), and this condition implies that $Y$ is Oka (see \cite[Corollary 5.6.14]{Forstneric2017E}).

\begin{proof}
(a)$\Rightarrow$(b): If $Y$ is Oka then by  \cite[Corollary 8.8.7]{Forstneric2017E} there is 
a dominating holomorphic spray $G:X\times \C^N\to Y$ over $f=G(\cdotp,0)$ for some $N\in\N$.
This means that 
%F changed t to w in the next line
\[
	\Theta:=\di_w|_{w=0} G(\cdotp,w):X\times \C^N\to f^*TY
\]
is a surjective holomorphic vector bundle map,  
so there is a holomorphic map $W:X\to\C^N$ such that $\Theta(x,W(x))=V(x)$ 
for all $x\in X$ (see \cite[Corollary 2.6.5]{Forstneric2017E}). 
The holomorphic spray $F:X\times \C\to Y$ defined by $F(x,t)=G(x,tW(x))$ 
then satisfies condition (b).

The implications (b)$\Rightarrow$(c)$\Rightarrow$(d) are obvious.

(d)$\Rightarrow$(a): Let $K\subset \C^n$ be a compact convex set, $X\subset \C^n$ be
an open convex set containing $K$, and $f\in \Oscr(X,Y)$. Fix $x\in K$. By condition (d) 
there is a  spray $F_1:X\times \C\to Y$ over $f$ such that the vector
$V_1:=\di_t|_{t=0}F_1(x,t)\in T_{f(x)}Y$ is nonzero. Applying condition (d) to $F_1$
gives a spray $F_2:X\times\C\times \C\to Y$ over $F_1$ such that the vector 
$V_2:=\di_t|_{t=0}F_2(x,0,t)\in T_{f(x)}Y$ is linearly independent from $V_1$.
Continuing in this way we obtain after $d=\dim Y$ steps a spray 
$F:X\times \C^d\to Y$ over $f$ which is dominating at $x$, and hence on a neighbourhood of $x$.
A repetition of this process over other points of $K$ 
gives a holomorphic spray over $f$ which is dominating on an open neighbourhood 
of $K$ in $X$. Thus, $Y$ enjoys condition \CEll1 and hence is Oka 
by Theorem  \ref{th:K2018}.
\end{proof}

%
%	C-CONNECTEDNESS
%
\subsection{$\C$-connectedness}\label{ss:Cconnect}
In \cite{Kusakabe2017}, Kusakabe characterized Oka manifolds 
by the following $\C$-connectedness property of the space of holomorphic maps from Stein manifolds.

\begin{theorem} \label{th:Kusakabe2017} 
{\rm (\cite[Theorem 3.2]{Kusakabe2017})}
For a complex manifold $Y$ the following are equivalent.
\begin{enumerate}[\rm (a)]
\item $Y$ is an Oka manifold.
\item For every Stein manifold $X$ and homotopic holomorphic maps $f_0,f_1:X\to Y$
there is a holomorphic map $F:X\times\C\to Y$ such that $F(\cdotp,0)=f_0$ and $F(\cdotp,1)=f_1$.
\item Condition {\rm (b)} holds for every bounded convex domain $X$ in $\C^n$, $n\in\N$.
\end{enumerate}
\end{theorem}

This result and its proof are also presented in \cite[Theorem 5.15.2]{Forstneric2017E}.
The implication (a)$\Rightarrow$(b) follows from the $1$-parameteric Oka principle 
for holomorphic maps into Oka manifolds, and (b)$\Rightarrow$(c) is obvious. 
The proof of the main implication (c)$\Rightarrow$(a) reduces to showing that condition (c) implies CAP
on special polyhedral pairs (see Lemma \ref{lem:SCAP}). This uses a similar idea
as the proof of Theorem \ref{th:K2018}. 

We mention the following open problem; see \cite[Problem 7.6.4]{Forstneric2017E}.

%
%	UNION PROBLEM
%
\begin{problem}[The union problem for Oka manifolds] \label{prob:union}
Let $Y$ be a complex manifold and $Y' \subset Y$ be a closed complex submanifold. 
If $Y'$ and $Y\setminus Y'$ are Oka, is $Y$ Oka? 
In particular, if $Y$ is a complex manifold and $p\in Y$ is such that $Y\setminus \{p\}$ is Oka,
is the blowup $\Bl_p Y$ Oka?
\end{problem}

This situation occurs in Kummer surfaces: every such surface admits $16$ pairwise disjoint 
embedded rational curves such that the complement of their union is Oka (see \cite[Sect.\ 7.2]{Forstneric2017E}).
In an attempt to approach this problem, 
Kusakabe combined Theorem \ref{th:Kusakabe2017} with \cite[Theorem 7.2.1]{Forstneric2017E} to 
show the following (see \cite[Theorem 4.4]{Kusakabe2021IUMJ}). 

\begin{theorem}\label{th:K2018-4.4}
Given a complex manifold $Y$ with a Zariski open Oka domain $U\subset Y$, the following conditions 
are equivalent.
\begin{enumerate}[\rm (a)]
\item $Y$ is an Oka manifold.
\item For every Stein manifold $X$ and map $f\in\Oscr(X,Y)$ which is homotopic 
to a continuous map $X\to U$ there exists $F\in\Oscr(X\times\C, Y)$ with 
$F(\cdotp,0)=f$ and $F(\cdotp, 1)\in \Oscr(X,U)$.
\item For every bounded convex domain $X$ in $\C^n$ $(n\in\N)$ and map $f\in \Oscr(X,Y)$
there is a holomorphic map $F:X\times\C\to Y$ such that $F(\cdotp,0)=f$ and $F(\cdotp, 1)\in \Oscr(X,U)$.
\end{enumerate}
\end{theorem}

It is not clear how to find a map $F$ satisfying condition (c) if $f(X)$ intersects both $U$ and the subvariety 
$Y'=Y\setminus U$. If $X$ is convex and $f(X)\subset (Y')_\reg$, then such a map exists by 
\cite[proof of Theorem 2]{ForstnericLarusson2014IMRN} (see also \cite[proof of Theorem 7.1.8]{Forstneric2017E}).

%%%%%%%%%%%%%%%%%%%%%%%%%%%
%
%	ELLIPTIC CHARACTERIZATION OF OKA MAPS
%
%%%%%%%%%%%%%%%%%%%%%%%%%%%
\subsection{Elliptic characterization of Oka maps} \label{ss:Okamap}

The main new result presented in this section is Kusakabe's characterization in \cite{Kusakabe2021MZ} 
of the Oka property of holomorphic submersions by {\em convex ellipticity}; see Theorem \ref{th:ellipticmap}. 
We begin with some background.

A holomorphic map $h:Y\to Z$ of reduced complex spaces is said to 
enjoy the {\em parameteric Oka property with approximation and interpolation} (POPAI) 
if for every holomorphic map $f:X\to Z$ from a reduced Stein space, each continuous lift $F_0:X\to Y$ 
is homotopic (through lifts of $f$) to a holomorphic lift $F=F_1:X\to Y$ as in the following diagram, 
\[
    \xymatrix{  & Y \ar[d]^{h} \\ X \ar[r]^{\ \ f} \ar@{-->}[ur]^{F} & Z}
\]
with approximation on a compact $\Oscr(X)$-convex subset of $X$
and interpolation on a closed complex subvariety of $X$ on which $F_0$ is 
holomorphic. Furthermore, the analogous conditions must hold for families of maps $f_p:X\to Z$
depending continuously on  a parameter $p$ in a compact Hausdorff space; see
\cite[Definitions 7.4.1 and 7.4.7]{Forstneric2017E} for the details.
When $Z$ is a singleton, these are the usual Oka properties of the complex manifold $Y$.

%
%   OKA
%
\begin{definition}\label{def:Okamap}
A holomorphic map $h:Y\to Z$ between reduced complex spaces 
is an {\em Oka map} if it enjoys POPAI and is a Serre fibration
(see \cite{Larusson2004,Forstneric2010CR} and \cite[Definition 7.4.7]{Forstneric2017E}).
\end{definition}

For a holomorphic submersion $h:Y\to Z$, POPAI is a local condition in the sense 
that it holds if (and only if) every point $z_0\in Z$ has an open neighbourhood $U\subset Z$ such that 
the restricted submersion $h:h^{-1}(U)\to U$ enjoys POPAI (see \cite[Theorem 4.7]{Forstneric2010PAMQ} 
or  \cite[Definition 6.6.5 and Theorem 6.6.6]{Forstneric2017E}).
Furthermore, for such $h$ the basic Oka property (referring to lifts of single maps $f:X\to Z$)
implies POPAI; see \cite{Forstneric2010CR}. 

Since a holomorphic fibre bundle projection is a Serre fibration, the above localization argument shows 
that it is an Oka map if and only if the fibre is an Oka manifold.
Furthermore, every stratified subelliptic holomorphic submersion satisfies POPAI,
so it is an Oka map provided that it is a Serre fibration (see \cite[Corollary 7.8.4]{Forstneric2017E}).

Let us look more closely at Oka maps between complex manifolds.

%
%   AN OKA MAP IS A SURJECTIVE SUBMERSION
%
\begin{proposition}\label{prop:Okamaps}
Let $h:Y\to Z$ be an Oka map between complex manifolds with $Z$ connected. Then, $h$ is a surjective
submersion and the fibres $h^{-1}(z)$ for $z\in Z$ are Oka manifolds.
\end{proposition}

\begin{proof}
Let $z_0=h(y_0)\in h(Y)$ for some $y_0\in Y$. Since $Z$ is connected, there is a path $\gamma$ from $z_0$ to any given point $z\in Z$. 
Since $h$ enjoys the homotopy lifting property, we can lift $\gamma$ to a path in $Y$ starting at $y_0$;
its terminal point $y$ then satisfies $h(y)=z$. Hence, $h$ is surjective. By a similar argument, given a point $y_0\in Y$
there are a contractible open neighbourhood $U\subset Z$ of $z_0=h(y_0)$ and a continuous 
section $f_0:U\to Z$ with $f_0(y_0)=z_0$ and $h\circ f_0=\Id_U$. Since $h$ is an Oka map, we can deform 
$f_0$ to a holomorphic section $f:U\to Y|_U$ of $h$ with $f(z_0)=y_0$. The restriction $h:f(U)\to U$ is then 
a biholomorphism, which shows that $h$ is a submersion at $y_0$. The fact that the fibres of $h$ are 
Oka manifolds follows from the definition of an Oka map (and it holds if $h$ enjoys POPAI).
\end{proof}

The following result gives many new examples of Oka manifolds from the existing ones.
The special case concerning holomorphic fibre bundles with Oka fibres was proved in 
\cite{Forstneric2006AM}; see also \cite[Theorem 5.6.5]{Forstneric2017E}. Both proofs contain
a minor glitch related to the (non-) existence of Stein neighbourhoods of certain sets, and 
we give a correct proof here.  

%
%    UP-DOWN
%
\begin{theorem}\label{th:updown}
If $h:Y\to Z$ is an Oka map of complex manifolds with $Z$ connected, then $Y$ is an Oka manifold 
if and only if $Z$ is an Oka manifold. 
This holds in particular if $h:Y\to Z$ is a holomorphic fibre bundle with an Oka fibre.
\end{theorem}

\begin{proof}
Assume first that $Y$ is an Oka manifold and let us prove that so is $Z$. We shall verify CAP.
Let $K$ be a compact convex in $\C^n$, and let $f_0:U\to Z$ be a holomorphic map from 
an open convex neighbourhood $U\subset\C^n$ of $K$. Since $h$ is an Oka map and $U$ 
is contractible, $f_0$ lifts to a holomorphic map $g_0:U\to Y$ with $h\circ g_0=f_0$. 
Since $Y$ is Oka, we can approximate $g_0$ as closely as desired
uniformly on $K$ by a holomorphic map $g:\C^n\to Y$. The map $f=h\circ g:\C^n\to Z$
then approximates $f_0$ on $K$, so $Z$ enjoys CAP and hence is Oka. 

To proves the converse part, we shall need the following lemma.

%
%  This lemma could be included in Section 3.3, next to Theorem 3.3.5
%  
%  
 
\begin{lemma}[Lemma 3.4 in \cite{Forstneric2004AIF}]\label{lem:AIF2004}
Let  $h': Y'\to Z'$ be a holomorphic submersion of a Stein manifold $Y'$ onto
a complex manifold $Z'$. Then there are an open Stein domain $W \subset Z'\times Y'$ 
containing the submanifold $S :=\{(z',y') \in Z' \times Y' :  h'(y')=z'\}$ and a holomorphic retraction 
$\wt \rho : W\to S$ of the form $\wt \rho(z',y')=(z',\rho(z',y'))$ for $(z',y')\in W$.
\end{lemma}

The proof of Theorem \ref{th:updown} can now be completed as follows.
Assuming that $Z$ is an Oka manifold, we shall verify that $Y$ enjoys CAP and hence is Oka.
Consider the manifolds $\wt Y=\C^n \times Y$, $\wt Z=\C^n \times Z$ and the projection 
$\tilde h: \wt Y\to \wt Z$ given by $\tilde h(\zeta,y)=(\zeta,h(y))$ for $\zeta\in\C^n$ and $y\in Y$. 
By Proposition \ref{prop:Okamaps}, $\tilde h$ is a surjective holomorphic submersion.
Let $U\subset \C^n$ be an open convex neighbourhood of a compact convex set $K\subset \C^n$
and $f_0:U\to Y$ be a holomorphic map. Set $g_0=h\circ f_0:U\to Z$. The graphs
\[
	\Gamma_{f_0}=\{(\zeta,f_0(\zeta)):\zeta \in U\}\subset \wt Y\quad\text{and}\quad 
	\Gamma_{g_0}=\{(\zeta,g_0(\zeta)): \zeta\in U\}\subset \wt Z
\]
are locally closed Stein submanifolds of $\wt Y$ and $\wt Z$, which 
by Siu's theorem \cite{Siu1976} admit open Stein neighbourhoods 
$Y' \subset \wt Y$ and $Z' \subset \wt Z$, respectively. 
They can be chosen such that $\tilde h|_{Y'}: Y' \to Z'$ is a surjective holomorphic submersion.
For every point $p\in Z'$, Lemma \ref{lem:AIF2004} furnishes a holomorphic 
retraction $\rho_p$ from a neighbourhood of the fibre $Y'_p=Y'\cap  \tilde h^{-1}(p)$ onto $Y'_p$, depending holomorphically on $p\in Z'$.  

Since $Z$ is Oka, we can approximate the map $g_0:U\to Z$ uniformly on $K$
by a holomorphic map $g:\C^n\to Z$. If the approximation is close enough then 
for all $\zeta$ in a neighbourhood $V\subset U$ of $K$ the point $(\zeta,f_0(\zeta))\in Y'$ 
lies in the domain of the retraction $\rho_{(\zeta,g(\zeta))}$. 
For $\zeta \in V$ let $f_1(\zeta)\in Y$ denote the projection of the point 
$\rho_{(\zeta,g(\zeta))}(\zeta,f_0(\zeta))$ to $Y$. Then, the map $f_1:V\to Y$ 
is holomorphic, uniformly close to $f_0$ on $K$, and it satisfies $h\circ f_1(\zeta)=g(\zeta)$ 
for $\zeta\in V$. Since $h:Y\to Z$ is an Oka map, $g:\C^n\to Z$
is a holomorphic map, and $f_1$ is a holomorphic lift of $g$ over $V\supset K$, 
we can approximate $f_1$ uniformly on $K$ by a holomorphic map $f:\C^n\to Y$ 
satisfying $h\circ f=g$. Hence, $Y$ enjoys CAP and so is an Oka manifold.
\end{proof}

\begin{remark}
Knowing that Oka maps are fibrations in the model structure constructed by L\'arusson 
\cite{Larusson2004} helps us understand and predict their behaviour. For example, it is immediate by abstract
nonsense that the composition of Oka maps is Oka, that a retract of an Oka
map is Oka, and that the pullback of an Oka map by an arbitrary holomorphic
map is Oka. Also, in any model category, the source of a fibration with a fibrant target is fibrant. 
It follows that the source of an Oka map with an Oka target is Oka. On the other hand, the fact
that the image of an Oka map with an Oka source is Oka is a surprising
feature of Oka theory not predicted by abstract nonsense, and its proof depends on the fact
that the Oka property can be detected using the CAP property, which pertains
to approximation of maps from contractible sets.
%to Stein inclusions of the special kind $T\hra\C^n$, where $T$ is contractible.
\end{remark}

We have already remarked that if $h:Y\to Z$ is a holomorphic submersion enjoying 
POPAI then every fibre of $h$ %$h^{-1}(z)$, $z\in Z$, 
is an Oka manifold. The converse fails in general. For example, let $g:Z\to \C$ be a continuous function 
on a domain $Z\subset\C$ and consider the map
\[
	h:Y=\{(z,w)\in Z\times \C: w\ne g(z)\}\to Z, \quad h(z,w)=z.
\] 
Every fibre of $h$ is the Oka manifold $\C^*$, $h$ is a topological fibre bundle and a
Serre fibration, but $h$ is an Oka map if and only if $g$ is a holomorphic function 
(see \cite[Corollary 7.4.10]{Forstneric2017E}).

The following result of Kusakabe \cite[Lemma 5.1]{Kusakabe2020complements} 
shows that a manifold is Oka if it admits sufficiently many submersive projections having the Oka property.

%
%    OKA PROJECTIONS IMPLY OKA
%
\begin{proposition} \label{prop:projections}
Assume that for every point $y$ in a complex manifold $Y$ there exist complex manifolds 
$Z_1,\ldots, Z_k$ and holomorphic submersions $h_j:Y\to Z_j$ $(j=1,\ldots,k)$ enjoying {\rm POPAI} 
such that $T_y Y = \sum_{j=1}^k T_y \left[h_j^{-1}(h_j(y))\right]$. 
Then $Y$ is an Oka manifold.
\end{proposition}

\begin{proof}
Let $f:X\to Y$ be a holomorphic map from a Stein manifold $X$.  Fix a point $x\in X$
and let $h_j:Y\to Z_j$ be submersions satisfying the hypothesis in the proposition at $y=f(x)\in Y$. 
The Oka property of $h_j$ furnishes 
a fibre dominating holomorphic spray $F_j:X\times \C^{N_j}\to Y$ over $f$
with $h_j \circ F=h_j\circ f$ (see \cite[Corollary 8.8.7]{Forstneric2017E}). In particular,
$\di_t|_{t=0}F_j(x,t) (\C^{N_j}) = T_{y} \left[h_j^{-1}(h_j(y))\right]$. 
Since $\sum_{j=1}^k T_{y} \left[h_j^{-1}(h_j(y))\right]=T_{y} Y$,  
the sprays $F_1,\ldots,F_k$ dominate $T_{y}Y$.
Hence, $Y$ is Oka by Corollary \ref{cor:Kcor4.1} (the equivalence of (a) and (c)). 
\end{proof}

As pointed out in the introduction to Kusakabe's paper \cite{Kusakabe2021MZ}, 
the two main types of maps which are known to satisfy POPAI
are (stratified) fibre bundles with Oka fibres and (stratified) subelliptic 
submersions. None of these two families is a subfamily of the other one: 
there are Oka manifolds which fail to be subelliptic (see Section \ref{sec:complements}),
and there are subelliptic submersions which are not locally trivial at any base point, 
e.g.\ a complete family of complex tori \cite[Theorem 16]{Larusson2012}. 
Kusakabe also gave an example of a holomorphic submersion enjoying POPAI
which does not belong to any of these two classes \cite[Proposition 5.10]{Kusakabe2021MZ}. 
It is therefore of interest to find a characterization of POPAI which unifies the theory
in the same way as CAP and $\Ell1$ characterize Oka manifolds (cf.\ Theorem \ref{th:K2018}). 
%F removed empty line, inserted "To this end"
To this end, Kusakabe introduced the following notion (see \cite[Definition 1.2]{Kusakabe2021MZ}).

%
%   CONVEXLY ELLIPTIC MAPS
%
\begin{definition} %[Definition 1.2 in \cite{Kusakabe2021MZ}]
\label{def:ellipticmap}
A holomorphic submersion $h:Y\to Z$ of complex spaces is {\em convexly elliptic} 
if there exists an open cover $\{U_i\}_{i\in I}$ of $Z$ such that for every compact convex set $K\subset\C^n$
$(n\in\N)$ and holomorphic map $f\in \Oscr(K,Y)$ 
with $f(K)\subset h^{-1}(U_i)$ for some $i\in I$ there are a neighbourhood $V\subset \C^n$
of $K$ and a holomorphic map $F : V \times \C^N \to Y$ satisfying %the following conditions:
\begin{enumerate}[\rm (i)]
\item $F(\cdotp,0)=f$,
\item  $h\circ F(z,t) = h\circ f(z)$ for all $z\in V$ and $t\in\C^N$, and
\item $F(z,\cdotp):\C^N\to h^{-1}(h(f(z)))$ is a submersion at $0\in\C^N$ for all $z\in V$. 
\end{enumerate}
\end{definition}

A map $F$ as in the above definition is called a {\em fibre dominating spray over $f$}.
Note that convex ellipticity is a fibred version of condition \CEll1 (cf.\ Definition \ref{def:elliptic} (c)).

%
% Kusakabe also introduced a stratified version of this condition (see \cite[Remark 3.3]{Kusakabe2021MZ}), although this is not really needed since in the proof one proceeds by induction on strata as in \cite[proof of Theorem 6.2.2]{Forstneric2017E}.

%
%   ELLIPTIC characterizATION OF OKA MAPS
%
\begin{theorem}{\rm (Kusakabe \cite[Theorem 1.3]{Kusakabe2021MZ})} \label{th:ellipticmap}
A holomorphic submersion of complex spaces enjoys {\rm POPAI} if and only if it is convexly elliptic.
In particular, a holomorphic submersion is an Oka map if and only if it is a convexly elliptic
Serre fibration.
\end{theorem}

In view of the fact that a complex manifold $Y$ is an Oka manifold if and only if the constant map $Y\to\mathrm{point}$
is an Oka map, Theorem \ref{th:ellipticmap} generalizes Theorem \ref{th:K2018}, the latter characterizing 
Oka manifolds by condition \CEll1. 

The proof of Theorem \ref{th:ellipticmap} in \cite[Sections 3-4]{Kusakabe2021MZ} is similar
to the proof of Theorem \ref{th:K2018}. First, the problem is reduced to the main special case which pertains 
to sections of a holomorphic submersion $h:Y\to Z$. 
In this case, and assuming that the base $Z$ is Stein, an axiomatic characterization of POPAI is 
provided by the {\em homotopy approximation property}, 
HAP, first introduced in \cite[Proposition 2.1]{Forstneric2010CR}. 
(See also \cite[Definition 6.6.5 and Theorem 6.6.6]{Forstneric2017E}.)\footnote{Condition HAP is not stated correctly in 
\cite[Definition 6.6.5]{Forstneric2017E}:  the same condition must hold for every local holomorphic spray of 
sections with parameter in a ball $\B\subset \C^n$. Equivalently, the stated condition must apply to each 
trivial extension $Z\times \B\to X\times\B$ of the given submersion $Z\to X$.
This holds for every subelliptic submersion $Z\to X$ by \cite[Theorem 6.6.2]{Forstneric2017E}.}
This condition, which is local on the base, is an axiomatization of the homotopy Runge theorem 
(see \cite[Theorem 6.6.2]{Forstneric2017E}).
The gist of Kusakabe's proof of Theorem \ref{th:ellipticmap} is to show that HAP is implied by 
convex ellipticity in a similar way as CAP is implied by condition \CEll1 (see Theorem \ref{th:K2018}
for the latter). We refer to \cite{Kusakabe2021MZ} for the details.

%%%%%%%%%%%%%%%%%%%
%
%	SECTION: OKA DOMAINS IN CN
%
%%%%%%%%%%%%%%%%%%%

\section{Oka domains in Euclidean spaces and in Stein manifolds with the density property}\label{sec:complements}

A long-standing problem in Oka theory asked whether the complement of 
every compact convex set $K$ in $\C^n$ for $n>1$ is an Oka manifold 
(see \cite[Problem 7.6.1]{Forstneric2017E}). In 2020, 
Kusakabe \cite{Kusakabe2020complements} answered this problem affirmatively 
and in a much greater generality. 

We recall the following notion introduced by Varolin \cite{Varolin2001}; 
see also \cite[Definition 4.10.1]{Forstneric2017E}.

%
%   DENSITY PROPERTY
%
\begin{definition}[Varolin \cite{Varolin2001}] \label{def:density}
A complex manifold $X$ has the {\em density property} if every holomorphic 
vector field on $X$ can be approximated uniformly on compacts by Lie combinations
(sums and commutators) of complete holomorphic vector fields on $X$.

An algebraic manifold $X$ has the {\em algebraic density property} if the 
Lie algebra of algebraic vector fields on $X$ is generated by complete algebraic vector fields.
\end{definition}

Every holomorphic vector field on an affine algebraic manifold is a limit of algebraic vector fields,
and hence the algebraic density property implies the holomorphic density property.
Note that flows of complete algebraic vector fields in the above definition
need not be algebraic. Algebraic vector fields having algebraic flows are called {\em locally nilpotent
derivations}, abbreviated LNDs, and they are much more special. 

On a Stein manifold $X$, the density property implies the Anders\'en--Lempert theorem concerning 
approximation of isotopies of biholomorphic maps between Stein Runge domains in $X$ by isotopies of holomorphic automorphisms of $X$; see \cite[Theorem 4.10.5]{Forstneric2017E}. 
Every Stein manifold with the density property has dimension $>1$ and is an Oka manifold 
(see \cite[Theorem 4]{KalimanKutzschebauch2008MZ} or \cite[Theorem 5.5.18]{Forstneric2011E}).
For surveys, see \cite[Chapter 4]{Forstneric2017E},   
\cite{ForstnericKutzschebauch2022}, and  \cite{Kutzschebauch2020}. 

We can now state Kusakabe's result. 

%
%   KUSAKABE'S THEOREM
%
\begin{theorem}
{\rm (Kusakabe \cite[Theorem 1.2 and Corollary 1.3]{Kusakabe2020complements}.)} 
\label{th:complement}
If $Y$ is a Stein manifold with the density property and $K$ is a compact 
$\Oscr(Y)$-convex set in $Y$ then the complement $Y\setminus K$ is an Oka manifold. In
particular, if $K$ is a compact polynomially convex set in $\C^n$ for $n>1$, then 
$\C^n\setminus K$ is an Oka manifold. 
\end{theorem}

Since the interior $X=\mathring K$ of a polynomially convex set $K$ in $\C^n$ 
is Stein \cite[Corollary 2.5.7]{Hormander1990}, Theorem \ref{th:complement}
gives many Stein--Oka decompositions 
$\C^n=\Omega\cup \overline X$, where $X$ is a bounded Stein domain with polynomially
convex closure and $\Omega = \C^n\setminus \overline X$ is an Oka domain.
This phenomenon is rather symbolic since Oka manifolds are in a certain sense dual to Stein
manifolds, being the most natural targets of holomorphic maps from Stein manifolds.

It seems reasonable to introduce the following property.

%
%    OKA AT INFINITY
%
\begin{definition}\label{def:Okainfinity}
A noncompact complex manifold $Y$ is {\em Oka at infinity} if there is
an exhaustion $K_1\subset K_2\subset \cdots \subset \bigcup_{j=1}^\infty =Y$
by compact sets such that $Y\setminus K_j$ is Oka for every $j\in\N$.
\end{definition}

Thus, Theorem \ref{th:complement} says that every Stein manifold with the density property 
is Oka at infinity. Besides its intrinsic interest, Theorem \ref{th:complement} is a
very useful tool for constructing proper holomorphic maps to such manifolds;
see Remark \ref{rem:propermaps}.

Kusakabe's proof of Theorem \ref{th:complement} uses 
the characterization of Oka manifolds by condition \CEll1 (see Theorem \ref{th:K2018}). 
Given a compact convex set $L\subset \C^N$ and a holomorphic map 
$f:L\to Y$ such that $f(z)\in Y\setminus K$ for every $z\in L$, he 
constructed a holomorphically varying family 
$f(z)\in \Omega_z \subset Y\setminus K$ $(z\in L)$ of nonautonomous basins with uniform bounds 
(i.e., basins of random sequences of automorphisms of $Y$ which are uniformly attracting at 
the fixed point $f(z)\in Y \setminus K$); such basins are elliptic manifolds as shown by 
Forn\ae ss and Wold \cite{FornaessWold2016}, hence Oka.  It is then possible
to find a dominating holomorphic spray $F:L\times \C^n\to Y$ over $f=F(\cdotp,0)$
such that $F(z,\zeta)\in \Omega_z$ for all $z\in L$ and $\zeta\in\C^n$. 
Thus, $Y\setminus K$ satisfies condition \CEll1 and hence is Oka.

Soon thereafter, Wold and the author pointed out in \cite{ForstnericWold2020MRL} that 
one can choose a spray $F$ as above such that $F(z,\cdotp):\C^n\to Y\setminus K$ 
is a Fatou--Bieberbach map for every $z\in L$. The following result
of independent interest therefore implies Theorem \ref{th:complement}. 

%%%%%%%%%%%%%%%%%
%
%   FRANCI IN ERLEND, 2020
%
%%%%%%%%%%%%%%%%%
\begin{theorem}[Theorems 1.1 and 3.1 in \cite{ForstnericWold2020MRL}] \label{th:FW2020}
Let $Y$ be a Stein manifold with the density property, $K$ be a compact 
$\Oscr(Y)$-convex set in $Y$, $L$ be a compact convex set in $\C^N$ for some $N\in\N$, 
and $f:U\to \C^n$ be a holomorphic map on an open neighbourhood $U\subset \C^N$ 
of $L$ such that $f(z)\in Y\setminus K$ for every $z\in L$. 
Then there are an open neighbourhood $V\subset U$ of $L$
and a holomorphic map $F:V\times \C^n\to Y$ with $n=\dim Y$ such that for every $z\in V$
we have that $F(z,0)=f(z)$ and the map $F(z,\cdotp):\C^n\to Y\setminus K$ is injective
(a Fatou--Bieberbach map).
If $Y=\C^n$ with $n>1$ then the same conclusion holds if $L$ is polynomially convex.
\end{theorem}

%
%  CONSTRUCTING PROPER MAPS TO MANIFOLDS WITH DENSITY PROPERTY
%
\begin{remark}\label{rem:propermaps}
In the papers \cite{AndristWold2014,AndristForstnericRitterWold2016,Forstneric2019JAM}
it was proved that every Stein manifold $X$ admits a proper holomorphic embedding
in any Stein manifold $Y$ with the density property, or the volume density property
with respect to a holomorphic volume form on $Y$, if $\dim Y>2\dim X$,
and it admits a proper holomorphic immersion if $\dim Y\ge 2\dim X$.
In the case when $Y$ has the density property, the proofs can be substantially 
simplified by using Theorem \ref{th:complement}; here is an outline. 
Suppose that $D$ is a relatively compact, smoothly bounded domain in $X$ 
whose closure is $\Oscr(X)$-convex and $f:X\to Y$ is a continuous map which is
holomorphic on $\bar D$. Given a compact $\Oscr(Y)$-convex set $L\subset Y$, 
one can use the technique in \cite{DrinovecForstneric2010AJM} 
to deform $f$ to another map $\tilde f:X\to Y$ which is 
holomorphic on $\bar D$, close to $f$ on a given compact subset of $D$ and 
satisfies $\tilde f(bD)\subset Y\setminus L$. Since the domain $Y\setminus L$ is Oka
by Theorem \ref{th:complement}, we can apply Theorem \ref{th:Oka2}
to approximate $\tilde f$ uniformly on $\bar D$ by a holomorphic map $g:X\to Y$
homotopic to $f$ such that $g(X\setminus \mathring D) \subset Y\setminus L$.
Continuing inductively and using also the general position theorem at every step, 
we obtain a sequence of holomorphic embeddings (or immersions) from an increasing 
sequence of domains exhausting $X$ to $Y$, converging to a proper holomorphic 
embedding or immersion $X\to Y$. 

This scheme does not work if $Y$ has the volume density property but not the
density property. A quintessential example is $(\C^*)^n$ 
which has the volume density property with respect to the volume form
$dz_1\wedge dz_2\wedge\cdots\wedge dz_n/z_1z_2\cdots z_n$, but it is
not known to have the density property. It is not known whether 
Theorem \ref{th:complement} holds for such manifolds
since holomorphic vector fields which are contracting at some point are not 
volume preserving. However, such a manifold still has Oka property at infinity 
for maps $X\to Y$ from Stein manifolds $X$ of dimension $\dim X<\dim Y$. 
\end{remark}

It was proved by Andrist, Shcherbina, and Wold \cite{AndristShcherbinaWold2016} that, in a Stein
manifold $X$ of dimension at least three, every compact holomorphically convex set $K$
with infinitely many limit points has non-elliptic complement $X\setminus K$. 
An important point in the proof is that every holomorphic line bundle $E\to X\setminus K$
extends to a holomorphic line bundle on the complement of at most finitely many points, 
and hence a spray $F:E\to X$ defined on such a bundle cannot have values contained in 
$X\setminus K$. Together with Theorem \ref{th:complement} this gives the following corollary which 
answers a question of Gromov \cite[3.2.A'']{Gromov1989} 
(see also \cite[p.\ 325]{Forstneric2017E}). 

%
%   OKA BUT NOT SUBELLIPTIC
%
\begin{corollary}\label{cor:Okanotell}
Let $n\ge 3$. For every compact polynomially convex set $K\subset \C^n$
with infinitely many limit points the complement $\C^n\setminus K$ is Oka but not subelliptic. 
The analogous result holds in any Stein manifold of dimension $\ge 3$ with the density property.
\end{corollary}

The first known examples of Oka manifolds which fail to be subelliptic were given by 
Kusakabe in \cite{Kusakabe2020PAMS}. One of the main results of that paper is the following.

%
%   COMPLEMENTS OF COUNTABLE SETS
%
\begin{theorem}[Theorem 1.2 in \cite{Kusakabe2020PAMS}]\label{th:KusakabePAMS}
If $S$ is a closed tame countable set in $\C^n$, $n>1$, whose set of limit points 
is discrete, then $\C^n\setminus S$ is an Oka domain.
\end{theorem}

Here, a closed countable set $S\subset\C^n$ is called tame if there is a holomorphic automorphism
$\Phi$ of $\C^n$ such that the closure of $\Phi(S)$ in $\CP^n$ does not contain the entire 
hyperplane at infinity. %The same definition applies to complex subvarieties of $\C^n$.
It was previously known that the complement of a closed tame subvariety of $\C^n$ 
of codimension at least $2$ %(in particular, the complement  of a closed tame discrete subset) 
is elliptic and hence Oka; see \cite[Proposition 5.6.17]{Forstneric2017E}. 

In \cite{Kusakabe2020PAMS}, Kusakabe also constructed compact countable 
sets in $\C^n$ with non-discrete sets of limit points and having non-elliptic Oka complements.
An example is the following. Let
$\N^{-1}=\{1/j:j\in\N\}$ and $\overline \N^{-1}=\N^{-1}\cup\{0\}\subset\C$.
The domain 
$
	X=\C^n\setminus \big((\overline \N^{-1})^2 \times \{0\}^{n-2}\big)
$
for $n\ge 3$ is an Oka manifold which is not weakly subelliptic (see \cite[Corollary 1.4]{Kusakabe2020PAMS}).

The corresponding problem in complex dimension $2$ remains open. The reason  
is that a holomorphic line bundle $E\to \C^2\setminus K$ need not extend to
a holomorphic line bundle on a bigger domain, and hence the argument in the proof
of Corollary \ref{cor:Okanotell} does not apply.

%
%   PROBLEM IN DIMENSION 2
%
\begin{problem}\label{prob:C2}
Is there a compact subset $K$ of $\C^2$ whose complement $\C^2\setminus K$ is 
an Oka domain which is not elliptic or (weakly) subelliptic?
\end{problem}

A closed unbounded set $S$ in $\C^n$ is said to be polynomially convex if it is exhausted by an 
increasing sequence of compact polynomially convex sets. 
Theorem \ref{th:complement} is a special case of the following result of Kusakabe
\cite[Theorem 1.6]{Kusakabe2020complements}; see also 
\cite[Theorem 4.2]{Kusakabe2020complements}.

%
%  COMPLEMENTS OF UNBOUNDED POLYNOMIALLY CONVEX SETS.
%
\begin{theorem}\label{th:unboundedsets}
If $S$ is a closed polynomially convex subset of $\C^n$ $(n\ge 2)$ such that
\begin{equation}\label{eq:proj}
	S\subset \left\{(z,w)\in \C^{n-2}\times \C^{2}: |w| \le c(1+|z|)\right\}
\end{equation}
for some $c>0$, then $\C^n\setminus S$ is an Oka manifold.
\end{theorem}

\begin{proof}[Sketch of proof]
Let $\pi:\C^n\to \C^{n-2}$ be the projection $\pi(z,w)=z$. 
To prove that $\C^n\setminus S$ is Oka, it suffices to show that the restricted projection 
$\pi:\C^n\setminus S\to \C^{n-2}$ satisfies POPAI; see Subsection \ref{ss:Okamap}. 
Indeed, note that condition \eqref{eq:proj} holds for all linear projections $\C^n\to \C^{n-2}$ 
sufficiently close to $\pi$. This gives finitely many linear projections $\C^n\setminus S\to \C^{n-2}$ 
enjoying POPAI whose kernels span $\C^n$, and hence the conclusion follows from Proposition \ref{prop:projections}.

In order to show that $\pi:\C^n\setminus S\to \C^{n-2}$ satisfies POPAI,  
it suffices to verify convex ellipticity; see Definition \ref{def:ellipticmap} and Theorem \ref{th:ellipticmap}. 
This means that for any compact convex set $L\subset \C^N$ 
and holomorphic map $f=(f',f''):L\to \C^n\setminus S$ (with $f':L\to\C^{n-2}$ and $f'':L\to \C^2$)
there is a fibre-dominating spray $F:L\times \C^m\to \C^n\setminus S$ over $f$ such that 
$\pi\circ F=f'$. By taking the pullback of the projection $\pi:\C^n\to\C^{n-2}$
by the base map $f':L\to\C^{n-2}$, all relevant properties are preserved and the problem gets
reduced to the one where $f$ is a holomorphic map from a neighbourhood of $L$  
to $\C^n$ such that $f(z)\in \C^n\setminus S_z$ $(z\in L)$, where $S_z$ is the fibre of $S$ over $z$.
(Here, $S$ is the new set obtained from the initial one by the pullback.) 
A spray $F$ with the desired properties can be obtained with $m=n$ as a family of 
Fatou-Bieberbach maps $\C^n\to \C^n\setminus S_z$ depending holomorphically on $z\in L$ 
by using the version of 
Theorem \ref{th:FW2020} for variable fibres $S_z$; see \cite[Remark 2.2]{ForstnericWold2020MRL}. 
\end{proof}

We mention a few applications of these results.

Gromov showed in \cite[0.5.B]{Gromov1989} that the complement $\C^n\setminus A$ 
of every closed algebraic subvariety $A$ of codimension $\ge 2$ is Oka; 
see also \cite[Proposition 5.6.10 and Sect.\ 6.4]{Forstneric2017E}.
Since every such subvariety $A$ satisfies condition \eqref{eq:proj} in some linear coordinate 
system on $\C^n$, it has a basis of closed neighbourhoods in $\C^n$ with Oka complements
\cite[Corollary 5.3]{Kusakabe2020complements}. 
The analogous result holds for tame discrete sets in $\C^n$;
see \cite[Corollaries 5.5 and 5.7]{Kusakabe2020complements}. 

Kusakabe also showed that the complement of every closed rectifiable curve $C$ 
in $\C^n$ for $n\ge 3$ is Oka; see \cite[Corollary 1.8]{Kusakabe2020complements}.
For rectifiable arcs in $\C^n$ this holds for all $n\ge 2$ since they are 
polynomially convex (see \cite[Corollary 1.8]{Kusakabe2020complements} and apply 
Theorem \ref{th:complement}). His proof for closed curves combines Theorem \ref{th:unboundedsets}
with the localization theorem (see Theorem \ref{th:localization}).
Here we give a different proof which also applies for $n=2$.
The next proposition yields examples of compact 
non-polynomially convex sets in $\C^n$ with Oka complements for any $n>1$.

%
%   COMPLEMENTS OF CURVES IN C^n, n>1
%
\begin{proposition}\label{prop:curveinC2}
If $C$ rectifiable simple closed curve in $\C^n$ $(n>1)$ then $\C^n\setminus C$ is Oka.
\end{proposition}

\begin{proof}
If $C$ is polynomially convex then $\C^n\setminus C$ is Oka by Theorem \ref{th:complement}.
Otherwise, the polynomial hull of $C$ equals $C\cup A$  where $A$ is an irreducible closed 
complex curve in $\C^n\setminus C$ with $\overline A=A\cup C$
(see Alexander \cite{Alexander1988} and 
\cite[Corollary 3.1.3 and Theorem 4.5.5]{Stout2007}). Pick a complex hyperplane 
$H\subset \C^n$ such that $C \cap H =\varnothing$ and $A\cap H\ne \varnothing$.
(Since the curve $C$ is rectifiable, a generic complex hyperplane $H$ avoids $C$.) 
%Choose coordinates $z=(z_1,\ldots,z_n)$ on $\C^n$ such that $H=\{z_n=0\}$.
Then, $C$ is holomorphically convex in the Stein domain 
$\C^n\setminus H \cong \C^{n-1} \times \C^*$. Since this domain has 
the density property (see Varolin \cite{Varolin2001} or \cite[Theorem 4.10.9]{Forstneric2017E}),
$\C^n \setminus (C\cup H)$ is Oka by Theorem \ref{th:complement}. 
Applying the same argument to $n+1$ hyperplanes $H_0=H,H_1,\cdots, H_n$
as above with $\bigcap_{i=0}^n H_i=\varnothing$ we obtain  
$\C^n\setminus C = \bigcup_{i=0}^n (\C^n\setminus C)\setminus H_i$,
so $\C^n\setminus C$ is Oka by Theorem \ref{th:localization}.
\end{proof}

By elaborating the idea in the proof of Proposition \ref{prop:curveinC2}
we now prove a considerably more general result. We recall the following 
theorem on polynomial hulls whose complex genesis is discussed in the 
monograph  \cite{Stout2007} of E.\ L.\ Stout.

%
%   ON POLYNOMIAL HULLS
%
\begin{theorem}[Theorem 3.1.1 in \cite{Stout2007}]
\label{th:Stout311}
Assume that $K$ is a compact polynomially convex set in $\C^n$ 
and $C$ is a subset of $\C^n$ contained in a compact connected set of finite length 
such that $C\cup K$ is compact. Then, $A=\wh{C\cup K}\setminus (C\cup K)$ is either empty
or a closed purely one-dimensional complex subvariety of $\C^n\setminus (C\cup K)$.
\end{theorem}

Examples in \cite{Stout2007} show that  $A$ may have infinitely many irreducible components.
We now prove the following result which generalizes Proposition \ref{prop:curveinC2}.

%
%  COMPLEMENT OF C\cup K
%
\begin{theorem}\label{th:CK}
Assume that $K$ is a compact polynomially convex set in $\C^n$, $n>1$,  
and $C$ is a subset of $\C^n$ contained in a compact connected set of finite length 
such that $C\cup K$ is compact. If the subvariety $A=\wh{C\cup K}\setminus (C\cup K)$ 
has at most finitely many irreducible components then $\C^n\setminus  (C\cup K)$ is an Oka domain. Furthermore, $C\cup K$ has a basis of compact strongly pseudoconvex
neighbourhoods with Oka complements in $\C^n$.
\end{theorem}

\begin{proof}
If $A$ is empty then $C\cup K$ is polynomially convex and the conclusion follows from 
Theorem \ref{th:complement}. Assume now that $A$ is nonempty. Pick a complex hyperplane 
$H$ in $\C^n$ which does not intersect the compact set $\wh{C\cup K}=C\cup K \cup A$. 
Choose coordinates $z=(z',z_n)$ on $\C^n$ %, with $z'=(z_1,\ldots,z_{n-1})\in \C^{n-1}$, 
such that $H=\{z_n=0\}$. Let $P=\{p_1,\ldots,p_m\}\subset A$ be a finite set 
containing a point in every irreducible component of $A$.
Pick distinct point $b_i=(b'_i,0) \in H$ for $i=1,\ldots,m$. Since $K$ is polynomially convex, 
there is a holomorphic automorphism $\phi\in \Aut(\C^n)$ which is close to the identity on a 
neighbourhood of $K$ and satisfies $\phi(b_i)=p_i$ for $i=1,\ldots,m$.
(This simple application of \cite[Theorem 4.12.1]{Forstneric2017E} is 
special case of \cite[Theorem 4.16.2]{Forstneric2017E} for finitely many points.)
It follows that the hypersurface $\phi(H)\subset \C^n$ contains the set $P$ and 
does not intersect $K$, but it may intersect $C$.
We can remove these superfluous intersections as follows. Set 
$K'=\phi^{-1}(K)$ and $C'=\phi^{-1}(C)$.
Note that $K'\cap H=\varnothing$. Choose holomorphic polynomials $f_1,\ldots,f_k$ 
on $\C^{n-1}$ whose common zero set equals $\{b'_1,\ldots,b'_m\}$ and consider 
the map $\psi:\C^{n-1} \times \C^k\to \C^{n-1}\times \C$ given by
\[
	\psi(z',t)=\Big(z', \sum_{j=1}^k t_j f_j(z')\Big) 
	\quad\text{for $z'\in\C^{n-1}$ and $t=(t_1,\ldots,t_k)\in\C^k$}.
\]
Note that $\psi$ preserves the fibres of the projection $(z',z_n)\mapsto z'$, and it is a 
submersion except on the fibres $z'=b'_i$ $(i=1,\ldots,m)$ where it equals the 
constant map $t\to (z',0)$. Since the set $C'\subset\C^n$ has finite linear measure
and it does not contain any of the points $(b'_i,0)$, the transversality argument
shows that the set of points $t\in \C^k\setminus\{0\}$ such that 
the range of the map $\psi_t=\psi(\cdotp,t):\C^{n-1}\to \C^n$ omits $C'$ 
is everywhere dense. Since $K'\cap H=\varnothing$, taking $t$ in this set and close enough 
to $0\in\C^{k}$ also ensures that $\psi_t(\C^{n-1})\cap K'=\varnothing$.
For such $t$, the holomorphic automorphism $\Phi\in \Aut(\C^n)$ given by 
\[
	\Phi(z',z_n) = \phi \Big(z', z_n + \sum_{j=1}^k t_j f_j(z')\Big),\quad \ z\in\C^{n}
\]
clearly satisfies $P\subset \Phi(H)$ and $\Phi(H)\cap (C\cup K)=\varnothing$.

By changing the coordinates on $\C^n$ using $\Phi$, this reduces the proof
of the theorem to the case when the hyperplane $H=\{z_n=0\}$ does not intersect $C\cup K$
but it intersects every irreducible component of $A=\wh{C\cup K}\setminus (C\cup K)$.
It follows that $C\cup K$ is holomorphically convex in the Stein domain 
$\C^n\setminus H = \C^{n-1}\times \C^*$. Since this domain has the density property
by Varolin \cite{Varolin2001}, Theorem \ref{th:complement} shows that 
$(\C^{n-1}\times\C^*) \setminus (C\cup K)$ is an Oka domain.
Applying the same argument to $n+1$ hyperplanes $H_0=H,H_1,\cdots, H_n$ in $\C^n$
close enough to $H=\{z_n=0\}$ with $\bigcap_{i=0}^n H_i=\varnothing$ 
(see Corollary \ref{cor:PC3}) we obtain  
\[
	\C^n\setminus (C\cup K) = \bigcup_{i=0}^n (\C^n\setminus (C\cup K))\setminus H_i,
\]
so $\C^n\setminus (C\cup K)$ is Oka by Theorem \ref{th:localization}.
Finally, $C\cup K$ clearly admits compact strongly pseudoconvex 
neighbourhoods which are holomorphically convex in $\C^n\setminus H$,
and hence also in $\C^n\setminus H_i$ for each $i=1,\ldots,m$ provided 
the hyperplanes $H_i$ are chosen close enough to $H=H_0$. This shows
that the complement of every such compact domain in $\C^n$ is Oka.
\end{proof}

A theorem of M.\ Lawrence \cite{Lawrence1995} (see also \cite[Theorem 4.7.1]{Stout2007}) 
says that if $C\subset \C^n$ is a compact set of finite length and $A$ is a  bounded
closed purely one-dimensional complex subvariety of $\C^n\setminus C$, 
then the number of irreducible components of $A$ does not exceed the rank of the first 
Chech cohomology group $\check H^1(C,\Z)$
(which is the number of simple closed curves contained in $C$). Together with Theorem \ref{th:CK}
this gives the following corollary. % generalizing Proposition \ref{prop:curveinC2}. 

%
%   COMPLEMENTS OF SETS OF FINITE LENGTH WITH FINITELY GENERATED H^1
%
\begin{corollary}
Let $C$ be a compact subset of $\C^n$, $n>1$, which is contained in a compact connected set of finite length.
If the group $\check H^1(C,\Z)$ has finite rank then $\C^n\setminus C$ is Oka. 
\end{corollary}

In a recent work \cite{ForstnericWold2022Oka}, E.\ F.\ Wold and the author proved 
that complements of most closed convex sets in $\C^n$ for $n>1$ are Oka. 
In particular, the following holds.  

%
%  COMPLEMENTS OF CLOSED CONVEX SETS ARE OKA
%
\begin{theorem}[Theorem 1.8 in \cite{ForstnericWold2022Oka}] \label{th:convexnoline}
If $E$ is a closed convex set in $\C^n$ for $n>1$ which does not contain any affine real line, 
then $\mathbb C^n\setminus E$ is an Oka domain. 
\end{theorem}

This result is new for unbounded convex sets; 
for bounded ones it follows from  Theorem \ref{th:complement}. It provides many 
model concave Oka domains $\Omega\subset\C^n$ $(n>1)$ of the form 
\begin{equation}\label{eq:model}
	\Omega=\{z=(z',z_n) \in \C^n: \Im z_n<\phi(z',\Re z_n)\},
\end{equation}
where $\phi\ge 0$ is a convex function, which are only slightly bigger than a halfspace, 
the latter being neither Oka nor hyperbolic. This gives examples of splitting
$\C^n$ for $n>1$ by a real hypersurface into a pair of a (convex) Kobayashi 
hyperbolic domain and a (concave) Oka domain which are close to a halfspace.

Theorem \ref{th:convexnoline} reduces to the following result 
(see \cite[Theorem 1.1]{ForstnericWold2022Oka})   
by combining complex analysis with convex geometry and projective geometry.
We consider $\C^n$ as an affine chart in the projective space $\CP^n$.
Given a subset $E\subset\C^n$ we denote by $\overline E$ its closure in $\CP^n$.

%
%   MAIN THEOREM
%
\begin{theorem}\label{th:mainFW2022}
If $E$ is a closed subset of $\C^n$ for $n>1$ and $\Lambda\subset \CP^n$ is 
a complex hyperplane such that $\overline E\cap \Lambda=\varnothing$ and $\overline E$ is 
polynomially convex in $\CP^n\setminus\Lambda \cong\C^n$, 
then $\C^n\setminus E$ is Oka. 
\end{theorem}

For a set $E$ as in Theorem \ref{th:convexnoline} it is shown in 
\cite{ForstnericWold2022Oka} that its projective closure
$K=\overline E\subset \CP^n$ is a compact polynomially convex set in the affine chart 
$\CP^n\setminus \Lambda\cong\C^n$ for some complex hyperplane $\Lambda\subset \CP^n$
with $K\cap\Lambda=\varnothing$, so Theorem \ref{th:mainFW2022} applies. 
In fact, $\CP^n\setminus K$ is the union of a connected family of complex hyperplanes 
in $\CP^n$, so Corollary \ref{cor:PC2} shows that $K$ is polynomially convex in the 
complement of each of them. 

Theorem \ref{th:mainFW2022} easily reduces to showing that for any 
compact polynomially convex set $K$ in $\C^n$
% (applied in the affine chart $\CP^n\setminus \Lambda\cong \C^n$) 
and affine complex hyperplane $H\subset\C^n$
the domain $\Omega=\C^n\setminus (H\cup K)$ is Oka; see 
\cite[Corollary 3.2]{ForstnericWold2022Oka}. This is proved 
% in a similar way as Theorem \ref{th:complement} 
by verifying condition \CEll1. 
Given a holomorphic map $f:L\to \Omega$ from a compact convex set $L$ in some $\C^N$, 
we find a dominating holomorphic spray $F:L\times\C^n\to \Omega$ 
such that for every $x\in L$ we have $F(x,0)=f(x)$
and the map $F(x,\cdotp):\C^n\to\Omega$
is injective, so its image is a Fatou--Bieberbach domain 
(see \cite[Theorem 2.3]{ForstnericWold2022Oka}). Thus, $\Omega$ satisfies condition 
\CEll1 (see Definition \ref{def:elliptic}), and hence is Oka by Theorem \ref{th:K2018}.
In the proof, we use the result of Varolin \cite{Varolin2001} 
that the Lie algebra of holomorphic vector fields on $\C^n$ vanishing on a complex 
hyperplane $H\subset \C^n$ enjoys the density property.

\begin{remark}
If a closed subset $E\subset\C^n$ satisfies the hypotheses of Theorem 
\ref{th:mainFW2022}, then $E$ has a basis of closed neighbourhoods
$E'\supset E$ satisfying the same condition (since a compact 
polynomially convex set has a basis of compact polynomially convex neighbourhoods).
Hence, $\C^n\setminus E'$ is Oka for any such $E'$. This yields some examples in the
literature that were previously obtained by different arguments. For example, 
if $E$ is a closed tame discrete set in $\C^n$ $(n>1)$ then, after applying an
automorphism of $\C^n$, we may assume that $E$ lies in a complex line $L\subset\C^n$,
and hence $\overline E=E\cup\{p\}\subset\CP^n$ where $p$ is the point at infinity
determined by $L$. By taking a hyperplane $\Lambda\subset\CP^n$ not interesting
$\overline E$, the set $\overline E$ is polynomially convex in 
$\CP^n\setminus \Lambda\cong\C^n$. Hence, the above argument and 
Theorem \ref{th:mainFW2022} show the following.

\begin{corollary}
Every tame discrete set $E\subset \C^n$ for $n>1$ admits a basis of closed
neighbourhoods whose complements are Oka.
\end{corollary}

A different proof was given by Kusakabe \cite[Corollaries 5.5 and 5.7]{Kusakabe2020complements}.
\end{remark}

%
%   COMPLEMENTS OF TOTALLY REAL SUBSPACES
%
The condition in Theorem \ref{th:convexnoline} that the set $E$ does not contain 
any affine real line is not necessary. 
The following theorem combines \cite[Proposition 4.9]{ForstnericWold2022Oka} 
(for the case $(k,n)=(1,2)$) with the result of Kusakabe 
\cite[Corollary 1.7]{Kusakabe2020complements} which covers the other cases.

\begin{theorem}\label{th:C2minusR}
If $E\cong\R^k$ is a totally real subspace of $\C^n$, where $1\le k \le n$, $n\ge 2$,
and $(k,n)\notin \{(2,2),(3,3)\}$, then $\C^n\setminus E$ is an Oka domain. 
\end{theorem}

%
%  BDD & FF 2022
%
Drinovec Drnov\v sek and Forstneri\v c showed in \cite{DrinovecForstneric2023}
that for many model concave domains $\Omega\subset \C^n$ as in \eqref{eq:model}, 
the Oka property with approximation holds for proper
holomorphic maps $X\to\C^n$ with $\dim X<n$ whose images lie 
in $\Omega$. They introduced the following notion.

%
%  BOUNDED CONVEX EXHAUSTION HULLS
%
\begin{definition} \label{def:BCEH}
A closed convex set $E$ in a real or a complex Euclidean space $V$ has 
{\em bounded convex exhaustion hulls} (BCEH) if for every compact convex set $K$ in $V$
\[
	\text{the set}\ \ h(E,K) = \Conv(E\cup K)\setminus E \ \ \text{is bounded.}
\]
\end{definition}

Here, $\Conv$ denotes the convex hull. The following is 
\cite[Theorem 1.3]{DrinovecForstneric2023}. 

%
%   THE MAIN THEOREM
%
\begin{theorem}\label{th:DF2022}
Let $E$ be an unbounded closed convex set in $\C^n$ $(n>1)$ with 
bounded convex exhaustion hulls. 
Given a Stein manifold $X$ with $\dim X<n$, a compact $\Ocal(X)$-convex set $K$ in $X$, 
and a holomorphic map $f_0:K\to\C^n$ with $f_0(bK)\subset \Omega=\C^n\setminus E$, 
we can approximate $f_0$ uniformly on $K$ by proper holomorphic maps $f:X\to \C^n$ 
satisfying $f(X\setminus \mathring K) \subset \Omega$. The map $f$ can be chosen an 
embedding if $2\dim X < n$ and an immersion if $2\dim X \le n$. 
\end{theorem}

The analogous results for compact
convex sets $E\subset\C^n$ was proved beforehand by Forst\-ne\-ri\v c and Ritter
\cite{ForstnericRitter2014}, and in this case the BCEH condition trivially holds.

Drinovec Drnov\v sek and Forstneri\v c proved 
(see \cite[Proposition 3.4]{DrinovecForstneric2023})
that an unbounded closed convex set $E$ in $\C^n$ satisfying BCEH is in 
some affine coordinates on $\C^n$ an epigraph
\begin{equation}\label{eq:EinCn}
	E =E_\phi=\{z=(z',z_n)\in\C^n: \Im z_n \ge \phi(z',\Re z_n)\}
\end{equation} 
of a convex function $\phi:\C^{n-1}\times\R\to\R_+$ with at least linear growth. 
Furthermore, they showed that any such function $\phi$ 
can be approximated uniformly on compacts 
by functions $\psi\le \phi$ of the same kind whose epigraphs $E_\psi$ 
have BCEH. This gives the following corollary.

%
%   EXTENSION OF THE MAIN THEOREM
%
\begin{corollary}[Corollary 1.4 in \cite{DrinovecForstneric2023}] \label{cor:main}
The conclusion of Theorem \ref{th:DF2022} holds for any convex epigraph $E_\phi$ 
of the form \eqref{eq:EinCn} such that $\phi\ge 0$ and the set $\{\phi=0\}$ 
is nonempty and compact.
\end{corollary}

We pose the following question reminiscent of the classical Levi problem.

%
%   GEOMETRIC CHARACTERIZATION OF OKA COMPLEMENTS
%
\begin{problem}\label{prob:geometric}
Let $K$ be a compact domain with smooth boundary in $\C^n$ for $n>1$. 
%Is there a characterization of $\C^n\setminus K$ being an Oka manifold in terms of 
%geometric properties of $bK$? In particular:
\begin{enumerate}[\rm (a)]
\item Assuming that $\C^n\setminus K$ is Oka, must $K$ be pseudoconvex?
\item Assuming that $K$ is (strongly) pseudoconvex, is
$\C^n\setminus K$ an Oka domain?
\item 
Is every strongly pseudoconcave domain $\Omega\subset\C^n$ 
of the form \eqref{eq:model} an Oka domain?
\end{enumerate}
\end{problem}

Parts (a) and (b) of the above problem are also of interest for domains in $\CP^n$.
Note that if $K$ is a smoothly bounded compact domain in a complex manifold $Y$ such that 
$Y\setminus K$ is Oka, then $K$ cannot have a strongly pseudoconcave boundary point, 
since this would yield a nonconstant bounded plurisubharmonic function on $Y \setminus K$.
This shows that the answer to (a) is affirmative in dimension two. 
%A good test case for problem (b) may be to understand whether complements
%of thin tubes around the standard totally real torus in $\C^2$ are Oka. 
I expect that the answer to (b) is negative in general.

%
%   THE COMPLEMENT OF HARTOGS FIGURE
%
\begin{example}\label{ex:Hartogs}
Denote the coordinates on $\C^n$ by $z=(z_1,z')$ with $z'=(z_2,\ldots,z_n)$.
Given a number $0<\delta<1$ we consider the closed Hartogs figure
\[
	H=\{(z_1,z'):|z_1|\le \delta,\ |z'|	\le 1\} \cup\{(z_1,z'):|z_1|\le 1,\ 1-\delta\le |z'|\le 1\}.
\]
We claim that $\C^n\setminus H$ fails to be Oka. To see this, 
let $h(z_1)$ be a bounded subharmonic function on $|z_1|>\delta$ which vanishes 
on $|z_1|\ge 1$ and is positive for $|z_1|$ close to $\delta$; an explicit example
is $h(z_1)=\max\{0,1/|z_1|^2-1\}$. Let $\rho$ be the 
function on $\C^n\setminus H$ which equals $\rho(z_1,z')=h(z_1)$
on $\{(z_1,z'):\delta<|z_1|\le 1,\ |z'|<1-\delta\}$ and equals zero on the complement of
the closed unit polydisc. Then $\rho$ is a nonconstant bounded plurisubharmonic function
on $\C^n\setminus H$, so $\C^n\setminus H$ fails to be Liouville, and hence it also fails
to be Oka. 
\end{example}

%
%   THE COMPLEMENT OF HARTOGS TRIANGLE
%
\begin{problem}\label{prob:Hartogs}
Let $H$ be the closed Hartogs triangle
\[
	H=\{(z_1,z_2)\in\C^2: 0\le |z_1| \le |z_2|\le 1\}.
\]
Is $\C^2\setminus H$ an Oka domain? Note that the argument in Example \ref{ex:Hartogs}
does not apply in this case. On the other hand, if 
$K$ is a small closed smoothly bounded neighbourhood of $H$ 
then $K$ cannot be pseudoconvex, so $\C^2\setminus K$ fails to be Oka.
\end{problem}

The following example suggests that there is no reasonable geometric characterization of 
Oka domains in $\C^n$ with unbounded complements.

%
%   A VOLUME HYPERBOLIC DOMAIN WHOSE COMPLEMENTS
%
\begin{example}\label{ex:twistedline}
For any $n>1$ there is an unbounded, closed, connected, strongly pseudoconvex domain $E\subset \C^n$ 
with arbitrarily small volume such that $\C^n\setminus E$ fails to be Oka.

To see this, recall that Rosay and Rudin \cite[Theorem 4.5]{RosayRudin1988} constructed for every $n>1$ 
a closed discrete set $A \subset \C^n$ whose complement $\C^n\setminus A$ is 
volume hyperbolic; in particular, any holomorphic map $\C^n\to  \C^n\setminus A$ has rank $<n$ 
at every point. Choose a proper smooth embedding $g:\R\hra \C^n$ whose image contains $A$. 
Then, the domain $\C^n\setminus g(\R)\subset\C^n\setminus A$ is volume hyperbolic and hence 
is not Oka. Let $v_1,\ldots,v_{2n-1}:\R\to\C^{n}$ be smooth maps such that for each $t\in\R$ 
the vectors $v_1(t),\ldots, v_{2n-1}(t)$ form an orthonormal set and they are 
orthogonal to $\dot g(t)$. Consider the map $G:\R^{2n}\to \C^n$ given by 
\[
	G(t,x_1,\ldots,x_{2n-1})=g(t)+\sum_{i=1}^{2n-1} x_i v_i(t).
\]
If $\epsilon:\R\to (0,1)$ is a smooth positive function which decreases sufficiently 
fast as $t\to \pm\infty$ then $G$ maps the tube
$T_\epsilon=\{(t,x)\in\R^{2n}: |x|\le \epsilon(t)\}$ diffeomorphically onto 
a strongly pseudoconvex tube $E$ around $g(\R)$ having arbitrarily small volume. 
The complement $\C^n\setminus E$ is a strongly pseudoconcave domain which 
fails to be Oka. 
%We leave the details to an interested reader.
\end{example}

The embedded real line in Example \ref{ex:twistedline} is necessarily very twisted, 
and its projective closure
may well contain the entire hyperplane at infinity. On the other hand, Theorem \ref{th:C2minusR}
for the case $k=1<n$ suggests that properly embedded real lines which behave sufficiently nicely at
infinity may have Oka complements. We introduce the following notion of tameness for 
embedded real lines, which extends the one of Rosay and Rudin \cite{RosayRudin1988} 
for discrete sets.

%
%   TAME LINES
%
\begin{definition}\label{def:tameline}
Let $n\ge 2$. A properly embedded real line $f:\R\hra\C^n$ or halfline $f:\R_+\hra \C^n$ is {\em tame}
if there is an automorphism $\Phi\in\Aut(\C^n)$ such that the projective closure 
$\overline{\Phi\circ f(\R)}\subset\CP^n$ (or $\overline{\Phi\circ f(\R_+)}$)
is a rectifiable arc or a rectifiable closed Jordan curve in $\CP^n$.
\end{definition}

\begin{remark}
It is easily seen that the closure of a tame embedded line intersects the hyperplane 
at infinity in precisely one point at every end. 
This definition of tameness is stronger than the one for closed countable sets, 
used in Theorem \ref{th:KusakabePAMS}, or the one for closed complex subvarieties of 
codimension $\ge 2$ in $\C^n$ \cite[Definition 4.11.3]{Forstneric2017E}. 
In those definitions one asks that, in some holomorphic coordinates on $\C^n$, 
the closure of the set in $\CP^n$ does not contain the hyperplane at infinity. 
On the other hand, the original definition of a tame discrete set in $\C^n$, 
given by Rosay and Rudin \cite{RosayRudin1988}, is equivalent to asking that 
the set can be mapped into a complex line by an automorphism of $\C^n$, so its 
closure in $\CP^n$ has a single point at infinity.
\end{remark}

The following result generalizes Theorem \ref{th:C2minusR} in the case $k=1<n$.
It is new for $n=2$, while for $n\ge 3$ it is a consequence of Theorem 
\ref{th:unboundedsets}, and in this case it holds under the weaker tameness condition 
in that result.

%
%   THE COMPLEMENT OF A TAME REAL LINE IS OKA
%
\begin{theorem}\label{th:tamelineOka}
The complement $\C^n\setminus E$ of a tame embedded line $\R\cong E\subset \C^n$ 
for $n>1$ is Oka. Furthermore, the complement of any closed subset of such 
a set $E$ is Oka.
\end{theorem}

\begin{proof}
We follow the idea of proof of \cite[Proposition 4.9]{ForstnericWold2022Oka}. 
We may assume that the tameness condition in Definition \ref{def:tameline} holds 
with $\Phi=\Id$. Write $\CP^n=\C^n\cup H$, 
where $H$ is the hyperplane at infinity. By dimension reasons, 
there is a complex hyperplane $\Lambda\subset\CP^n$ which does not intersect the 
rectifiable curve $C=\overline E\subset\CP^n$.

If $C$ is polynomially convex in $X=\CP^n\setminus\Lambda\cong\C^n$, 
the result follows from Theorem \ref{th:mainFW2022}. 
This holds in particular if $C$ is an arc. The same holds for any closed subset of $C$. 

Assume now that $C$ is not polynomially convex in $X$. By Theorem \ref{th:Stout311}
its polynomial hull in $X$ equals $C\cup A$, where $A$ is a closed irreducible one-dimensional 
complex subvariety of $X\setminus C$ with $\overline A=A\cup C$. 
Then, $\C^n\setminus \overline A$ is an Oka domain by Theorem \ref{th:mainFW2022}. 
(Here, $\C^n=\CP^n\setminus H$.)
Choose a complex hyperplane $\Lambda'\subset\CP^n$ which intersects $A$ but avoids $C$; 
such a hyperplane exists by dimension reasons since $C$ is a rectifiable curve. Then, 
the polynomial hull of $C$ in $X'=\CP^n\setminus \Lambda'\cong\C^n$ 
does contain any point of $A$. If $C$ is polynomially convex in $X'$, then 
$\C^n\setminus E$ is Oka by Theorem \ref{th:mainFW2022} and we are done. 
Otherwise, its polynomial hull in $X'$ equals $C\cup A'$, where $A'$ is a closed irreducible 
one-dimensional complex subvariety of $X'\setminus C$ and $\overline{A'}=A'\cup C$.
(See Theorem  \ref{th:Stout311}.) In this case, $A\cup A'\cup C$ is a closed 
complex curve in $\CP^n$ by the boundary uniqueness theorem 
(see \cite[Proposition 1, p.\ 258]{Chirka1989}). By the same argument as above 
we infer that $\C^n\setminus \overline{A'}$ is an Oka domain. Note that 
\[
	\C^n\setminus (E\cup (A\cap A')) = 
	(\C^n\setminus \overline A)\cup (\C^n\setminus \overline{A'})
\] 
and both Oka domains on the right hand side are Zariski open in 
$\C^n\setminus (E\cup (A\cap A'))$, so their union is Oka by 
Theorem \ref{th:localization}. If $A\cap A'=\varnothing$, we are done. 
Otherwise, there is a complex hyperplane $\Sigma\subset\CP^n$ 
passing through a point of $A\cap A'$ and avoiding $C$. 
In this case, $C$ is polynomially convex in $\CP^n\setminus \Sigma$, 
and hence $\C^n\setminus E$ is Oka by Theorem \ref{th:mainFW2022}. 
\end{proof}

\section{Oka domains in projective spaces}\label{sec:CPn}
In this section we exhibit some new examples of Oka domains in complex projective spaces.
We begin with the following result.

%
%  COMPLEMENT OF A HOLO CONVEX SET
%
\begin{theorem}\label{th:K}
If $\Lambda$ is a closed complex hypersurface in $\CP^n$ $(n>1)$ such that 
the manifold $\Omega=\CP^n\setminus \Lambda$ has the density property
(see Definition \ref{def:density}), then for any compact $\Oscr(\Omega)$-convex set 
$K\subset \Omega$ the complement $\CP^n\setminus K$ is an Oka domain. 
In particular, the hypersurface $\Lambda$ has a basis of open Oka neighbourhoods 
in $\CP^n$.
\end{theorem}

\begin{proof} %[Proof of Theorem \ref{th:K}]
The hypersurface $\Lambda$ is given in homogeneous coordinates
by the zero set $\{P=0\}$ of a homogeneous polynomial $P$ of degree $k=\deg\Lambda$. 
With respect to the $k$-th Veronese embedding $\CP^n\hra \CP^{N}$ 
%(where $N={{n+k}\choose {k}}-1$), 
whose components are all homogeneous monomials of degree $k$ in $n+1$ variables, 
$\Lambda$ is the intersection of the image of $\CP^n$ with a hyperplane $H\subset \CP^N$, so 
$\CP^n\setminus \Lambda$ is a closed affine (hence Stein) submanifold of 
$\CP^N\setminus H=\C^N$. 

The projective linear group $G=PGL_n(\C)$ acts transitively 
on $\CP^n$ by holomorphic automorphisms.
Hence, there are finitely many maps $A_0=\Id,A_1,\ldots, A_m\in G$ in any given neighbourhood of
the identity map such that the hypersurfaces $\Lambda_i=A_i(\Lambda)\subset\CP^n$ satisfy 
$\bigcap_{i=0}^m \Lambda_i=\varnothing$. For each $i=1,\ldots, m$ the domain 
$
	\Omega_i=\CP^n\setminus \Lambda_i=A_i(\CP^n\setminus \Lambda)
$
is biholomorphic to $\Omega_0=\CP^n\setminus \Lambda$, so it has the density property.
Assuming that $A_i$ is close enough to the identity map, there is a path $A_{i,t}\in G$
$(t\in [0,1])$ connecting $A_{i,0}=A_i$ to $A_{i,1}=\Id$ such that for every $t\in [0,1]$ the hypersurface 
$	
	\Lambda_{i,t}=A_{i,t}(\Lambda) = \{P\circ A_{i,t}^{-1}=0\} 
$
avoids the given compact set $K\subset \CP^n\setminus\Lambda$. 
Note that $\Lambda_{i,0}=\Lambda_i$ and $\Lambda_{i,1}=\Lambda$. Since $K$ is 
$\Oscr(\Omega_0)$-convex, Corollary \ref{cor:PC3} shows that $K$ is 
holomorphically convex in $\Omega_{i}=\CP^n\setminus \Lambda_{i}$ for every 
$i=1,\ldots,m$. Since $\Omega_i$ has the density property, Theorem \ref{th:complement} 
implies that $\Omega_i\setminus K$ is Oka for $i=0,\ldots,m$. 
Since $\Omega_i\setminus K=(\CP^n\setminus K)\setminus \Lambda_i$ is Zariski open 
in $\CP^n\setminus K$ and $\CP^n\setminus K = \bigcup_{i=0}^m \Omega_i\setminus K$, 
Theorem \ref{th:localization} shows that $\CP^n\setminus K$ is Oka.
\end{proof}

\begin{corollary}\label{cor:complementCPn}
If $K$ is a compact polynomially convex set in $\C^n$, $n>1$, then $\CP^n\setminus K$ is Oka. \end{corollary}

In light of Theorem \ref{th:K} it is natural to ask the following question.

%
%   WHICH COMPLEMENTS OF HYPERSURFACES HAVE DP?
%
\begin{problem}
\begin{enumerate}[\rm (a)]
\item For which complex hypersurfaces $\Lambda\subset\CP^n$ is $\CP^n\setminus \Lambda$ 
an Oka manifold? 
\item For which complex hypersurfaces $\Lambda\subset\CP^n$ does 
$\CP^n\setminus \Lambda$ have the density property?
\end{enumerate}
\end{problem}

\begin{example}\label{ex:unionofhyperplanes}
If $\Lambda_1,\ldots, \Lambda_k\subset \CP^n$ $(n>1,\ 1\le k\le n+1)$ 
are hyperplanes in general position then $\Omega=\CP^n\setminus \bigcup_{i=1}^k \Lambda_k$ 
is isomorphic to $\C^{n-k+1}\times (\C^*)^{k-1}$.
If $k\le n$ then $\Omega$ has the density property (see Varolin \cite[p.\ 136]{Varolin2001}). 
If $k=n+1$ then $\Omega$ isomorphic to $(\C^*)^n$ which is Oka but is not known to have 
the density property. The complement of more than $n+1$ hyperplanes in $\CP^n$ 
fails to be Oka (see Hanysz \cite[Theorem 3.1]{Hanysz2014}).
\end{example}

We now show that Theorem \ref{th:K} holds if $\Lambda$ is a hyperquadric.  
It was shown by Kusakabe \cite[Corollary 4.9 (1)]{Kusakabe2021IUMJ}
that the complement of a smooth hyperquadric in $\CP^n$ is Oka.

%
%	COMPLEMENTS OF QUADRIC
%
\begin{theorem} \label{th:quadric}
If $\Lambda$ is a quadric hypersurface in $\CP^n$ $(n>1)$ and $K$ is a compact 
holomorphically convex set in the Stein domain $\CP^n\setminus \Lambda$, 
then $\CP^n\setminus K$ is an Oka manifold.

In particular, if $\RP^n\subset \CP^n$ is the standard embedding of the real projective space in the 
complex projective space, then $\CP^n\setminus \RP^n$ is Oka for any $n>1$.
\end{theorem}

\begin{proof}
A singular hyperquadric in $\CP^n$ is a union of two hyperplanes. Its complement is 
isomorphic to $\C^{n-1}\times \C^*$, which has the density property  \cite{Varolin2001}, 
so the conclusion follows from Theorem \ref{th:K}. 
Assume now that $\Lambda$ is smooth. Lacking a reference for the density property 
of $\CP^n\setminus \Lambda$, we proceed as follows.
(I owe this idea to Stefan Nemirovski.) There are homogeneous coordinates 
on $\CP^n$ in which $\Lambda=\{z_0^2 + z_1^2 + \cdots +z_n^2=0\}$.
% $\Lambda$ is given by the equation $z_0^2 + z_1^2 + \cdots +z_n^2=0$. 
The restriction of the projection $\pi:\C^{n+1}\setminus \{0\}\to\CP^n$ 
to the affine quadric %(the complex $n$-sphere) given by 
$
	X=\{z_0^2 + \cdots +z_n^2=1\}\subset \C^{n+1}\setminus \{0\}
$
is a two-sheeted covering map $\pi|_X:X\to \CP^n\setminus \Lambda$. 
The quadric $X$ has the density property according to Kaliman and Kutzschebauch
\cite{KalimanKutzschebauch2008MZ}.
(Indeed, $X$ is linearly equivalent to the Danielewski hypersurface
\[
	\Sigma = \bigl\{(u,v,z_2,\ldots,z_n)\in\C^{n+1}: uv=P(z)=z_2^2 + \cdots +z_n^2 -1\bigr\},
\]
with the polynomial $P$ having smooth reduced zero fibre.)
Since $K$ is holomorphically convex in 
$\CP^n\setminus \Lambda$, its preimage $L=(\pi|_X)^{-1}(K)$ is $\Oscr(X)$-convex. 
By Theorem \ref{th:complement} the complement $X\setminus L$ is an Oka manifold. 
Since $\pi:X\setminus L \to \CP^n\setminus (K\cup \Lambda)$ is a covering map, the domain 
$\CP^n\setminus (K\cup \Lambda)$ is Oka by \cite[Proposition 5.6.3]{Forstneric2017E}. 
It remains to apply the argument in the proof of Theorem \ref{th:K}, 
varying $\Lambda$ among nearby quadrics and using Corollary \ref{cor:PC3} 
and Theorem \ref{th:localization}. 

The preimage $\pi^{-1}(\RP^n)$ of the real projective space is the sphere $S^n=X\cap\R^{n+1}$ of 
real points in $X$, which is holomorphically convex in $X$. This gives the last statement.
\end{proof}

%
%   ABOUT CUBICS IN CP^2
%
\begin{remark}[Complements of cubics]
It was shown by Kusakabe \cite[Corollary 4.9 (2)]{Kusakabe2021IUMJ} that the complement of  
every irreducible singular cubic in $\CP^2$ is Oka, but it is not clear whether the conclusion
of Theorem \ref{th:quadric} holds in this case. There are two such cubics up to automorphisms
of $\CP^2$, given respectively by $y^2 z = x^3$ and $y^2 z = x^3 + x^2 z$. 
It is not known whether the complement of the smooth cubic $x^3+y^3+z^3=0$ in $\CP^2$ 
is Oka, although it is dominable by $\C^2$; see the discussion in Hanysz 
\cite[Sect.\ 4]{Hanysz2014}.
\end{remark}

Another family of examples of Oka domains in $\CP^n$ is given by the following proposition.

%of Oka domains in projective spaces are complements of compact rectifiable curves.

%
%   COMPLEMENTS OF CURVES IN PROJECTIVE SPACES
%
\begin{proposition}
\label{prop:CPn-curve}
If $C$ is a compact rectifiable Jordan arc or a rectifiable simple closed Jordan curve in $\CP^n$ for $n>1$, 
then $\CP^n\setminus C$ is an Oka manifold.
\end{proposition}

\begin{proof}
Since $C$ has finite length, we have $C\cap \Lambda=\varnothing $ for 
almost every projective hyperplane $\Lambda\subset \CP^n$.
Fix such $\Lambda$ and let $X_\Lambda = \CP^n\setminus \Lambda\cong\C^n$.
By Proposition \ref{prop:curveinC2}, $X_\Lambda \setminus C$ is Oka.
Note that $X_\Lambda \setminus C$ is a Zariski open domain in $\CP^n\setminus C$.
Clearly we can cover $\CP^n\setminus C$ by finitely many Zariski open sets of this form,
and hence $\CP^n\setminus C$ is Oka by Theorem \ref{th:localization}.
\end{proof}

A similar argument gives the following result.

\begin{theorem} \label{th:CK-CPn}
If $C\cup K\subset\C^n$ is as in Theorem \ref{th:CK} then $\CP^n\setminus (C\cup K)$ is Oka.
\end{theorem}

\begin{comment}
A compact orientable submanifold without boundary of real dimension at least $n$ in $\C^n$ 
is not polynomially convex (see Browder \cite{Browder1961}), and its polynomial hull may have nonempty interior.
An example is the standard $n$-torus in $\C^n$, product of the unit circles in coordinate
complex lines, whose hull is the standard polydisc.
On the other hand, a generic compact totally real submanifold of dimension $<n$ 
is polynomially convex by a result of L\o w and Wold \cite{LowWold2009}.
It is reasonable to expect that the following problem has affirmative answer.

\begin{problem}
Let $n\ge 3$. Does a generic closed totally real submanifold $M$ in $\CP^n$ with $\dim M<n$ 
have Oka complement $\CP^n\setminus M$?
\end{problem}
\end{comment}

%%%%%%%%%%%%%%%%
%
%  ALGEBRAIC OKA THEORY
%
%%%%%%%%%%%%%%%%
\section{Algebraic Oka theory}\label{sec:aOka}

The algebraic Oka theory concerns Oka properties of regular algebraic maps from affine 
algebraic manifolds (the algebraic analogues of Stein manifolds) to algebraic manifolds. 
Not surprisingly, the situation is much more rigid than in the holomorphic case, 
and there are many examples where the Oka principle holds for holomorphic maps 
but it fails for algebraic maps. Indeed, we shall see that 
no compact algebraic manifold is algebraically Oka,
and we do not know a single example of a noncompact algebraically Oka manifold.
Nevertheless, certain weaker Oka properties are still of interest in the algebraic case.

% 

%
%   APPROXIMATION OF HOLOMORPHIC MAPS BY ALGEBRAIC  MAPS
%
\subsection{Algebraically subelliptic manifolds and algebraic approximation}
We have seen that holomorphic approximation plays a crucial role in Oka theory. 
Likewise, the problem of approximating holomorphic maps by algebraic maps is of 
central importance. Algebraic approximants in general do not exist 
even for maps between very simple affine algebraic manifolds.
For instance, there are no nontrivial algebraic morphisms $\C\to \C\setminus\{0\}$.
A major role in these problems play the following 
classes of algebraic manifolds which were discussed by 
Gromov \cite{Gromov1989}; see also \cite[Definition 2.1]{Forstneric2006AJM} 
and \cite[Definition 5.6.13 (e)]{Forstneric2017E}.

%
%   DEF: ALGEBRAIC SUBELLIPTICITY
%
\begin{definition}\label{def:Asubelliptic}
Let $Y$ be an algebraic manifold.
\begin{enumerate}[\rm (a)] 
\item $Y$ is {\em algebraically elliptic} if it admits a dominating  algebraic 
spray $F:E\to Y$ defined on the total space of an algebraic vector bundle $E\to Y$ 
(see \eqref{eq:sprayonY}).
\item $Y$ is {\em algebraically subelliptic} if it admits a finite family of 
algebraic sprays $F_j:E_j\to Y$ from algebraic vector bundles $E_j\to Y$ 
$(j=1,\ldots,m)$ such that 
\[
	\sum_{j=1}^m dF_j(0_y)(E_{j,y}) = T_{y}Y\ \ \text{for every $y\in Y$}.
\]
%\item $Y$ is {\em locally algebraically (sub-)elliptic} if it is covered by Zariski open domains which are algebraically (sub-)elliptic.
\item $Y$ is {\em locally algebraically subelliptic} if every point $y\in Y$ has a 
Zariski neighbourhood $U\subset Y$ and a finite dominating family of algebraic
sprays on $U$ with values in $Y$.
\item $Y$ is {\em weakly algebraically subelliptic}
if for every point $a\in Y$, the tangent space $T_a Y$ is spanned by vectors $v$ 
such that there is an affine Zariski open neighbourhood $U$ of $a$ in $Y$ 
and a regular map $f:U\times \C\to Y$ with $f (y, 0) =y$ for all $y\in U$ and 
$\frac{d}{dt}\Big|_{t=0} f(a,t) = v$.
\item $Y$ satisfies condition $a\Ell1$ if the condition in 
Definition \ref{def:elliptic} (b) holds for algebraic maps $X\to Y$ from affine 
algebraic manifolds.
\end{enumerate}
\end{definition}

Examples and properties of such manifolds can be found in 
\cite[Section 6.4]{Forstneric2017E} and elsewhere in the cited book. 
Any one of these conditions implies that the manifold is Oka. 
It turns out that all these properties are pairwise equivalent.

\begin{theorem}\label{th:wae}
For an algebraic manifold $Y$ the following conditions are equivalent:
\begin{enumerate}[\rm (a)]
\item $Y$ is algebraically elliptic.
\item $Y$ is algebraically subelliptic.
\item $Y$ is locally algebraically subelliptic.
\item $Y$ is weakly algebraically subelliptic.
\item $Y$ satisfies condition $a\Ell1$.
\end{enumerate} 
\end{theorem}

The implications (a)\ $\Rightarrow$\ (b)\ $\Rightarrow$\ (c) \ $\Rightarrow$\ (d) are trivial
consequences of definitions. The implication (c)\ $\Rightarrow$\ (b) was shown by 
Gromov \cite[3.5.B, 3.5.C]{Gromov1989} (see also 
\cite[Proposition 6.4.2]{Forstneric2017E}); this is called the localization property
for subelliptic manifolds. Essentially the same proof gives the implication 
(d)\ $\Rightarrow$\ (b) as pointed out by L\'arusson and Truong in 
\cite[p.\ 205, proof of Theorem 1]{LarussonTruong2019}.
The most surprising implication (b)\ $\Rightarrow$\ (a), which was a long-standing
open problem, has been shown very recently
by Kaliman and Zaidenberg \cite[Theorem 0.1]{KalimanZaidenberg2023}.
The implication (a)\ $\Rightarrow$\ (e) follows from the obvious fact that
by pulling back a dominating algebraic spray on $Y$ by an algebraic map $f:X\to Y$
gives a dominating algebraic spray over $f$,
so condition $a\Ell1$ holds. (The analogous implication holds for holomorphic maps.) 
Conversely, since every algebraic manifold is covered by Zariski open domains which are affine 
manifolds, condition (e) implies local algebraic ellipticity of $Y$ (condition (c)). 

Despite the fact that algebraic ellipticity is equivalent to algebraic subellipticity,
we shall keep using the latter term in some of the subsequent results to indicate
that the arguments do not use this recently established equivalence.

A major source of algebraically elliptic manifolds are {\em flexible manifolds}
in the sense of Arzhantsev et al.\ 
\cite{ArzhantsevFlennerKalimanKutzschebauchZaidenberg2013DMJ},
i.e., manifolds whose tangent space at every point is spanned by locally nilpotent 
derivations, LNDs. Indeed, the composition of (algebraic) flows of finitely many 
LNDs on a flexible manifold yields a dominating algebraic spray
(see \cite[Proposition 5.6.22 (c)]{Forstneric2017E}). 
Likewise, a complex manifold which is flexible in the holomorphic sense is 
weakly subelliptic, hence Oka (see \cite[Proposition 5.6.22 (a)]{Forstneric2017E}).
For recently found examples of flexible manifolds, see 
\cite{Perepechko2013,ParkWon2016,Gizatulin2018,MichalekPerepechkoSuss2018,ProkhorovZaidenberg2018,Gaifullin2019X,Perepechko2021} and Theorem \ref{th:toric-aOka}. 

%
%   NULL QUADRIC
%
\begin{example}
For every integer $n\ge 3$ the quadric hypersurface in $\C^n$ given by
\[
	A=\bigl\{(z_1,z_2,\ldots,z_n) \in \C^n: z_1^2+z_2^2+\cdots+z_n^2=0 \bigr\},
\]
and the image $\Sigma\subset \CP^{n-1}$ of $A^*=A\setminus \{0\}$ under the natural 
projection $\pi:\C^n\setminus\{0\} \to\CP^{n-1}$, play a major role in the theory of 
minimal surfaces in real Euclidean space $\R^n$, and in the related theory 
of holomorphic null curves in $\C^n$; see \cite{AlarconForstnericLopez2021}. 
The manifold $A^*$ is flexible (see \cite[Proposition 1.15.3]{AlarconForstnericLopez2021}), 
hence Oka. Since $\pi:A^*\to \Sigma$ is a holomorphic fibre bundle with Oka fibre $\C^*$, 
the hypersurface $\Sigma\subset \CP^{n-1}$ is Oka as well by 
\cite[Theorem 5.6.5]{Forstneric2017E}. 
\end{example}

For later reference we recall the following result  \cite[Theorem 3.1]{Forstneric2006AJM}, 
which gives a relative Oka principle for algebraic maps from affine algebraic varieties to 
algebraically subelliptic manifolds. (See also \cite[Theorem 6.15.1]{Forstneric2017E}.) 
All algebraic maps are assumed to be regular (morphisms).
An inspection of the proof in \cite{Forstneric2006AJM} 
also gives the additional statement concerning the interpolation
of a given initial algebraic map $f:X\to Y$ on a subvariety of $X$.

%
%
%  THEOREM ON ALGEBRAIC APPROXIMATION
%
\begin{theorem}  \label{th:Asubelliptic}
Let $X$ be an affine algebraic variety and $Y$ be an algebraically subelliptic manifold.
Given an algebraic map $f:X\to Y$, a compact holomorphically convex set $K$ in $X$, and 
a homotopy of holomorphic maps $f_t:U\to Y$ $(t\in[0,1)$ on an open neighbourhood $U$ 
of $K$ with $f_0=f|_U$, there are algebraic maps $F:X\times \C\to Y$ satisfying 
$F(\cdotp,0)=f$ such that $F(\cdot,t)$ approximates $f_t$ as closely as desired uniformly 
on $K$ and uniformly in $t\in[0,1]$.
If in addition the homotopy $f_t$ is fixed on a closed algebraic subvariety $X'\subset X$ 
then $F$ can be chosen such that $F(x,t)=f(x)$ for all $x\in X'$ and $t\in\C$. 

In particular, a holomorphic map $X\to Y$ that is homotopic
to an algebraic map is a limit of algebraic maps uniformly on compacts in $X$.
\end{theorem}

Note that a homotopy of continuous maps $f_t:X\to Y$ connecting a pair of holomorphic maps $f_0,f_1$ can be deformed with fixed end to a homotopy of holomorphic maps since $Y$ is an Oka manifold (see Theorem \ref{th:Oka}).

%
%   aCAP
%
\begin{corollary}[Corollary 6.15.2 in \cite{Forstneric2017E}]\label{cor:aCAP}
Every algebraically subelliptic manifold $Y$ satisfies the following 
{\em algebraic convex approximation property}:

\noindent {\rm aCAP:} Every holomorphic map $K\to Y$  from a compact convex set 
$K\subset \C^n$ can be approximated uniformly  on $K$ by regular algebraic maps 
$\C^n\to Y$.
\end{corollary}

Conversely, if the conclusion of Theorem \ref{th:Asubelliptic} 
holds for an algebraic manifold $Y$, it follows easily that $Y$ is weakly algebraically 
subelliptic (cf.\  L\'arusson and Truong \cite[Theorem 1]{LarussonTruong2019}), 
and hence algebraically elliptic by Theorem \ref{th:wae}. 
Summarizing, we have the following. % corollary. 

\begin{corollary}\label{cor:aOka}
For an algebraic manifold $Y$ the following conditions are equivalent:
\begin{enumerate}[\rm (a)]
\item $Y$ is algebraically elliptic.
\item $Y$ is algebraically subelliptic.
\item $Y$ satisfies condition $a\Ell1$. 
\item $Y$ has the algebraic homotopy approximation property 
(i.e., Theorem \ref{th:Asubelliptic} holds).
\end{enumerate}
\end{corollary}

Note that Theorem \ref{th:Asubelliptic} does not provide an algebraic map in 
every homotopy class. Indeed, there are algebraically subelliptic manifolds $Y$ which 
have no algebraic representatives in some homotopy classes of maps $X\to Y$ from 
affine algebraic varieties (see \cite[Examples 6.15.7 and 6.15.8]{Forstneric2017E}). 
%One of the simplest examples is $\CP^n$ for $n\ge 3$. In this case, to ask whether every continuous map $X\to Y=\CP^n$ is homotopic to an algebraic map amounts to asking whether every topological complex line bundle on $X$ is isomorphic to an algebraic line bundle. As shown in \cite[Example 6.15.8]{Forstneric2017E}, this fails if $X=M\setminus D$ where $M$ is a smooth quartic surface in $\CP^3$ (a K3 surface) and $D$ is a smooth hyperplane section in $M$.
A much more precise result is given by Theorem \ref{th:P1isbad}. 

Let us consider a homogeneous algebraic manifold $Y$ for some linear algebraic group $G$. 
We have $Y\cong G/H$ where $H\subset G$ is the isotropy subgroup of a point $y\in Y$.
Recall that a character of $G$ is a homomorphism of algebraic groups 
$\chi:G\to \C^*=\C\setminus\{0\}$. 
%The equivalence (b)\ $\Leftrightarrow$\ (c) holds by Theorem \ref{th:wae}, but this is not needed in the proof.

%
%   Note: relevant information obtained from Frank K. on 21.1.2022.
%
%   HOMOGENEOUS MANIFOLDS OF LINEAR ALGEBRAIC GROUPS
%
\begin{proposition}\label{prop:characters}
If $G$ is a connected linear algebraic group and 
$Y=G/H$ is an algebraic $G$-homogeneous manifold,
then the following conditions are equivalent.
\begin{enumerate}[\rm (a)]
\item $G$ has no nontrivial characters $\chi:G\to \C^*$ with $\chi(H)=1$.
\item The $G$-homogeneous manifold $G/H$ is algebraically elliptic.
\item The manifold $G/H$ is algebraically subelliptic.
\end{enumerate}
\end{proposition}

\begin{proof}
(a)$\Rightarrow$(b): If the group $G$ is connected and without nontrivial characters, then every 
$G$-homo\-gen\-eous algebraic manifold $Y=G/H$ is algebraically flexible 
%that is, the tangent space to $Y$ at each point  is spanned by LNDs; see
\cite[Proposition 5.4]{ArzhantsevFlennerKalimanKutzschebauchZaidenberg2013DMJ},
and hence algebraically elliptic \cite[Proposition 5.6.22 (c)]{Forstneric2017E}.
Furthermore, if a subgroup $H$ of $G$ does not lie in the kernel 
of any character $\chi:G\to\C^*$ then the manifold $Y=G/H$ is algebraically flexible; 
see \cite[proof of Theorem 11.7]{KalimanKutzschebauch2017MA}
and \cite[Theorem 4.1]{AndristKutzschebauch2022X}.  

The implication (b)$\Rightarrow$(c) is trivial. Note that 
(c)\ $\Rightarrow$\ (b) holds by Theorem \ref{th:wae}, but this is not needed in the proof.

We prove (c)$\Rightarrow$(a) by contradiction. Assume that $G$ has a nontrivial character 
$\chi: G\to \C^*$ with $\chi(H) =1$. The regular map $\phi:Y=G/H \to \C^*$ defined by 
$\phi(gH)=\chi(g)$ for $g\in G$ is surjective. 
Therefore, there is a holomorphic map $f: \D\to Y$ from the disc %$\D\subset \C$
such that the holomorphic map $\phi\circ f : \D \to \C^*$ is nonconstant. Now, $f$ 
cannot be approximated %on smaller discs in $\D$ 
by regular maps $F:\C\to Y$ since $\phi\circ F:\C\to \C^*$ would then 
be a nonconstant regular map, a contradiction. 
By Theorem \ref{th:Asubelliptic}, $Y=G/H$ is not algebraically subelliptic.
\end{proof}

%
%  BOCHNAK AND KUZHARZ: RESERVE COPY
%
\begin{remark}
As an aside, we mention an approximation theorem, related to Theorem \ref{th:Asubelliptic}, 
which was proved by Bochnak and Kucharz \cite{BochnakKucharz2020}.
Assume that $X$ and $Y$ are algebraic manifolds and $K$ is a compact set in $X$. 
A map $f:K\to Y$ is said to be holomorphic if it is given by a holomorphic map $U\to Y$ 
from an open neighbourhood $U\subset X$ of $K$, and is said to be {\em regular} 
if it is given by a regular algebraic map $U\to Y$ from a Zariski open neighbourhood 
$U$ of $K$ in $X$. The following is \cite[Theorem 1.1]{BochnakKucharz2020}. 
% (The special case when $Y$ is a complex Grassmanian was considered earlier by Kucharz \cite{Kucharz1995}. See also \cite{BochnakKucharz2003}.)  

%
%    BOCHNAK AND KUCHARZ
%
\begin{theorem}\label{th:BochnakKucharz}
Assume that $X$ is an affine algebraic manifold, $K$ is a compact holomorphically convex
set in $X$, and $Y$ is a homogeneous algebraic manifold for some linear algebraic group. 
Then the following conditions are equivalent for a holomorphic map $f:K \to Y$. 
\begin{enumerate}[\rm (a)]
\item The map $f$ can be approximated uniformly on $K$ by regular maps from $K$ to $Y$.
\item The map $f$ is homotopic to a regular map from $K$ to $Y$.
\end{enumerate}
\end{theorem}

The following is an obvious corollary to Theorem \ref{th:BochnakKucharz}; 
see \cite[Corollaries 1.2 and 1.3]{BochnakKucharz2020}.
Note that every continuous map from a geometrically convex set is null-homotopic.

%
%  BOCHNAK AND KUCHARZ - COROLLARY
%
\begin{corollary}
For $X,\ K$, and $Y$ as in Theorem \ref{th:BochnakKucharz}, 
every null-homotopic holomorphic map
from $K$ to $Y$ can be approximated uniformly on $K$ by regular maps from $K$ to $Y$.
In particular, every holomorphic map from a compact convex set in $\C^n$ to $Y$
can be approximated by regular maps from $K$ to $Y$.
\end{corollary}

By Proposition \ref{prop:characters}, a homogeneous algebraic manifold $Y$ 
for a linear algebraic group $G$ need not be algebraically subelliptic (an example is $\C^*$), 
and in such case Theorem \ref{th:Asubelliptic} fails. As pointed out in 
\cite[Example 1.5]{BochnakKucharz2020}, 
Theorem \ref{th:BochnakKucharz} gives an optimal weaker conclusion under a weaker assumption. 
The proof of Theorem \ref{th:BochnakKucharz} in \cite{BochnakKucharz2020} closely follows 
that of Theorem \ref{th:Asubelliptic}, given in \cite{Forstneric2006AJM}, taking into account 
the issue described above.
\end{remark}

Algebraically (sub-)elliptic manifolds appear in many applications, some of which are 
mentioned in \cite{Forstneric2017E}. A further list of properties of such manifolds, and relations
with other properties such as (local) algebraic flexibility in the sense
of Arzhantsev et al.\ \cite{ArzhantsevFlennerKalimanKutzschebauchZaidenberg2013DMJ},
can be found in \cite[Remark 2]{LarussonTruong2019}. 
L\'arusson and Truong gave the following new examples in this class; previously it
was known that such manifolds are Oka (see \cite[Theorem 5.6.12]{Forstneric2017E}).

%
%	TORIC VARIETIES
%
\begin{theorem}[Theorem 3 in \cite{LarussonTruong2019}] \label{th:toric-aOka}
Every smooth nondegenerate toric variety is locally flexible  
and hence algebraically subelliptic (as well as algebraically elliptic
by Theorem \ref{th:wae}).
\end{theorem}

Kusakabe proved in \cite[Theorem 1.2]{Kusakabe2021JGA}
the jet transversality theorem for regular algebraic maps from affine algebraic manifolds to a 
certain subclass of algebraically subelliptic manifolds. 
A local version of the transversality theorem for algebraic 
maps to all algebraically subelliptic manifolds was proved in 2006 
(see \cite[Theorem 4.3]{Forstneric2006AJM} and \cite[Theorem 8.8.6]{Forstneric2017E});
here, {\em local} means that one can achieve the transversality condition on any compact subset 
of the source manifold.  This suffices for many applications, see \cite[Sect.\ 9.14]{Forstneric2017E}.
By using the algebraic jet transversality theorem, Kusakabe extended some of these results
to the algebraic setting. Together with the results from his 
recent preprint \cite{Kusakabe2022surjective}, Kusakabe also found new applications 
to the construction of surjective strongly dominating morphisms
$\C^N\to Y$ onto algebraically subelliptic manifold. Let us recall this story.

It was shown by Forstneri\v c in 2017 that every Oka manifold $Y$ admits a holomorphic map
$f:\C^n\to Y$ with $n={\dim Y}$ such that $f(\C^n\setminus \mathrm{br}(f))=Y$,
where $\mathrm{br}(f)$ is the branch locus of $f$
\cite[Theorem 1.1]{Forstneric2017Indam}, and if $Y$ is a compact subelliptic manifold
then there is a regular algebraic map with this property 
\cite[Theorem 1.6]{Forstneric2017Indam}. He asked whether 
the latter result also holds if $Y$ is not compact. 
Arzhantsev proved \cite[Proposition 2]{Arzhantsev2022} (2022) that every very 
flexible variety is the image of an affine space by an algebraic morphism.
Kusakabe obtained the following more precise result for a wider class of manifolds
\cite[Theorem 1.2]{Kusakabe2022surjective}.

%
%  SURJECTIVE MORPHISMS FROM AFFINE SPACES
%
\begin{theorem}\label{th:surjectivemorphisms}
For every algebraically subelliptic manifold $Y$ there is a regular algebraic map
$f:\C^{\dim Y+1}\to Y$ such that $f(\C^{\dim Y+1} \setminus \mathrm{br}(f))=Y$.
\end{theorem}

It remains an open question whether every algebraically subelliptic manifold $Y$
is the image of a surjective morphism $\C^{\dim Y}\to Y$. 

An application of Theorem \ref{th:surjectivemorphisms}, 
and of \cite[Theorem 1]{Arzhantsev2022}, gives the following characterization 
of open images of morphisms between affine spaces. 

\begin{corollary}[Corollary 1.4 in \cite{Kusakabe2022surjective}]\label{cor:imagesofmorphisms} 
For a Zariski open subset $\Omega$ of $\C^n$, the following conditions are equivalent:
\begin{enumerate}
\item $\Omega$ is the image of a morphism from an affine space.
\item The complement $\C^n\setminus \Omega$ is a subvariety of codimension at least two.
\end{enumerate}
\end{corollary}

This clearly fails for entire maps $\C^n\to\C^n$ whose images may omit
a hypersurface.

\subsection{Algebraic Oka properties}
%
%   LARUSSON AND TRUONG 
%
The following algebraic analogues of basic Oka properties 
(see \cite[Sect.\ 5.15]{Forstneric2017E} for the latter) were 
studied by L\'arusson and Truong  \cite{LarussonTruong2019} in 2019.

%
%  ALGEBRAIC OKA PROPERTIES
%
\begin{definition}\label{def:aproperties} 
Let $Y$ be an algebraic manifold.
\begin{enumerate}[\rm (a)]
\item $Y$ enjoys the {\em (basic) algebraic Oka property} {\rm (aBOP)} if 
every continuous map $X\to Y$ from an affine algebraic manifold $X$ 
is homotopic to an algebraic map.
\item $Y$ enjoys the {\em algebraic approximation property} {\rm (aAP)}
if every continuous map $X\to Y$ from an affine algebraic manifold,
which is holomorphic on a neighbourhood of a compact holomorphically convex subset $K$ of $X$, 
can be approximated uniformly on $K$ by algebraic maps $X\to Y$.
\item $Y$ enjoys the {\em algebraic interpolation property} {\rm (aIP)} 
if every algebraic map $X'\to Y$ from an algebraic subvariety $X'$ of an affine algebraic manifold
$X$ has an algebraic extension $X\to Y$ provided that it has a continuous extension.
\end{enumerate}
\end{definition}

Note that properties aAP and aIP are algebraic versions of the corresponding properties 
BOPA and BOPI in the holomorphic category; however, in aAP and aIP we do not ask for 
the existence of homotopies connecting the initial map to the final map.

We have already mentioned examples of algebraic manifolds which are Oka 
but aBOP fails (see \cite[Examples 6.15.7, 6.15.8]{Forstneric2017E}). 
The following result of L\'arusson and Truong  \cite[Theorem 2]{LarussonTruong2019}
shows in particular that no compact algebraic manifold satisfies conditions aBOP, aAP, and aIP.
Hence, it is natural to look at affine algebraic manifolds in these questions. 

%
%   P1 IS BAD FOR A-OKA
%
\begin{theorem}\label{th:P1isbad}
If $Y$ is an algebraic manifold which contains a rational curve %$\CP^1$ 
or is compact, then $Y$ does not have any of the properties {\rm aBOP, aAP, aIP}.
\end{theorem}

Although the proof of the general case requires nontrivial results from algebraic geometry, the
basic idea for the case when $Y$ is a projective manifold is not difficult to explain. 
First of all, it is easily seen that each of the properties aIP and aBOP implies the existence
of a nontrivial rational curve $g:\CP^1\to Y$. Assuming now $Y$ that admits such a curve, 
we will show that $Y$ does not satisfy aIP; a similar argument excludes the other properties.
The basic case to consider is $Y=\CP^1$. 
Let $S\subset \C^2$ be an algebraic curve whose projective
closure is not rational. Then, $S$ admits an algebraic line bundle $L\to S$ all of whose nonzero tensor
powers are algebraically nontrivial, and every such bundle is the pullback of the universal bundle 
$U\to\CP^1$ by an algebraic map $f:S\to\CP^1$. Since $S$ is an open Riemann surface, 
$f$ is null-homotopic and hence it extends to a continuous map $\C^2\to\CP^1$.
If $\CP^1$ satisfies aIP then $f$ also extends to a regular map $\C^2\to\CP^1$, and hence
the line bundle $f^*U\to \C^2$ is algebraically trivial by the Quillen--Suslin theorem.
This contradicts the fact that the restriction $L= f^*U|_S \to S$ 
is algebraically nontrivial, so $\CP^1$ does not satisfy aIP. 
In the general case when $Y$ is a projective manifold and $g:\CP^1\to Y$ 
is a nontrivial rational curve, taking an ample line bundle $E\to Y$, the pullback 
$g^*E\to \CP^1$ is algebraically nontrivial, which shows as before that the map 
$g\circ f:S\to Y$ does not extend to an algebraic map $\C^2\to Y$; 
hence $Y$ does not satisfy aIP. For a general compact algebraic manifold $Y$, 
one uses finitely many blowups in order to obtain a projective manifold.

%
%   REMARK ON aOKA
%
\begin{remark}
L\'arusson and Truong proposed in \cite{LarussonTruong2019}
to call an algebraic manifold satisfying the equivalent conditions in Corollary \ref{cor:aOka} an 
{\em algebraically Oka manifold}, aOka. My reservation to this choice of term is that 
algebraically subelliptic manifolds do not abide by the philosophy that 
Oka properties refer to the 
existence of solutions of analytic or algebraic problems in the absence of topological 
obstructions. Indeed, Theorem \ref{th:P1isbad} shows that most such manifolds do not have 
absolute Oka properties such as aBOP. Furthermore, in light of Theorem \ref{th:wae} 
we now know that algebraically subelliptic manifolds coincide with algebraically
elliptic manifolds, a standard notion since Gromov's paper \cite{Gromov1989}.

This being said, we do not know a single example of an affine algebraic
manifold with nontrivial topology for which aBOP is known to hold. 
We propose the following test case.

\begin{problem}
Does $\C^2\setminus \{0\}$ enjoy aBOP?
\end{problem}

The first nontrivial homotopy group is $\pi_3(\C^2\setminus \{0\})=\Z$, a generator being
the unit sphere $S^3\subset \C^2\cong\R^4$. The linear projection $A:\C^4\to \C^2$ given by
$A(z_1,z_2,z_3,z_4)=(z_1+\imath z_2,z_3+\imath z_4)$ maps the affine quadric 
$X=\{z_1^2+z_2^2+z_3^2+z_4^2=1\}$ to $\C^2\setminus \{0\}$, and its restriction 
to the $3$-sphere $X\cap \R^4$ of real points in $X$ 
is the identity map under the standard identification $\R^4\cong \C^2$.
Note that $X\cap \R^4$ is a deformation retract of $X$, hence a generator of $\pi_3(X)=\Z$.
Thus, the algebraic map $A:X\to \C^2\setminus \{0\}$ induces an isomorphism 
$\pi_3(X)\stackrel{\cong}{\longrightarrow} \pi_3(\C^2\setminus \{0\})=\Z$, so 
the generator of $\pi_3(\C^2\setminus \{0\})$ is realized by an algebraic map. What about 
other nontrivial maps $S^n\to \C^2\setminus \{0\}$ from spheres of dimensions $n\ge 3$?

The analogous argument applies to $\C^n\setminus \{0\}$ for any $n\ge 2$: the generator
of the lowest nontrivial homotopy group $\pi_{2n-1}(\C^n\setminus \{0\})=\Z$ is represented by 
an algebraic map $X\to \C^n\setminus \{0\}$ 
from the complex $(2n-1)$-sphere $X=\{\sum_{i=1}^{2n}z_i^2=1\} \subset \C^{2n}$. 
\end{remark}

%
%   
%   BLOWUPS
%
\subsection{Oka properties of blowups}
On the theme of Oka properties of blowups of algebraic manifolds, 
we mention the following recent result of Kusakabe \cite[Corollary 4.3]{Kusakabe2020IJM}. 

%
%   KUSAKABE - BLOWUPS
%
\begin{theorem}\label{th:blowup}
Let $Y$ be an algebraic manifold and $A\subset Y$ be a closed algebraic submanifold
of codimension at least two. If $Y$ enjoys {\rm aCAP} 
(in particular, if $Y$ is algebraically subelliptic), then the blowup $\Bl_A Y$ 
also enjoys {\rm aCAP}, and hence is an Oka manifold.
\end{theorem}

Note that in Theorem \ref{th:blowup} it is not claimed that $\Bl_A Y$ is 
algebraically subelliptic even if $Y$ is such. Kusakabe proved this result by reducing it to 
\cite[Theorem 1]{LarussonTruong2017} by L\'arusson and Truong, which 
pertains to algebraic manifolds covered by Zariski open sets
equivalent to complements of codimension $\ge 2$ algebraic subvarieties in affine spaces
(see also \cite[Theorem 6.4.8]{Forstneric2017E}).
Note that Theorem \ref{th:blowup} subsumes the result of
Kaliman et al.\ \cite{KalimanKutzschebauchTruong2018}.

%Another result in this direction is due to Kaliman et al.\ \cite{KalimanKutzschebauchTruong2018}.
%
\begin{comment}
In the latter paper the authors showed that the blowup of a locally stably flexible algebraic
manifold (a notion introduced by them) at a smooth algebraic submanifold 
is subelliptic, and hence Oka. This class includes class $\Acal$ and is stable under 
removal of thin algebraic subvarieties.
\end{comment}

The following result of Kusakabe \cite[Corollary 1.5]{Kusakabe2021IUMJ} is a
consequence of \cite[Corollaries 5.6.18 and 6.4.13]{Forstneric2017E} and of the localization
theorem (see Theorem \ref{th:localization}).

%
%   BLOWUPS OF TORI AT FINITELY MANY POINTS
%
\begin{theorem} Let $Y$ be a complex manifold of dimension $n\ge 2$ 
which is Zariski locally isomorphic to $(\C^*)^n$. Then, for any finite subset $A\subset Y$, the
complement $Y\setminus A$ and the blowup $\Bl_A Y$ are Oka. This holds in particular for any
smooth toric variety $Y$.
\end{theorem}

Recent results concerning the Oka property of blowups of certain complex linear algebraic groups
along tame discrete subsets, and complements of such sets, are due to Winkelmann 
\cite{Winkelmann2021TG} (2022). We mention the following one. 

%
%   WINKELMANN: COMPLEMENTS OF TAME SETS
%   \cite[Theorem 20]{Winkelmann2021TG}.
%
\begin{theorem}[Theorem 20 in \cite{Winkelmann2021TG}] 
\label{th:WinkelmannTh20}
Let $G$ be a complex linear algebraic group, and let $D$ be a tame
discrete subset of $G$. Then $G\setminus D$ is an Oka manifold. Furthermore, there exists an
infinite discrete subset $D'$ of $G$ such that $G\setminus D'$ is not an Oka manifold.
\end{theorem}

%
%   WINKELMANN: BLOWING UP TAME SETS \cite[Proposition 16]{Winkelmann2021TG}.
%
\begin{proposition} 
[Proposition 16 in \cite{Winkelmann2021TG}]
\label{prop:WinkelmannP16} 
If $D$ is a closed tame discrete subset in a character-free complex linear algebraic group $G$, 
then the blowup $\Bl_D G$ is an Oka manifold.
\end{proposition}

%
% REMARK - WINKELMANN
%
\begin{remark}\label{rem:Winkelmann}
The content of \cite[Theorem 18]{Winkelmann2021TG} by Winkelmann (stated without a citation) is that
every subelliptic complex manifold is Oka. This is the main result of the paper 
\cite{Forstneric2002MZ} from 2002, and it appears as \cite[Corollary 5.6.14]{Forstneric2017E}.
Also, \cite[Proposition 8.3]{Winkelmann2021TG} is seen by noting that such a manifold $X$
is weakly elliptic and hence Oka by \cite[Corollary 5.6.14]{Forstneric2017E}.
\end{remark}

The above results provide a significant contribution to the following problem. 

\begin{problem}{\rm (See \cite[Problem 6.4.9]{Forstneric2017E})}.
Is the blowup of an algebraic Oka manifold along an algebraic submanifold Oka? 
\end{problem}

The  corresponding problem in the holomorphic category was answered 
negatively by Kusakabe in \cite[Example A.3]{Kusakabe2021IUMJ}: there are discrete sets 
$A\subset \C^n$ for any $n>1$ such that the blowup $\Bl_A\C^n$ is volume Brody hyperbolic, 
and hence it is not Oka. Note that such a set $A$ cannot be tame in view of 
\cite[Proposition 6.4.12]{Forstneric2017E}.

%
%	TOPOLOGICAL PROPERTIES OF ALG SUBELLIPTIC MANIFOLDS
%
\subsection{Topological properties of algebraically subelliptic manifolds}
In 1989, Gromov posed the following problem \cite[0.7.B'']{Gromov1989}.
(In Gromov's paper, what is now called an Oka manifold is called an
$\mathrm{Ell}_\infty$ manifold, but the meaning is the same.)

\begin{problem} Does there exist an Oka manifold which is
homotopy equivalent to a given finite CW complex?
\end{problem}

Although there are no obvious obstructions, there has been no 
progress on this question, except for what can be inferred 
from the known examples. Very recently, Kusakabe 
proved the following result for algebraically subelliptic manifolds;
see \cite[Theorem 1.3]{Kusakabe2022pi1}.

%
%   FINITE FUNDAMENTAL GROUPS
%
\begin{theorem}\label{th:pi1}
The fundamental group of any algebraically subelliptic manifold is finite.
Conversely, for any finite group $G$ there exists an algebraically subelliptic manifold 
$Y$ whose fundamental group $\pi_1(Y)$ is isomorphic to $G$.
\end{theorem}

Since an unramified finite covering of an algebraically subelliptic manifold
is also such a manifold (cf.\ \cite[Proposition 6.4.10]{Forstneric2017E}), 
Theorem \ref{th:pi1} implies the following corollary.

\begin{corollary}[Corollary 1.5 in \cite{Kusakabe2022pi1}] \label{cor:pi1} 
The universal cover of an algebraically subelliptic manifold is also an 
algebraically subelliptic manifold. 
\end{corollary}

Another consequence of Theorem \ref{th:pi1} and of the algebraic approximation
theorem for holomorphic maps to algebraically subelliptic manifolds
(see Theorem \ref{th:Asubelliptic}) is the following.

\begin{corollary}[Corollary 1.6 in \cite{Kusakabe2022pi1}] \label{cor:approx}
Let $Y$ be an algebraically subelliptic manifold. For any holomorphic
map $f:\C^*\to Y$ and any sufficiently large natural number $n$ the holomorphic
map $\C^* \to Y$, $z\mapsto f(z^n)$ can be approximated by algebraic morphisms 
$\C^*\to Y$.
\end{corollary}

As pointed out by Kusakabe, these results fail in general for an arbitrary algebraic Oka
manifold $Y$. For example, Theorem \ref{th:pi1} and Corollary \ref{cor:approx} fail 
for any elliptic curve; such a curve is holomorphically elliptic but is not algebraically 
(sub-)elliptic.

%%%%%%%%%%%%%
%
%   STUDER
%
%%%%%%%%%%%%%
\section{Oka pairs of sheaves and a homotopy theorem for Oka theory}\label{sec:Studer}
Luca Studer made several contributions to Oka theory in his PhD dissertation. 
One of them in \cite{Studer2021APDE} provides a gluing lemma for sections of coherent analytic sheaves.
Gluing  lemmas are of key importance in Oka theory. 
Those in the work by Gromov \cite{Gromov1989} and in my joint works with Prezelj
\cite{ForstnericPrezelj2002,ForstnericPrezelj2001}, and their generalizations in 
\cite{Forstneric2017E} (see in particular \cite[Proposition 5.8.1]{Forstneric2017E}), 
pertain to the sheaf of holomorphic sections of a holomorphic submersion and its subsheaf 
of sections vanishing to a given order on a subvariety. Studer proved 
a gluing lemma for sections of an arbitrary coherent analytic sheaf.
This gives shortcuts in the proofs of Forster and Ramspott's 
Oka principle for admissible pairs of sheaves \cite{ForsterRamspott1966IM1} 
and of the interpolation property for sections of elliptic submersions in \cite{ForstnericPrezelj2001}.
The main technical part of Studer's proof is a certain lifting theorem 
\cite[Theorem 1]{Studer2021APDE} which reduces the splitting 
problem to sections of a free sheaf. %, which is well understood. 

The second main result of Studer is a homotopy theorem based on Oka theory,
presented in \cite{Studer2020MA}. He pointed out that all proofs of
Oka principles can be divided into an analytic first part and a purely topological second part
which can be formulated very generally, thereby providing a reduction
of the proof to the key analytic difficulties. This general topological statement
is \cite[Theorem 1]{Studer2020MA}. Its assumptions list the properties one
has to show in the first part of the proof of an Oka principle, and its conclusion is an Oka
principle. This extends Gromov's homomorphism theorem from \cite{Gromov1986} so that it
applies in complex analytic settings and carries out ideas sketched in \cite{Gromov1989}
and developed in \cite{ForstnericPrezelj2002} and \cite[Chapter 6]{Forstneric2017E}. 

Studer also gave a more general result, \cite[Theorem 2]{Studer2020MA}.
Let $X$ be a paracompact Hausdorff space that has an exhaustion by finite
dimensional compact subsets, and let $\Phi\hra\Psi$ be a local weak homotopy equivalence
of sheaves of topological spaces on $X$. He showed that under suitable conditions on $\Phi$ 
and $\Psi$ the inclusion $\Phi(X)\hra \Psi(X)$ of spaces of sections is a weak homotopy equivalence.
The relevant conditions reflect what is happening when approximating and gluing
sprays of sections in \cite{Gromov1989,ForstnericPrezelj2002}. Studer's proof is essentially 
an abstraction of the proof of the Oka principle for 
subelliptic submersions in \cite{ForstnericPrezelj2002,Forstneric2017E}.
He then showed how the known examples of the Oka principle fit into this general theorem.
%Precise statements would require considerable preparations, and we refer instead to Studer's papers.

%%%%%%%%%%%%%%%%%%%%
%
%	CARLEMAN AND ARAKELIAN
%
%%%%%%%%%%%%%%%%%%%%
\section{Carleman and Arakelian theorems for manifold-valued maps}\label{sec:Carleman}
The basic Oka property with approximation (BOPA) is one of the classical Oka properties of a complex manifold $Y$ which characterizes the class of Oka manifolds (see Section \ref{sec:history}). 
It refers to the possibility of approximating any holomorphic map $f\in \Oscr(K,Y)$, 
where $K$ is a compact $\Oscr(X)$-convex set in a Stein manifold (or Stein space) $X$,  
uniformly on $K$ by entire maps $F\in \Oscr(X,Y)$ provided that $f$ extends continuously from $K$ to $X$. 
Recently, B.\ Chenoweth \cite{Chenoweth2019PAMS} proved Carleman-type approximation theorems 
in the same context. Recall that Carleman approximation (after T.\ Carleman \cite{Carleman1927}) 
refers to approximation of holomorphic functions and maps in fine Whitney topologies 
on closed unbounded sets. 

Let $X$ be a complex manifold. Given a compact set $C$ in $X$ we define
\[
	h(C):=\overline{\widehat C_{\Oscr(X)}\setminus C}.
\]

%
%   BOUNDED EXHAUSTION HULLS
%
\begin{definition}\label{def:BEHn}
Let $X$ be a Stein manifold and $E$ be a closed subset of $X$.
\begin{enumerate}[\rm (a)] 
\item $E$ is $\Oscr(X)$-convex if it is exhausted by compact $\Oscr(X)$-convex sets.
\item $E$ has {\em bounded exhaustion hulls} if for every compact % $\Oscr(X)$-convex 
set $K$ in $X$ there is a compact set $K'\subset X$ such that for every compact $L\subset E$ we have that 
$
	h(K\cup L)\subset K'.
$ 
\end{enumerate}
\end{definition}

%The following is a special case of Chenoweth's main result in \cite{Chenoweth2019PAMS} (2019).

%
%   CHENOWETH
%
\begin{theorem}[Chenoweth \cite{Chenoweth2019PAMS}]
\label{th:Chenoweth}
Let $X$ be a Stein manifold and $Y$ be an Oka manifold. If $K$ is a compact 
$\Ocal(X)$-convex set in $X$ and $E$ is a closed totally real submanifold of $X$ of class $\Cscr^r$ 
$(r\in\N)$ with bounded exhaustion hulls such that $K \cup E$ is 
$\Ocal(X)$-convex, then for every $k\in \{0,1,\ldots,r\}$ the set $K \cup E$ admits 
$\Cscr^k$-Carleman approximation of maps $f \in \Cscr^k(X,Y)$ which are holomorphic on 
neighbourhoods of $K$ by holomorphic maps $X\to Y$. 
%If $K$ is the closure of a strongly pseudoconvex domain then the same holds if $f$ is $\dibar$-flat to order $k$ on $K$.
\end{theorem}

This is proved by inductively applying the Mergelyan theorem for admissible sets
in Stein manifolds (see \cite[Theorem 3.8.1]{Forstneric2017E} or 
\cite[Theorem 34]{FornaessForstnericWold2020}) together with the basic Oka property
(BOPA) for maps from Stein manifolds to Oka manifolds; see \cite[Theorem 5.4.4]{Forstneric2017E}. 
These two methods are intertwined at every step of the induction procedure. 
The special case of Theorem \ref{th:Chenoweth} for functions (i.e., for $Y=\C$) is due to 
Manne, Wold, and {\O}vrelid \cite{ManneWoldOvrelid2011}, and the necessity of 
the bounded exhaustion hulls condition was shown by Magnusson and Wold 
\cite{MagnussonWold2016}.

Given a closed unbounded set $E$ in a Stein manifold $X$, one can ask when is it possible to 
uniformly approximate every continuous function on $E$ which is holomorphic on the interior of $E$ 
by functions holomorphic on $X$. This type of approximation is named after 
Norair U.\ Arakelian  \cite{Arakelian1964} who proved that for a closed subset $E$ of a planar
domain $X\subset \C$, uniform approximation on $E$ is possible if and only if 
$E$ is holomorphically convex in $X$ %(which in this case means that $E$ has no holes in $X$)
and its complement $\wh X\setminus E$ in the one-point compactification $\wh X=X\cup\{\infty\}$
is locally connected at $\infty$. For a closed set $E$ in an open Riemann surface $X$ the latter property
is equivalent to $E$ having bounded exhaustion hulls.
A set $E$ with these two properties is called an {\em Arakelian set}.  
(See also \cite[Theorem 10]{FornaessForstnericWold2020} and the related discussion.)
The following result from \cite{Forstneric2019MMJ} is an extension of Arakelian's theorem to manifold-valued maps.

% seems to be the first known extension of Arakelian's theorem to manifold-valued maps.

%
%   ARAKELIAN FOR MAPS TO COMPACT HOMOGENEOUS MANIFOLDS
%
%
\begin{theorem} %[Forstneri{\v c} \cite{Forstneric2019MMJ}] 
\label{th:Arakelian1}
If $E$ is an Arakelian set in a domain $X\subset \C$ and $Y$ is a compact complex homogeneous manifold,
then every continuous map $X\to Y$ which is holomorphic in $\mathring E$ can be approximated 
uniformly on $E$ by holomorphic maps $X\to Y$.
\end{theorem}

Since the target manifold $Y$ is compact, the notion of uniform approximation does not depend on the specific
choice of the metric on $Y$.  The analogous result holds if $X$ is an open Riemann surface which admits bounded 
holomorphic solution operators for the $\dibar$-equation; see \cite[Theorem 5.3]{Forstneric2019MMJ}.
On plane domains one can use the classical Cauchy-Green operator. However, Arakelian's theorem for functions fails 
on some open Riemann surface as shown by examples in \cite{GauthierHengartner1982}
and \cite[p. 120]{BoivinGauthier2001}. 
Note also that Carleman approximation in the fine topology is impossible in general if the interior of $E$ 
is not relatively compact. 

The scheme of proof of Theorem \ref{th:Arakelian1} in \cite{Forstneric2019MMJ} follows the 
proof of Arakelian's theorem by Rosay and Rudin \cite{RosayRudin1989}. 
The main new analytic ingredient developed in \cite{Forstneric2019MMJ} 
is a technique for gluing sprays with uniform bounds 
on certain noncompact Cartan pairs. The proof does not apply to general Oka target manifolds, 
not even to noncompact homogeneous manifolds. 

Not much seems known concerning the Arakelian approximation on closed sets 
whose interior is not relatively compact in higher dimensional Stein manifolds. 
Recently, A.\ Lewan\-dowski \cite{Lewandowski2021CM} proved a result of this kind for functions
on a ray of balls in $\C^n$.

%%%%%%%%%%%%%%%%%%%%%%%%%%%
%
%	EUCLIDEAN DOMAINS IN COMPLEX MANIFOLDS
%
%%%%%%%%%%%%%%%%%%%%%%%%%%%
\section{The Docquier--Grauert tubular neighbourhood theorem revisited}\label{sec:DG}

Given a complex submanifold $M$ in a complex manifold $X$, let  
$
	 \nu_{M,X} = TX|_M / TM
$
denote the holomorphic normal bundle of $M$ in $X$. A theorem of Docquier and Grauert 
\cite{DocquierGrauert1960} says that if $M$ is Stein then the inclusion 
of $M$ onto the zero section of $\nu_{M,X}$ extends to a biholomorphic map from a 
neighbourhood of $M$ in $X$ onto a neighbourhood of the zero section in $\nu_{M,X}$. 
(See also \cite[Theorem 3.3.3]{Forstneric2017E}. The assumption in \cite{DocquierGrauert1960} 
that the manifold $X$ be Stein is unnecessary in view of Siu's theorem \cite{Siu1976} on the existence of 
open Stein neighbourhoods of Stein subvarieties.) This clearly 
implies that the images of holomorphic embeddings $M\hra X$, $M\hra X'$ 
with isomorphic normal bundles have biholomorphic neighbourhoods. % in $X$ and $X'$. 

We now present a generalization from \cite{Forstneric2022JMAA} (2022) 
of the Docquier--Grauert theorem to configurations of the following type.

%
%   ADMISSIBLE PAIRS
%
\begin{definition}\label{def:admissible}
A subset $S$ of a complex manifold $X$ is {\em admissible} if $S=K\cup M$ where 
\begin{enumerate}[\rm (a)]
\item $K$ is a compact set with a Stein neighbourhood $U\subset X$ such that  $K$ is $\Oscr(U)$-convex,
\item $M$ is a locally closed embedded Stein submanifold of $X$, and 
\item $K\cap M$ is a compact $\Oscr(M)$-convex subset of $M$.
\end{enumerate}
\end{definition}

It was shown in \cite[Theorem 1.2]{Forstneric2005AIF} (see also \cite[Theorem 3.2.1]{Forstneric2017E})
that an admissible set $S=K\cup M$ has a basis of open Stein neighbourhoods $V \subset X$ such that $M$ is closed 
in $V$ and $K$ is $\Oscr(V)$-convex. The case $K=\varnothing$ is Siu's theorem \cite{Siu1976}. 

%We also introduce the notion of a biholomorphism between a pair of admissible sets.

%
%   BIHOLOMORPHISM OF ADMISSIBLE SETS
%
\begin{definition}\label{def:biholo-admissible}
Assume that $X$ and $X'$ are complex manifolds of the same dimension and  
$S=K\cup M\subset X$, $S'=K'\cup M'\subset X'$ are admissible sets. 
A homeomorphism $F:S\to S'$ with $F(M)=M'$ and $F(K)=K'$ is a biholomorphism % of $S$ onto $S'$ 
if $F|_M:M\to M'$ is a biholomorphism and $F$ extends to a biholomorphism   
from a neighbourhood of $K$ onto a neighbourhood of $K'$.
\end{definition}

The conditions on $F$ clearly imply that if one of the sets $K\cup M$ and $K'\cup M'$ is admissible then so is the other one. 
The following result \cite[Theorem 1.4]{Forstneric2022JMAA} says that, under suitable conditions,
a biholomorphism $K\cup M \stackrel{F}{\longrightarrow} K'\cup M'$ 
of admissible sets can be approximated uniformly on $K$ and interpolated on the submanifold 
$M$ by ambient biholomorphisms. The Docquier--Grauert theorem 
\cite{DocquierGrauert1960} % or \cite[Theorem 3.3.3]{Forstneric2017E}) 
corresponds to the special case with $K=\varnothing$ and $K'=\varnothing$.

%
%   GENERALIZED DOCQUIER-GRAUERT THEOREM
%
\begin{theorem}\label{th:DG}
Let $S=K\cup M \subset X$ and $S'=K'\cup M'\subset X'$ be admissible sets and 
$F:S\to S'$ be a biholomorphism (see Definition \ref{def:biholo-admissible}).
Assume that there is a topological isomorphism $\Theta: \nu_{M,X}\to \nu_{M',X'}$ 
of the normal bundles over $F:M\to M'$ which is given over a neighbourhood of $K\cap M$ 
by the differential of $F$. Given $\epsilon>0$ there are an open Stein neighbourhood $\Omega \subset X$ of $S$ 
and a biholomorphic map $\Phi:\Omega \stackrel{\cong}{\longrightarrow} \Phi(\Omega)\subset X'$ such that 
\[
	\Phi|_M=F|_M
	\quad \text{and}\quad
	\sup_{x\in K} \dist_{X'}\left(\Phi(x),F(x)\right) <\epsilon.
\]
\end{theorem}

The hypothesis in the theorem is illustrated by the following diagram:
\[
	\xymatrix{   \nu_{M,X} \ar[d] \ar[r]^{\Theta}  & \nu_{M',X'} \ar[d]  \\
				   M          \ar[r]^{F}      & M'
	}			   
\]
Note that an isomorphism $\Theta: \nu_{M,X}\to \nu_{M',X'}$ in Theorem \ref{th:DG} exists 
if $\dim X\ge \left[\frac{3\dim M+1}{2}\right]$ and the restricted tangent bundles
$TX|_M$ and $TX'|_{M'}$ are isomorphic over the biholomorphic map 
$F:M\to M'$ (see \cite[Corollary 2.3]{Forstneric2022JMAA}).

Theorem \ref{th:DG} is proved by an inductive procedure commonly used in Oka theory.
An ambient biholomorphism $\Phi$ is obtained by stepwise extending the given map
$F$, changing it only slightly on a neighbourhood of $K$ at each step but keeping it
fixed on $M$, to injective holomorphic maps $F_i$ on neighbourhoods of $K\cup M_i$, where 
$M_1\subset M_2\subset \cdots \subset \bigcup_{i=1}^\infty M_i=M$ is an exhaustion of $M$ by compact
strongly pseudoconvex domains. The initial set $K_1$ is chosen such that $K\cap M\subset K_1\subset U$,
where $U$ is a neighbourhood of $K$ in $X$ on which $\Phi$ is defined.
Every step of the induction uses the gluing lemma with interpolation on a Cartan pair, 
combined with another procedure to take care of the occasional changes of topology of the sets $M_i$.  
In the approximation and gluing procedures we pay close attention to
the normal jet of the map  to ensure that no branch points occur on $M$.
To this end, we use the information provided by the isomorphism $\Theta$ over $F$ 
between the normal bundles of $M$ and $M'$. 

An important technical ingredient in the proof of Theorem \ref{th:DG} is a new version 
of the splitting lemma for biholomorphic maps close to the identity on a Cartan pair 
(cf.\ \cite[Theorem 4.1]{Forstneric2003AM} and \cite[Theorem 9.7.1]{Forstneric2017E}) 
with added interpolation on a complex submanifold; see \cite[Theorem 3.7]{Forstneric2022JMAA}. 
This result may be of independent interest. The original splitting lemma \cite[Theorem 4.1]{Forstneric2003AM}
was generalized to the parameteric case, with continuous dependence on both the domain and the map,
by L.\ Simon \cite{Simon2019JGA} and A.\ Lewandowski \cite{Lewandowski2021JGA}.

Theorem \ref{th:DG} along with \cite[Theorem 15]{ForstnericRitter2014}, which gives 
proper holomorphic embeddings $M\hra\C^n$ with geometric control of the image where $M$ is a Stein
manifold and $n\ge 2\dim M+1$, gives the following result on the existence of Euclidean neighbourhoods 
of certain admissible sets in complex manifolds (see \cite[Theorem 1.1]{Forstneric2022JMAA}). 

%
%   MAIN THEOREM
%
\begin{theorem}\label{th:Euclidean}
Assume that $S=K\cup M$ is an admissible set in a complex manifold $X$ such that $n=\dim X \ge 2\dim M+1$ 
and $TX|_M$ is a trivial bundle. Let $\Omega_0\subset X$ be an open neighbourhood of $K$ and 
$\Phi_0:\Omega_0\stackrel{\cong}{\longrightarrow} \Phi_0(\Omega_0) \subset\C^n$ % with $n=\dim X$ 
be a biholomorphic map such that $\Phi_0(K)$ is polynomially convex in $\C^n$. 
Given $\epsilon>0$ there exist a Stein neighbourhood $\Omega \subset X$ 
of $S$ and a biholomorphic map $\Phi:\Omega \stackrel{\cong}{\longrightarrow} \Phi(\Omega) \subset \C^n$ 
such that $\Phi(M)$ is a closed complex submanifold of $\C^n$ and $\sup_{x\in K}|\Phi(x)-\Phi_0(x)| <\epsilon$.

If $\dim X=2\dim M$ then $\Phi$ can in addition be chosen an immersion which is proper on $M$
and satisfies $\Phi(\Omega\setminus K)\subset \C^n\setminus \Phi(K)$.

If $S$ is closed in $X$ and $K$ is $\Oscr(X)$-convex, there is a holomorphic map $\Phi:X\to \C^n$
which satisfies the above conditions on a neighbourhood $\Omega$ of $S$ and is univalent over $\Phi(K)$:
$
	\Phi(X\setminus K) \subset \C^n\setminus \Phi(K).
$
\end{theorem}

For the last statement see \cite[Theorem 5.3]{Forstneric2022JMAA}. 
An analogous result holds if we replace $\C^n$ by an 
arbitrary Stein manifold with the density property (see  \cite[Theorem 5.2]{Forstneric2022JMAA}).

%%%%%%%%%%%%%%%%%%%%%%%%%%%
%
%	DEGENERATION OF C^n IN STEIN FIBRATIONS
%
%%%%%%%%%%%%%%%%%%%%%%%%%%%
\section{Degeneration of $\C^n$ in Stein fibrations}\label{sec:degeneration}
It is known that, in a holomorphic family of complex manifolds, the set of Oka manifolds is not closed 
in general. In particular, compact complex surfaces that are Oka can degenerate to a non-Oka surface;  
see \cite[Corollary 5]{ForstnericLarusson2014IMRN} or \cite[Corollary 7.3.3]{Forstneric2017E}. 

Since Euclidean spaces are the most basic examples of Oka manifolds, 
it is of interest to understand whether they can degenerate to a non-Oka manifold in a Stein fibration. 
A related question is whether a Stein fibration with fibres $\C^n$ is necessarily locally trivial.
For $n=1$, the answer is negative for the first question and positive for the second one.
Indeed, it was proved by Toshio Nishino \cite{Nishino1969} in 1969 
that if $X$ is a Stein manifold of dimension $m+1$ and $\pi:X\to \D^m$ is 
a holomorphic submersion onto a polydisc such that every fibre $X_z=\pi^{-1}(z)$ 
$(z\in\D^m)$ is biholomorphic to $\C$, then $X$ is fibrewise biholomorphic to $\D^m\times \C$. 
This was extended by H.\ Yamaguchi \cite{Yamaguchi1976} (1976) to the case when the fibre is a 
connected Riemann surface different from $\D$ and $\D^*=\D\setminus \{0\}$. 
(Note that if the fibres of a holomorphic submersion are compact and biholomorphic to each other,
then the submersion is a fibre bundle according to Fischer and Grauert \cite{FischerGrauert1965}.)
It follows from their results % of Nishino and Yamaguchi 
that if $X$ is Stein and $\pi:X\to \D^m$ is a holomorphic submersion
such that every fibre $X_z=\pi^{-1}(z)$ for $z\in \D^m\setminus \{0\}$ is biholomorphic to $\C$, 
then the central fibre $X_0$ is also biholomorphic to $\C$. 

Recently, Takeo Ohsawa generalized Nishino's theorem to the case when $X$ is a complete 
K\"ahler manifold and the fibres of the submersion $X\to\D^m$ equal $\C$ \cite[Theorem 0.1]{Ohsawa2020MRL}, 
or when the fibres are $\CP^n\setminus\{\mathrm{point}\}$ and $X$ is $n$-convex
\cite[Theorem 0.2]{Ohsawa2020MRL}. (Note that a $1$-convex manifold is a Stein manifold.) 
Ohsawa asked the following question \cite[Q3]{Ohsawa2020MRL}; 
I wish to thank Yuta Kusakabe for having brought this to my attention. 

\smallskip
{\em Let $X$ be a complete K\"ahler manifold and $\pi:X\to \D$ be a holomorphic submersion onto the disc 
such that the fibre $X_t=\pi^{-1}(t)$ is biholomorphic to $\C^n$ $(n>1)$ for every $t\in \D^*=\D\setminus\{0\}$. 
Does it follow that $X_0=\pi^{-1}(0)$ is also biholomorphic to $\C^n$?}
\smallskip

We give a counterexample to Ohsawa's question with $X$ a Stein manifold. 
(Every Stein manifold embeds properly holomorphically into a Euclidean space, 
so it is complete K\"ahler.) We state the result for $n=2$, but the same proof gives examples for any $n\ge 2$.
% For example, one can multiply our $X$ by $\C^{n-2}$ and extend the projection $X\to\C$ in a trivial way.

%
%  DEGENERATING FIBRE
%
\begin{theorem}\label{th:limitfibre1}
For every $k\in\N$ there is a Stein threefold $X$ 
and a holomorphic submersion $\pi:X\to\C$ % such that $\pi:\pi^{-1}(\C^*) \to \C^*$ 
which is a trivial holomorphic fibre bundle with fibre $\C^2$ over $\C^*$, while the limit fibre $X_0=\pi^{-1}(0)$ 
is biholomorphic to the disjoint union of $k$ copies of $\D\times \C$.
\end{theorem}

\begin{proof}
It was shown by J.\ Globevnik \cite[Theorem 1.1]{Globevnik1998} that there is a 
Fatou--Bieberbach domain $\Omega\subset\C^2$ (a proper subdomain of $\C^2$
which is biholomorphic to $\C^2$) whose closure $\overline \Omega$ intersects the complex line
$\C\times \{0\}$ in a closed disc $\overline U$, which may be chosen an arbitrarily small 
perturbation of the unit disc. 
(It was later shown by Wold in \cite{Wold2012} that $\Omega$ may be chosen such that 
the unit disc is a connected component of $\Omega\cap (\C\times \{0\})$, but
the intersection may contain other connected components.) 
Let $\Phi:\C\times\C^2\to \C\times \C^2$ be 
the map given by $\Phi(t,z)=(t,\phi_t(z))$, where $\phi_t(z_1,z_2)=(z_1,tz_2)$ for $t\in\C$. Set 
\begin{equation}\label{eq:X}
	X=\Phi^{-1}(\C\times \Omega) =\{(t,z_1,z_2)\in\C^3: (z_1,tz_2)\in \Omega\}.  
\end{equation}
Note that $X$ is Stein since it is the preimage of the Stein domain $\C\times\Omega\subset\C^3$ by 
the holomorphic map $\Phi$. Observe that $\phi_t(z_1,z_2)=(z_1,tz_2)$ is an automorphism of $\C^2$ 
if $t\ne 0$, and $\phi_0(z_1,z_2)=(z_1,0)$. Let $\pi:X\to\C$ denote the projection $\pi(t,z)=t$ and set 
\begin{equation}\label{eq:Xt}
	X_t=\pi^{-1}(t)=\left\{(z_1,z_2)\in\C^2: \phi_t(z_1,z_2)=(z_1,tz_2) \in\Omega\right\},\quad t\in\C.
\end{equation}
For $t\ne 0$ the domain $X_t=\phi_{t}^{-1}(\Omega)$ is biholomorphic to $\C^2$ and 
$\pi:X\setminus X_0\to\C^*$ is a trivial holomorphic fibre bundle with fibre $\C^2$. 
Indeed, $\Phi:X\setminus X_0\to \C^*\times \Omega\cong \C^*\times\C^2$
is a biholomorphism. On the other hand, the fibre $X_0=U\times \C$ is biholomorphic to $\D\times \C$. 

A minor modification of this example yields a limit fibre $X_0$ which is a disjoint union 
of any given finite number of copies of $\D\times\C$. One applies the same construction to 
a Fatou-Bieberbach domain $\Omega\subset \C^2$ whose closure intersects the line 
$\C\times \{0\}$ in $\bigcup_{i=1}^k \overline D_i$, where 
$\overline D_1,\ldots, \overline D_k$ are pairwise disjoint 
closed discs with $\Cscr^1$ boundaries. The existence of such $\Omega$ 
follows from \cite[Corollary 1.1]{Globevnik1998} of Globevnik.
\end{proof}

\begin{remark}
In Theorem \ref{th:limitfibre1} we can replace the base $\C$ by an arbitrary open 
Riemann surface $M$ and find a surjective holomorphic submersion $X\to M$ from a Stein threefold $X$
with generic fibre $\C^2$ which degenerates over each point in a given closed discrete subset $P$
of $M$. Indeed, it suffices to choose $\phi_t(z_1,z_2)=(z_1,h(t)z_2)$ in \eqref{eq:Xt},
where $h$ is a holomorphic function on $M$ whose zero locus equals $P$. 
Precomposing $\phi_t$ by a family of automorphisms $\psi_t:\C^2\to\C^2$
depending holomorphically on $t\in M$ one can also obtain limiting fibres over the points in 
$P$ with different number of connected components.  
\end{remark}

In the proof of Theorem \ref{th:limitfibre1} we used dilatations in a coordinate direction
of a suitably chosen Fatou--Bieberbach domain $\Omega\subset\C^2$ which intersects the complex line
$z_2=0$ but does not contain it. One may ask whether a similar phenomenon can be achieved
by using dilatations $z\mapsto tz$ for $t\in\C^*$ and $z\in\C^2$. The following result 
shows that this is not the case.

\begin{proposition}\label{prop:rescaling}
If $\Omega\subset  \C^n$ is a Fatou-Bieberbach domain containing the origin, then the domain
$
	X=\{(t,z)\in \D\times \C^n:tz\in \Omega\}
$
is Stein and the projection $\pi:X\to \D$ given by $\pi(t,z)=t$ is a trivial fibre bundle with fibre $\C^n$. 
\end{proposition}

\begin{proof}
Pick a biholomorphism  $g:\C^n\to\Omega$ with $g(0)=0$.
Let $g(z)=Az+O(|z|^2)$ near $z=0$. Replacing $g$ by $g\circ A^{-1}$ we may assume that 
$g(z)=z+O(|z|^2)$. For $t\in\C^*$ let $\theta_t\in\Aut(\Omega)$ be 
obtained by conjugating the map $z\mapsto tz$ by $g$:
\begin{equation}\label{eq:theta}
	\theta_t(z)=g(tg^{-1}(z)),\quad z\in\Omega. 
\end{equation}
Note that $\theta_{st}=\theta_s\circ \theta_t$, and $\theta_t$ is globally attracting to the origin if $|t|<1$. 
The map $t^{-1}\theta_t$ is a biholomorphism of $\Omega$ onto the fibre $X_t=t^{-1}\Omega$ of $X$ 
over $t\in\C^*$. We claim that
\[
	\lim_{t\to 0} t^{-1}\theta_t = g^{-1}:\Omega\to\C^n
\]
holds uniformly on compacts in $\Omega$, so we get a holomorphic trivialization 
$\D\times\Omega \stackrel{\cong}{\longrightarrow} X$.
Indeed, near $z=0$ we have that $g(z)=z+O(|z|^2)$ and hence 
\[
	t^{-1}\theta_t(z) = t^{-1} g(tg^{-1}(z)) = g^{-1}(z)+ O(|t| \cdotp |g^{-1}(z)|^2).
\]
This shows that $t^{-1}\theta_t$ converges to $g^{-1}$ uniformly on a ball $0\in B\subset\Omega$
as $t\to 0$. Globally on $\Omega$ the same holds since for any compact set $K\subset \Omega$ 
we can choose $s\in\C^*$ close to $0$ such that $\theta_s(K) \subset B$. 
Fix such $s$. For any $z\in K$ we then have that
\[
	t^{-1}\theta_t(z) = s^{-1}(t/s)^{-1}\theta_{t/s}(\theta_s(z)) 
	\,\stackrel{t\to 0}{\longrightarrow}\,
	s^{-1}g^{-1}(\theta_s(z)) = g^{-1}(z),
\]
where the last equality holds by \eqref{eq:theta}. 
\end{proof}

% We conclude by mentioning several open questions on this subject. 

\begin{problem}
Let $\pi:X\to\D$ be a Stein submersion which is a holomorphic fibre bundle with fibre
$\C^n$ $(n>1)$ over $\D^*$. What are the possible limit fibres $X_0 = \pi^{-1}(0)$?
\end{problem}

It seems that Nishino's problem for fibres $\C^n$, $n>1$, is still open:

\begin{problem}
Assume that $X$ is a Stein manifold and $\pi:X\to\D$ is a holomorphic submersion such that 
every fibre $X_t=\pi^{-1}(t)$ is isomorphic to $\C^n$ for some $n>1$. 
\begin{enumerate}[\rm (a)] 
\item Is $\pi:X\to\D$ necessarily locally trivial?
\item Assuming that $\pi:X\to\D$ is locally trivial over $\D^*$, is it also locally trivial at $0\in\D$?
% the punctured disc $\D^*$, is it also locally trivial at the origin $0\in\D$?
%\item What is the answer to these questions if $X$ and the submersion $X\to\C$ are algebraic?
\end{enumerate}
\end{problem}

As a specific example related to the proof of Theorem \ref{th:limitfibre1}, 
we ask the following question.

\begin{problem}\label{prob:FB}
Let $\Omega$ be a Fatou--Bieberbach domain in $\C^2$ containing the line
$\C\times \{0\}$. Define $X\subset \C^3$ by \eqref{eq:X}. Clearly, $X$ is Stein,
all fibres of the projection $\pi:X\to\C$, $\pi(t,z)=t$, are biholomorphic to $\C^2$, and 
$\pi^{-1}(\C^*)\to\C^*$ is a trivial bundle. Is $\pi$ locally trivial over $0$?
\end{problem}

%
%
%	METRIC PROPERTIES
%
%
\section{Oka manifolds, Campana specialness, and metric properties}\label{sec:future}
In this section we describe some open problems regarding the relationship between Oka
manifolds, specialness in the sense of Campana, and curvature properties of K\"ahler metrics.

\smallskip
\noindent{\bf Campana special manifolds and Oka manifolds.}  
Special manifolds play an important role in Campana's structure theory of compact K\"ahler manifolds, 
developed in \cite{Campana2004AIF,Campana2004AIF-2}. The definition of specialness, which is 
a type of holomorphic flexibility property, is quite technical; we refer to the cited papers or to 
\cite[Definition 2.1]{CampanaWinkelmann2015}. A connected compact
K\"ahler manifold is special if and only if it does not admit any dominant rational map onto an orbifold of
general type \cite{Campana2004AIF,Campana2011}. If $Y$ is special 
then no unramified cover of $Y$ admits a dominant meromorphic map onto a 
positive dimensional manifold of general type. Compact K\"ahler manifolds which are rationally
connected or have Kodaira dimension zero are special \cite{Campana2004AIF}.
By Kobayashi and Ochiai \cite{KobayashiOchiai1975}, 
the existence of a dominant holomorphic map $\C^n\to Y$ 
to a connected compact complex manifold $Y$ 
implies that $Y$ is not of general type. By an extension of their argument, 
Campana proved that such a manifold is special \cite[Corollary 8.11]{Campana2004AIF}.
In particular, every compact Oka manifold is special in view of 
\cite[Theorem 1.1]{Forstneric2017Indam}.

Let us recall the following notion which was already considered by Gromov \cite{Gromov1989}.

%
%   BOP
%
\begin{definition}
A complex manifold $Y$ satisfies the basic Oka principle (BOP) if every continuous map $X\to Y$ 
from a Stein manifold $X$ is homotopic to a holomorphic map. 
\end{definition}

Note that a complex manifold satisfying BOP need not be an Oka manifold.
In particular, every topologically contractible manifold satisfies BOP since
every map is homotopic to a constant map, but many such manifolds (e.g., bounded convex domains in $\C^n$)
are not Oka.  This trivial obstruction does not arise in the class of compact projective manifolds.
The following result is due to Campana and Winkelmann \cite{CampanaWinkelmann2015} (2015). 

%
%   Campana and Winkelmann, 2005
%
\begin{theorem}\label{th:KW2015}
If $Y$ is a compact projective manifold satisfying BOP, then $Y$ is special and every 
holomorphic map $Y\to Z$ to a Brody hyperbolic K\"ahler manifold $Z$ is constant. 
\end{theorem}

Campana and Winkelmann conjectured that their result holds for every 
compact K\"ahler manifold $Y$. It is not known whether the converse to Theorem \ref{th:KW2015} holds:

%
%   PROBLEM 4: SPECIALNESS AND OKA
%
\begin{problem}
\label{prob:special}
Does every special compact projective manifold enjoy BOP? Is it Oka? %every such manifold Oka?
\end{problem}

As pointed out by Campana and Winkelmann in \cite{CampanaWinkelmann2015},
the answer is negative for some quasiprojective special manifolds.
In light of Campana's results in \cite{Campana2004AIF,Campana2004AIF-2}, an affirmative answer 
to Problem \ref{prob:special} would imply that every projective manifold which is dominable, rationally connected, or 
has Kodaira dimension zero is an Oka manifold. 
Campana also conjectured that a complex manifold $Y$ is special if and only if it is $\C$-connected 
if and only if its Kobayashi pseudometric vanishes identically. These questions seem to remain open. 

%
%   METRIC POSITIVITY AND OKA MANIFOLDS
%
\smallskip
\noindent{\bf Metric positivity and Oka manifolds.}  
The definitions of Kobayashi hyperbolicity and of the Oka property only depend on  
the complex structure of the underlying complex manifold, and they do not involve any auxiliary 
structures such as hermitian or K\"ahler metrics. 
Nevertheless, it has been known since 1938, when Ahlfors \cite{Ahlfors1938} 
proved his generalization the Schwarz--Pick lemma, 
that hyperbolicity has a tight relationship with metric negativity. 
%
%
%
\begin{comment}
The Ahlfors lemma was generalized by S.-T.\ Yau \cite{Yau1978AJM} in 1978, 
leading to the theorem that every compact Hermitian manifold with 
negative holomorphic sectional curvature is hyperbolic \cite[Theorem 3.7.1]{Kobayashi1998}. 
%
\end{comment}
%
The observation that a compact hermitian manifold with negative holomorphic sectional
curvature is Kobayashi hyperbolic is due to Grauert and Reckziegel 
\cite{GrauertReckziegel1965}; see also Wu \cite[p.\ 217]{Wu1967},
Kobayashi \cite[p.\ 61]{Kobayashi1970}, and Greene and Wu \cite[p.\ 85]{GreeneWu1979}. 
More generally, it was proved by Greene and Wu \cite{GreeneWu1979} in 1979 
that a not necessarily complete hermitian manifold whose 
holomorphic sectional curvature is bounded above by $-c/(1 + r^2)$, 
where $c > 0$ and $r = \dist(p,\cdotp)$ is the distance 
from a fixed point $p$ in the manifold, is Kobayashi hyperbolic. 
Moreover, if a hermitian metric with this property is complete then the 
manifold is complete Kobayashi hyperbolic, and this result is close to sharp 
(see \cite[p.\ 85]{GreeneWu1979} or \cite{Seshadri2006}). It follows that no
such manifold is Oka. 

Weaker rigidity properties than Kobayashi hyperbolicity, 
such as volume hyperbolicity, are also obstructions to a manifold being Oka. 
In particular, a manifold which satisfies some form of the Schwarz lemma 
for holomorphic maps from higher dimensional balls is not Oka. 
Results in this direction were obtained by many authors; 
see in particular Chern \cite{Chern1968}, Lu \cite{Lu1968},
Royden \cite{Royden1980}, and S.-T.\ Yau \cite{Yau1975,Yau1978AJM}, among others. 
A more complete discussion of the history of the 
Schwarz lemma and its relationship to negativity of hermitian metrics 
can be found in the recent survey by Broder \cite{Broder2022}. 
See also the article by Osserman \cite{Osserman1999}
relating the Ahlfors--Schwarz--Pick lemma to comparison principles in differential geometry.

In a related direction, Kobayashi and Ochiai \cite{KobayashiOchiai1975} proved 
in 1975 that a compact complex manifold of general Kodaira type is not dominable by 
Euclidean spaces, hence it is not Oka. 
More recently, Wu and Yau \cite{WuYau2016IM,WuYau2016CAG} (2016) and 
Diverio and Trapani \cite{DiverioTrapani2019} (2019) proved that a compact connected complex manifold $Y$, which admits a K\"ahler metric whose holomorphic sectional curvature is 
everywhere nonpositive and is strictly negative at some point, has positive canonical bundle $K_Y$.
(See also Tosatti and Yang \cite{TosattiYang2017} and Nomura \cite{Nomura2018}.)
Hence, such a manifold $Y$ is projective of general type, and therefore it does not admit any 
dominating holomorphic map $\C^n\to Y$ by Kobayashi and Ochiai \cite{KobayashiOchiai1975}.

In light of these results, which broadly speaking suggest that negativity properties of hermitian 
metrics imply rigidity properties of holomorphic maps into the given manifold,
one may wonder whether there is a relationship between the Oka property of a complex 
manifold and positivity of complete hermitian or K\"ahler metrics on it. 
Evidence for this comes from 
%
%  FRANKEL CONJECTURE AND MOK'S THEOREM
%
the Frankel Conjecture, solved affirmatively by Mori \cite{Mori1979} (1979) and 
Siu and Yau \cite{SiuYau1980} (1980), saying that a compact K\"ahler manifold with positive 
holomorphic bisectional curvature is biholomorphic to a complex projective space, 
and hence is Oka. (Mori's theorem holds under the weaker assumption that the manifold 
has ample tangent bundle.) This was generalized by Mok in 1988 whose main result 
\cite[Main Theorem]{Mok1988} implies the following.

%
%   NONNEGATIVE BISECTIONAL CURVATURE IMPLIES OKA
%
\begin{theorem}\label{th:Mok}
Every compact K\"ahler manifold with nonnegative holomorphic bisectional curvature is an Oka manifold.
\end{theorem}

Mok's result says that for a compact K\"ahler manifold $(Y,g)$ of nonnegative holomorphic 
bisectional curvature, its metric universal cover $(\wt Y,\tilde g)$ is isometrically biholomorphic to
\[
	\wt Y = \C^k \times \CP^{n_1}\times\cdots\times\CP^{n_l} \times M_{1}\times\cdots\times M_p 
\] 
where $\C^k$ is endowed with the flat metric, each projective space in the above decomposition 
is endowed with a K\"ahler metric with nonnegative holomorphic bisectional curvature, 
and each $M_j$ is a compact hermitian symmetric space with its canonical complex structure and K\"ahler metric. 
Recall that a product of Oka manifolds is Oka \cite[Theorem 5.6.5]{Forstneric2017E}. 
Since $\C^n$ and $\CP^n$ are Oka manifolds, Theorem \ref{th:Mok} follows from the following observation. 

%
%   EVERY COMPACT HERMITIAN SYMMETRIC SPACE IS OKA
%
\begin{proposition}\label{prop:CHSS}
Every compact hermitian symmetric space is a complex homogeneous manifold, 
and hence an Oka manifold.
\end{proposition} 

\begin{proof}
Let $M$ be a compact hermitian symmetric space. The identity component of the isometry 
group of $M$ acts transitively by holomorphic automorphisms of $M$ 
(see Zheng \cite[Sect.\ 8.5]{Zheng2000}). 
By a theorem of Bochner and Montgomery \cite{BochnerMontgomery1947-2}, 
the holomorphic automorphism group $G$ of a compact 
complex manifold $M$ is a complex Lie group, and if the action is transitive 
then $M$ is a complex homogeneous manifold biholomorphic to $G/H$ where $H$ is the 
isotropy group of a point in $G$. By Grauert \cite{Grauert1958MA}, every complex homogeneous
manifold is an Oka manifold (see also \cite[Proposition 5.6.1]{Forstneric2017E}). 
\end{proof}

On the other hand, a noncompact hermitian symmetric space is biholomorphic to a bound\-ed domain
in a complex Euclidean space, so it is not Oka. The simplest example is the disc.
We refer to \cite[Sect.\ 8.5]{Zheng2000} for more information.

In contrast to Theorem \ref{th:Mok} which pertains to nonnegativity of holomorphic 
{\em bisectional} curvature, the relationship between positivity of holomorphic 
{\em sectional} curvature and the Oka property remains poorly understood. 
We pose the following problems.

%
%   PROBLEM 5: POSITIVITY AND OKA
%
\begin{problem} \label{prob:semipositive}
\begin{enumerate}[\rm (a)] 
\item Is every compact (or complete) K\"{a}hler manifold with (semi-) positive holomorphic sectional curvature an 
Oka manifold?
\item Assuming that the tangent bundle of a compact K\"{a}hler manifold $Y$ is numerically 
effective (nef), is $Y$ an Oka manifold?
\end{enumerate}
\end{problem}

If the holomorphic sectional curvature of a compact K\"{a}hler manifold is positive then, 
by Yau's conjecture solved by X.\ Yang \cite{YangX2018} in 2018, the manifold is rationally connected.
In both cases of Problem \ref{prob:semipositive} for compact K\"{a}hler manifolds, the results 
of Matsumura \cite[Theorem 1.1]{Matsumura2020} and Demailly et al.\ \cite{DemaillyPeternellSchneider1994}
imply that the manifold admits a finite \'etale cover which is the total space of a holomorphic fibre 
bundle over an Oka manifold with compact rationally connected fibre enjoying the corresponding semipositivity.
In view of Theorem \ref{th:updown} this reduces the compact case of Problem \ref{prob:semipositive} 
to rationally connected manifolds.

An affirmative answer to the Campana--Peternell conjecture \cite[Conjecture 11.1]{CampanaPeternell1991} 
would solve Problem \ref{prob:semipositive} (b) affirmatively.  
By \cite[Theorems 3.1 and 10.1]{CampanaPeternell1991} and 
%the known results on Oka manifolds (see especially 
Theorem \ref{th:updown} this holds true for projective manifolds of dimension at most three. 
A summary of possible cases can be found on \cite[p.\ 170]{CampanaPeternell1991}. 
This result is worthwhile recording.

\begin{theorem}
If $Y$ is a compact projective manifold of dimension at most three whose tangent bundle is nef, then $Y$ is an Oka manifold.
\end{theorem}

\smallskip
\noindent{\bf Calabi--Yau manifolds and Oka manifolds.} 
It would be of interest to know the position of Oka manifolds in the class of Calabi--Yau manifolds,
one of the most intensively studied classes of complex manifolds.

A Calabi--Yau manifold is sometimes defined 
as a compact K\"ahler manifold $Y$ with holomorphically trivial canonical bundle $K_Y$.
This implies that the first integral Chern class $c_1(Y)$ vanishes and the Kodaira
dimension of $Y$ equals zero. The converse is not true, the simplest examples being hyperelliptic surfaces
(finite quotients of complex 2-tori). 

A weaker definition defining a bigger class of manifolds, which is more standard one among complex differential geometers, 
is that a Calabi--Yau manifold is a compact K\"ahler manifold whose first real Chern class vanishes. 
As an example, Enriques surfaces fit into this more general definition of the Calabi--Yau class but not into
the former one. A fundamental result in the field is Yau's solution \cite{Yau1978CPAM} (1978) of the Calabi conjecture, 
which says that a compact K\"ahler manifold with vanishing first real Chern class has a 
K\"ahler metric in the same class with vanishing Ricci curvature. 
(The class of a K\"ahler metric is the cohomology class of its fundamental $(1,1)$-form.) 
Calabi \cite{Calabi1957} showed back in 1957 that such a metric, if it exists, is unique. 
Yau's result justifies the second definition of Calabi--Yau manifolds.
Besides their intrinsic interest in complex geometry, Calabi--Yau threefolds  are important in superstring theory 
as shapes that satisfy the requirements  for the six extra spatial dimensions. 
For all these reasons, it would be of interest to understand the following:

%
%   PROBLEM: WHICH CY ARE OKA?
%
\begin{problem}\label{prob:CY}
Which Calabi--Yau manifolds of dimension $n\ge 2$ are Oka? What if any are the implications
of the Oka property of a Calabi--Yau threefold to superstring theory? 
\end{problem}

The only Calabi--Yau manifolds of dimension one are tori, which are Oka.
The Ricci-flat metric on a torus is actually flat. 
Among compact K\"ahler surfaces, K3 surfaces furnish the only simply connected 
Calabi--Yau manifolds. They arise as quartic hypersurfaces in $\CP^3$
defined by homogeneous polynomials in four variables. An example is the quartic
\[
	\bigl\{[z_0:z_1:z_2:z_3] \in \CP^3: z_0^4+z_1^4+z_2^4+z_3^4=0 \bigr\}.
\]
Other examples of Calabi--Yau surfaces arise as elliptic fibrations, as quotients of abelian surfaces, 
or as complete intersections.  Enriques surfaces and hyperelliptic surfaces have first Chern class 
that vanishes as a real cohomology class (so Yau's theorem on the existence of a Ricci-flat metric applies)  
but it does not vanish as an integral cohomology class.
For the class of compact complex surfaces with vanishing Kodaira dimension 
it is known that hyperelliptic surfaces, Kodaira surfaces, and tori are Oka,
but it is unknown whether any or all K3 surfaces or Enriques surfaces are Oka
(see \cite[Section 7.3]{Forstneric2017E}).

More generally, for every $n\in\N$ the zero set in the homogeneous coordinates on $\CP^{n+1}$ 
of a nonsingular homogeneous polynomial of degree $n+2$ in $n + 2$ variables is a compact 
Calabi--Yau $n$-fold. The case $n = 1$ gives elliptic curves while $n=2$ gives K3 surfaces. 

%
%   NON-KAEHLER Calabi--Yau MANIFOLDS
%
We only discussed Calabi--Yau manifolds within the class of compact K\"ahler manifolds.
Recently there has been considerable interest in non-K\"ahler Calabi--Yau manifolds; 
see Tosatti \cite{Tosatti2015}. The first part of Problem \ref{prob:CY} is also of interest 
in this bigger class. 

%
%
%   ACKNOWLEDGMENTS
%
%
%\subsection*{Acknowledgements}

\noindent{\bf Acknowledgements.}
Research was supported by the European Union (ERC Advanced grant HPDR, 101053085) 
and grants P1-0291, J1-3005, and N1-0237 from ARRS, Republic of Slovenia. 
I wish to thank Rafael Andrist, Kyle Broder, Yuta Kusakabe, Frank Kutzschebauch, 
Finnur L\'arusson, Stefan Nemirovski, Takeo Ohsawa, Edgar Lee Stout, J\"org Winkelmann, 
and Erlend F.\ Wold for helpful discussions and remarks. 

%
%
%	APPENDIX: OKA'S CRITERION FOR HOLOMORPHIC CONVEXITY
%
%
\appendix

\section{Oka's criterion for holomorphic convexity and applications}

In this appendix we collect some consequences of the following criterion for 
holomorphic convexity of a compact set in a Stein manifold, due to  
Oka  \cite{Oka1937}, which are used in the paper.
(See also \cite[Theorem 2.1.3]{Stout2007} for $X=\C^n$ and \cite[Lemma 5.3.4 ]{Stout2007}
for the general case.) %We give a much shorter proof.

%
%   OKA'S CRITERION
%
\begin{theorem}\label{th:PC}
Let $X$ be a Stein manifold, $K$ be a compact subset of $X$, 
and $f:[0,1]\times X\to \C$ be a continuous function such that $f_t:=f(t,\cdotp):X\to \C$ is 
holomorphic for every $t\in [0,1]$, $f$ has no zeros on $[0,1]\times K$, and $f_1$ has no 
zeros on the holomorphic hull $\wh K_{\Oscr(X)}$ of $K$. Then none of the hypersurfaces 
$V_t=\{x\in X:f_t(x)=0\}$ $(t\in [0,1])$ intersect $\wh K_{\Oscr(X)}$. 
\end{theorem}

\begin{proof}
If $g$ is a holomorphic function on a neighbourhood of $\wh K_{\Oscr(X)}$ then, in view of the 
Oka--Weil theorem, we have that $\max\{|g(x)|:x\in K\}=\max\{|g(x)|:x\in \wh K_{\Oscr(X)}\}$. 
If the statement of the theorem is false, there is a biggest number $t_0\in [0,1)$ 
such that $V_{t_0}$ intersects $\wh K_{\Oscr(X)}$. As $t\in (t_0,1]$ decreases to $t_0$, 
the norm of the function $1/f_t \in \Oscr(\wh K_{\Oscr(X)})$ on $\wh K_{\Oscr(X)}$ 
increases to $+\infty$ while it remains bounded on $K$, a contradiction.
\end{proof}

The following is an obvious corollary to Theorem \ref{th:PC}. % this is how the result is often stated.

%
%	COROLLARY PC
%
\begin{corollary} \label{cor:PC}
If $X$ is a Stein manifold, $K$ is a compact set in $X$, and 
$V_t=\{f_t=0\}$ with $f_t\in \Oscr(X)$ for $t\in [0,1)$ is a continuous path of principal complex hypersurfaces which avoid $K$ and diverge to infinity in $X$ as $t\nearrow 1$, 
then $V_t  \cap  \wh K_{\Oscr(X)}=\varnothing$ for all $t\in[0,1)$.
\end{corollary}

We now apply Theorem \ref{th:PC} to hypersurfaces in projective spaces. 
Let $[z_0:z_1:\cdots:z_n]$ be homogeneous coordinates on $\CP^n$. 
Denote by $\mathscr{V}_k(\CP^n)$ the space of complex hypersurfaces of degree $k$
in $\CP^n$, possibly with positive integral multiplicities (i.e., effective chains of hypersurfaces). 
By Chow's theorem (see \cite[p.\ 74]{Chirka1989}) 
every $V\in \mathscr{V}_k(\CP^n)$ is of the form 
\[
	V=V(P) = \{[z_0:\cdots:z_n]: P(z_0,\ldots,z_n)=0\}
\]
where $P$ is a nonzero homogeneous polynomial of degree $k$ in $n+1$ variables. 
The complement $\CP^n\setminus V$ is an affine manifold, hence a Stein manifold. 
(See the argument in the proof of Theorem \ref{th:K}.)
Denote by ${\mathscr H}(k,n)\cong \C^{N+1}$ with $N+1={{n+k}\choose {k}}$ the complex 
vector space of all homogeneous complex polynomials in $n+1$ variables. The projection 
\[
	\C^{N+1}\setminus \{0\}\cong \mathscr H(k,n)\setminus \{0\} 
	\to \mathscr{V}_k(\CP^n)\cong \CP^N,
	\quad P\mapsto V(P)
\] 
is a fibre bundle with fibre $\C^*$, and hence any path in $\mathscr{V}_k(\CP^n)$
lifts to a path in $\mathscr H(k,n)\setminus \{0\}$. 
In other words, a path of degree $k$ hypersurfaces in $\CP^n$ is defined by a path 
of homogeneous polynomials of degree $k$ on $\C^{n+1}$.

%
%	COROLLARY PC1
%
\begin{corollary}\label{cor:PC1}
Let $V_t\in \mathscr{V}_k(\CP^n)$ $(t\in [0,1])$ be a path of hypersurfaces of degree $k$ 
and set $X=\CP^n\setminus V_1$. If $K$ is a compact set in $\CP^n$ such that 
$K\cap V_t =\varnothing$ for all $t\in [0,1]$, then $\wh K_{\Oscr(X)} \cap V_t=\varnothing$
for all $t\in [0,1]$.
\end{corollary}

\begin{proof}
By what was said above, we have 
$V_t=\{f_t=0\}$ for a path $\{f_t\}_{t\in [0,1]}\subset {\mathscr H}(k,n)\setminus \{0\}$.
The functions $F_t=f_t/f_1 : X \to\C$ for $t\in [0,1)$ are well-defined, holomorphic, 
continuous in $t$, and nonvanishing on $K$. As $t\to 1$, the affine hypersurfaces 
$\{F_t=0\}=V_t\setminus V_1 \subset X$ diverge to infinity in $X$, 
so the conclusion follows from Corollary \ref{cor:PC}.
\end{proof}

%
%	COROLLARY PC2
%
\begin{corollary}\label{cor:PC2}
If $B$ is a nonempty connected open set in $\mathscr{V}_k(\CP^n)$ %$(k\ge1,\ n\ge 2)$ 
and $\Omega=\Omega(B)\subset\CP^n$ is the union of all $V \in B$ 
(considered as hypersurfaces in $\CP^n$), 
then for any $V\in B$ the compact set $L=\CP^n\setminus \Omega$ is holomorphically convex 
in the Stein domain $\CP^n\setminus V$ . 
\end{corollary}

\begin{proof}
Fix $V\in B$. As $B$ is connected, given $z\in \Omega\setminus V$  
there is a path $\{V_t\}_{t\in [0,1]} \subset B$ with $z\in V_0$ and $V_1=V$. 
By Corollary \ref{cor:PC1}, $z$ does not belong to the hull of $L$ in 
$\CP^n\setminus V$. 
\end{proof}

%
%	COROLLARY PC3
%
\begin{corollary}\label{cor:PC3}
Let $V_t\in \mathscr{V}_k(\CP^n)$ $(t\in [0,1])$ be a path of hypersurfaces of degree $k$ in $\CP^n$. If
$K\subset \CP^n$ is a compact set such that $K\cap V_t =\varnothing$ for all $t\in [0,1]$ and $K$ is
holomorphically convex in $\CP^n\setminus V_1$, then $K$ is holomorphically convex in 
$\CP^n\setminus V_t$ for every $t\in [0,1]$.
\end{corollary}

\begin{proof}
Set $X_t=\CP^n\setminus V_t$ for $t\in[0,1]$. By thickening the path $\{V_t\}_{t\in [0,1]}$ into an 
open connected domain $B\subset \mathscr V_k(\CP^n)$, we obtain a domain $\Omega=\Omega(B)\subset \CP^n$ 
as in Corollary \ref{cor:PC2} such that $K\cap \Omega=\varnothing$. 
That corollary implies that the compact set $L=\CP^n\setminus \Omega$ is $\Oscr(X_t)$-convex   
for every $t\in[0,1]$. Note that $K\subset L$. Since $K$ is assumed to be $\Oscr(X_1)$-convex, 
it is also $\Oscr(L)$-convex, and hence $\Oscr(X_t)$-convex for every $t\in[0,1]$. 
\end{proof}

%%%%%%%%%%
%%%%%%%%%%
%%%%%%%%%%
%%%%%%%%%%   THE BIBLIOGRAPHY
%%%%%%%%%%
%%%%%%%%%%

%{\bibliographystyle{abbrv} \bibliography{references}} \begin{comment}

%\end{comment}

\end{document}